\documentclass{article}
\usepackage{graphicx} 
\newcommand{\matrixwidth}{.90\textwidth}
\newcommand{\OneAndAHalfSpacedXI}{%
    \renewcommand{\baselinestretch}{1.5}\normalsize%
}

\OneAndAHalfSpacedXI
\usepackage{bbm}
\usepackage{tikz}
\usepackage[normalem]{ulem}
\usetikzlibrary{arrows,positioning,shapes.geometric}
\usepackage{caption}
\usepackage{subcaption}
\usepackage{booktabs} 
\usepackage{float}
\usepackage{graphicx} 
\usepackage{geometry} 
\usepackage{indentfirst,csquotes}
\usepackage[round,authoryear]{natbib}
\usepackage{multirow} 

\usepackage{appendix}
\usepackage{amssymb,amsthm,amsmath}
\DeclareMathOperator{\Tr}{Tr}
\DeclareMathOperator{\sech}{sech}
\DeclareMathOperator{\Cov}{Cov}
\DeclareMathOperator{\diag}{diag}
\usepackage{xcolor,paralist,titlesec,fancyhdr,etoolbox}
\usepackage[
  colorlinks=true,     
  linkcolor=blue,      
  citecolor=blue,      
  urlcolor=blue        
]{hyperref}
\newtheorem{theorem}{Theorem}[section]
\newtheorem{definition}[theorem]{Definition}
\newtheorem{example}[theorem]{Example}
\newtheorem{assumption}[theorem]{Assumption}
\newtheorem{lemma}[theorem]{Lemma}
\newtheorem{proposition}[theorem]{Proposition}
\newtheorem{corollary}[theorem]{Corollary}

\theoremstyle{definition} 
\newtheorem{remark}[theorem]{Remark} 
\theoremstyle{remark} 

\title{Dynamic Factor Models with Forward-Looking Views}

\author{Anas Abdelhakmi\footnote{Institute of Operations Research and Analytics, National University of Singapore, Email: a.anas@u.nus.edu}, Andrew E.B. Lim\footnote{Department of Analytics and Operations, Department of Finance, and Institute of Operations Research and Analytics, National University of Singapore, Email: andrewlim@nus.edu.sg} }

\begin{document}

\maketitle


\begin{abstract}
Prediction models calibrated using historical data may forecast poorly if the dynamics of the present and future differ from observations in the past. For this reason, predictions can be improved if information like forward looking views about the state of the system are used to refine the forecast. We develop an approach for combining a dynamic factor model for risky asset prices calibrated on historical data, and noisy expert views of future values of the factors/covariates in the model, and study the implications for dynamic portfolio choice. By exploiting the graphical structure linking factors, asset prices, and  views, we derive closed-form expressions for the dynamics of the factor and price processes after conditioning on the views. For linear factor models, the price process becomes a time-inhomogeneous affine process with a new covariate formed from the views. We establish a novel theoretical connection between the conditional factor process and a process we call a \emph{Mean-Reverting Bridge (MrB)}, an extension of the classical Brownian bridge. We derive the investor's optimal portfolio strategy and show that views influence both the myopic mean-variance term and the intertemporal hedge. The optimal dynamic portfolio when the long-run mean of the expected return is uncertain and learned online from data is also derived. More generally, our framework offers a generalizable approach for embedding forward-looking information about covariates in a dynamic factor model.
\end{abstract}

\noindent {\bf Keywords:} Dynamic Factor Model, forward-looking views, Mean-Reverting Bridge, Kalman Smoothing, Optimal Portfolio, Stochastic Control.

\section{Introduction}



Decision makers commonly rely on quantitative prediction models to make decisions in uncertain environments. Such models, however, are typically calibrated to fit  historical data and may forecast poorly if the future market environment  differs from what has been observed in the  past. This suggests that predictions and decision making can  be improved if  forward-looking information is used to adjust the original  model calibrated on  (backwards-looking) historical data. In the case of a {\it dynamic factor model} for asset prices, one source of forward-looking information is (noisy) expert opinions about the future values of the economic (e.g., GDP, interest rates) or company-performance indicators used as covariates in the factor model. We develop an approach for incorporating such forward-looking information in a backwards-looking factor model, and study the impact  on covariate and price dynamics  and dynamic portfolio choice.

A classical approach to blending backward-looking models and forward-looking views is the Black-Litterman model \citep{Black_1991, Black_1992}. This model, however, is designed for a single-period problem and cannot handle complex information structures like views on covariates or views on returns over horizons different from the investor's.  It provides no guidance for the dynamic investor on updating forecasts as prices and covariates evolve over time. While dynamic Black-Litterman models have been proposed (\cite{Frey2012}, \cite{Sass2017}, \cite{Davis2013, Davis2020, Davis2021, Davis2022}), they typically consider views on the \emph{current} state of a hidden factor process. One exception is the work by \cite{Abdelhakmi2025_BL} who consider  forward looking views on asset returns when prices are geometric Brownian motion.


In this paper we formulate a generalization of the dynamic Black-Litterman model in which experts provide noisy views on {\it future values of the covariates}. This differs from  \cite{Abdelhakmi2025_BL} where views are expressed on returns; there are no factors in their model. The generalization  to a factor model with forward-looking views on factors is important because 
it is often easier and  more natural for an expert to predict future values of a covariate (e.g., an economic indicator like inflation, interest rates, GDP, or a company performance indicator like earnings) than asset price returns. We show that forward-looking views on the covariates result in richer price and covariate dynamics --  the coefficients of existing factors become time-varying and  new factors defined in terms of the views are introduced in the expected return  -- due to the intertemporal correlation   between asset returns and changes in the covariates. We establish a novel connection between view-conditional factor dynamics and the notion of a mean-reverting bridge,  and also consider the setting where the long-run mean of the risky asset is  uncertain but learned over time (all parameters in \cite{Abdelhakmi2025_BL} as assumed known). We then explore the implications for dynamic portfolio choice problem, quantifying explicitly how such views shape the optimal policy.


\subsection*{Summary of contributions}
The main contributions of the paper are as follows:

\begin{enumerate}


\item We formulate a Bayesian graphical model in which asset prices are driven by risk factors and experts provide noisy forecasts of the value of these factors at future times. We show that factors conditional on views evolve as a time-inhomogeneous mean-reverting process with time-dependent reversion rate and mean. There is also a new factor defined in terms of the views. To our knowledge, this is the first dynamic Black-Litterman model where experts provide forward-looking views on the covariates in a dynamic factor model.

\item We establish a link between the conditional mean-reverting process and the \emph{Mean-Reverting Bridge (MrB)}. In the one-dimensional setting, noisy views about a mean-reverting risk factor transform it into a mean-reverting bridge with terminal value and hitting time that depends on the noise structure of the view. We explicitly characterize the properties of MrB and provide sufficient conditions for its existence in the multi-dimensional case.

\item When views are provided only on risk factors, one might expect that they only affect factor (covariate) dynamics but not the structure of the expected returns of the price process. This is not the case when shocks to asset returns and the risk factors are correlated. We show that constant coefficients become time-varying and a new view-dependent covariate that depends on the view.


\item  We formulate and solve the optimal dynamic portfolio choice problem and explicitly characterize the investor's optimal portfolio strategy. We show that views influence both the myopic mean-variance term and the intertemporal hedging demand. The optimal dynamic policy can also be decomposed into a no-views policy and a views-induced adjustment term. We use this to show that the magnitude of the adjustment induced by views depends on the correlation between assets and is increasing in the precision of the views. Finally, we show that this adjustment can be obtained by solving the same control problem as the no-views case, but with different terminal conditions.

    \item We extend our results to the case where the long-run mean of the expected return is uncertain and learned online with market data. We derive dynamics of the factor process and the posterior mean of the uncertain parameter, and show how learning interacts with the views. We derive the optimal allocation policy and show that the investor hedges both the underlying risk factors and posterior updates from learning.

\item Our experiments show that the dynamic Black-Litterman investor with views  significantly outperforms an investor who rebalances her portfolio to the myopic Black-Litterman holding at each rebalancing epoch and with significantly less turnover. This occurs because views also affect how the dynamic investor hedges changes in the factors. The myopic investor does not hedge and is forced to rebalance more aggressively at every rebalancing epoch. The dynamic investor with views also outperforms the dynamic investor without views, though the degree of outperformance depends on the noise in the views.

\end{enumerate}

\subsection*{Literature review}
Expert forecasts have proven effective in improving investment outcomes and deriving more diversified portfolios in portfolio allocation decisions.  
This has been extensively explored in a single-period setting (e.g., \cite{Black_1991,Black_1992, BertsimasBL, Andrew2020}). 
 Single period-models make it difficult to include views of returns over horizons different from the investor's. It is also interesting and practically important to consider dynamic investment with expert views.
 
 To address these limitations, recent literature has focused on integrating expert forecasts into dynamic and continuous-time models by treating them as additional information streams. Most prior work adopts a Bayesian filtering approach, modeling expert views as noisy signals about the \emph{current} state or drift of an underlying factor process. For example, \cite{Frey2012}, \cite{Sass2017, Sass03042023}, and \cite{GabihKondakjiWunderlich2024} assume experts provide unbiased but noisy estimates of the instantaneous drift; the investor's belief about the true drift is then dynamically updated through stochastic filtering. Similarly, \cite{Davis2013, Davis2020, Davis2021, Davis2022} consider variants of a factor model in which  experts provide noisy views over time about the current value of the  factors (assumed to be hidden). Factors are estimated using Kalman filtering. Their model also accounts for biased views.  In contrast, the experts in our setup provide noisy views about the value of observed covariates at a given time in the future (e.g., company earnings in the next quarter). Our derivation of conditional price and factor dynamics is related to Kalman smoothing, though ``smoothed" estimates need to be updated as prices and factors are realized over time.

In contrast, \cite{Abdelhakmi2025_BL} propose a framework where expert views are noisy estimates of asset returns at a future point in time. In their model, assets follow a geometric Brownian motion with constant drift and volatility, and expert forecasts lead to a forward-looking adjustment of the asset dynamics. There are no factors in their model. Our paper extends this forward-looking model to a  multi-factor setting with views being given about the value of the factors at a future point in time. This structure is practically important because views are often expressed about economic or company-performance indicators (e.g., earnings over the next quarter). Relative to  \cite{Abdelhakmi2025_BL}, views about factors lead to richer price and covariate dynamics due to the correlation between asset returns and factor innovations over time.
   
Forward-looking views naturally lead to the emergence of {\it bridges}, stochastic processes conditional on information about their value at some point in the future. (See, for example,  \cite{BrownianBridgesPINSKY, JeanBlanc} for the Brownian bridge, which is probably the most well known example of such a process). In this paper, we introduce the \emph{Mean-Reverting Bridges (MrB)}, a generalization of this idea that replaces Brownian motion with mean-reverting processes. 
Previous work has analyzed Ornstein-Uhlenbeck (OU) bridges (e.g., \cite{OUB2008,OUB2013,OUB2017}), which condition OU processes on a known future value. We generalize this to a multidimensional setting with noisy, correlated observations of the future values of the mean-reverting process. We characterize properties of the resulting process and derive its dynamics in both the one- and multi-dimensional settings.

\subsection*{Outline}
We introduce the graphical model for market dynamics and expert views in Section \ref{sec2}. In Section \ref{sec3}, we derive the market dynamics conditioned on the views and show that it becomes an affine model, with risk-factors and views as covariates. We define the notion of mean-reverting bridge in Section \ref{sec4} and provide an intuitive interpretation of the conditional dynamics of the risk-factors in terms of this process. In Section \ref{sec5}, we formulate and solve the associated dynamic portfolio choice problem. Section \ref{sec6}, extends  the model to cases where the drift is uncertain and investors must learn it from observed market data. Section \ref{sec7} presents numerical results and simulations, and Section \ref{sec8} concludes.
To enhance the flow of the paper, proofs have been relegated to the Appendix.

\subsection*{Notation}
Throughout the paper, we assume that all random variables and stochastic processes are defined on the common probability space $(\Omega, \mathcal{F}, \mathbb{P})$. All vectors are column vectors. For a vector $S \in \mathbb{R}^N$, we denote its $i^{th}$ element by $S_i$, $i \in [N]$, and the diagonal matrix of stock prices by $D(S) = \diag(S_1, \dots, S_N) \in \mathbb{R}^{N \times N}$. For a matrix $L$, we use $L_i$ to denote its $i^{th}$ column, and $L_{ij}$ its $(i,j)^{th}$ element. We use $I_N$ for the $N$ by $N$ identity matrix, and $\top$ for the transpose operator.

\section{Model Setup}
\label{sec2}
In this section, we introduce the continuous time affine model of asset prices with forward-looking expert views. We specify the dynamics of the underlying factors and assets in Section \ref{sec21}, introduce the view model in Section \ref{sec22}, and the associated Bayesian graphical representation in Section \ref{sec23}.

\subsection{Financial Market}
\label{sec21}
Consider a financial market of $N$ risky assets $S = (S_1, \dots, S_N)^\top$ and a risk-free asset $S_0$ with interest rate $r_f > 0$. The drift of the processes is driven by a $d$ risk factors $X = (X_1, \dots, X_d)^\top$, whose dynamics are given by
\begin{equation}
\label{eq:dX}
    dX(t) = \Theta(\mu - X(t)) dt + L^X dW(t),
\end{equation}
where $W$ is an $N'$-dimensional standard Brownian motion, $\Theta \in \mathbb{R}^{d \times d}$ is the mean-reversion rate, $\mu \in \mathbb{R}^d$ is the long-term mean, and $L^X \in \mathbb{R}^{d \times N'}$ is the diffusion matrix. 

The price $S_i(t)$ of the risky asset $i \in [N]$ evolves as 
\begin{equation*}
    \frac{dS_i(t)}{S_i(t)} = \big(\alpha_i + \sum_{j=1}^{d} \beta_{i,j} X_j(t)\big) dt + \sum_{j=1}^{N'}L^S_{i,j} dW_j(t)
\end{equation*}
where $\beta_{i,j} \in \mathbb{R}$ can be interpreted as the regression coefficient of the asset $i \in [N]$ on the factor $j \in [d]$. For the rest of the paper, we express the price process as
\begin{equation}
\label{eq:PricesVec}
    dS(t) = D(S(t)) \big((\alpha + \beta X(t)) dt + L^S dW(t)\big)
\end{equation}
where $\alpha = (\alpha_1, \dots, \alpha_N)^\top$, and $\beta = (\beta_{i,j})$ with $i \in \{1, \dots,N\}$ and $ j \in \{ 1,\dots, d\}$. A key feature of our model is the interaction between asset and factor shocks, which we characterize through the covariance matrix \(\Sigma^{S,X}:= L^S (L^X)^\top\). We allow this covariance to be non-zero, as the correlation structure between asset returns and economic covariates is empirically significant (e.g., \cite{PastorStambaugh2009}). We assume a linear model for factors as this is standard in the literature. However, our model of views and the derivation of the factor dynamics conditional on views extends to broader classes of factor-price dependencies (see Section \ref{sec3}).


The parameters in \eqref{eq:dX} and \eqref{eq:PricesVec} are assumed to be known to the investor; in reality, they are calibrated from historical data. We consider an extension in Section \ref{sec6} where  $\alpha$ is also uncertain and learned online by the investor.
The mean-reversion parameters $(\Theta,\mu)$ can be estimated by fitting a first-order vector autoregression to the discretised factor series, following \citet{NIELSEN200083}; the intercepts $\alpha$ and loadings $\beta$ come from an ordinary-least-squares regression of log-returns on lagged factors.  
Finally, the diffusion matrices $L^{X}$ and $L^{S}$ are deduced from the sample quadratic variation of factors and returns \citep{He2002,walters2013}. 

  

\subsection{Expert Views}
\label{sec22}
At the start of the investment horizon, the investor receives a noisy prediction about the value of the factor $X$ at a point of time in the future, the horizon of which may differ from the terminal time of her investment problem.  This differs from much of the  literature (e.g., \cite{Davis2020,Davis2022,Frey2012}) where views are noisy observations of the current value of a latent process. \cite{Abdelhakmi2025_BL} consider forward-looking views on asset returns. The key difference in the present paper is that views are on the factors \(X\) that drive returns. 

Let $Y(t_1,t_2) \in \mathbb{R}^K$ be the vector of $K$ views given at time $t_1$, about the value of the factors $X$ at time $t_2 \geq t_1$. Conditioned on the value of the factors being $X(t_2)$, we assume that $Y(t_1,t_2)$ is normally distributed
\begin{equation*}
    Y(t_1,t_2) \,|\, X(t_2) = P X(t_2) + \epsilon(t_1,t_2) \sim \mathcal{N}\big( P X(t_2), \Omega(t_1,t_2) \big)
\end{equation*}
where $P \in \mathbb{R}^{K \times  d}$ is a linear mapping from factors to views, and $\epsilon(t_1,t_2) \sim \mathcal{N}\big(0, \Omega(t_1,t_2) \big)$. The covariance matrix $\Omega(t_1,t_2)$ models the accuracy of the views and the dependence between them. 
Notably, $\Omega(t,t) = 0$. For more details on estimating $\Omega$ from historical data, we refer to \cite{Omega_estimation}.

For simplicity, we assume that all views are given at the start of the investment horizon ($t_1 = 0$) and that the horizon of the views matches the investment horizon ($t_2 = T$).  This is expressed as
\begin{equation}
    \label{eq:Y_view}
    Y(0,T)\,|\,X(T) = P X(T) + \epsilon(0,T) \sim \mathcal{N}\big(PX(T), \Omega(0,T)\big).
\end{equation}
There is no essential difficulty allowing the view horizon to differ from $T$, though this extension comes with some notational baggage. For the remainder of the paper, we denote $\Omega(0,T)$ simply as $\Omega$.


\begin{example}
    To illustrate the idea, suppose the market consists of three risk factors, $X_{Inf}$ (Inflation rate), $X_{Int}$ (Interest rate), $X_{DY}$: (Dividend yield). The vector of risk factors is given by
    \begin{equation*}
        X(t) = \begin{bmatrix}
            X_{Inf}(t) \\ X_{Int}(t) \\ X_{DY}(t)
        \end{bmatrix}.
    \end{equation*}
At the start of the investment period $t = 0$, the investor receives expert views related to future realizations of $X$ at time $t = T$ (e.g., $T = 3$ months). Forecasts could be \textit{absolute views} about the future realizations of a factor, for example,``The dividend yield will go up by $3\%$", or \textit{relative view} about some linear combination of the factors, for example,``The inflation rate will be higher than the interest rate by $1\%$". In this case, the investor receives the view
\begin{equation*}
    y = \begin{bmatrix}
        3\% \\ 1\%
    \end{bmatrix}\in \mathbb{R}^2
\end{equation*}
of the linear mapping of factors
\begin{equation*}
    PX(T) = \begin{bmatrix}
        0 & 0 & 1\\
        1 & -1 & 0
    \end{bmatrix}
        \begin{bmatrix}
            X_{Inf}(T) \\ X_{Int}(T) \\ X_{DY}(T)
        \end{bmatrix}
    = \begin{bmatrix}
        X_{DY}(T) \\ X_{Inf}(T) - X_{Int}(T)
    \end{bmatrix}.
\end{equation*}

\end{example}

\begin{assumption}[Non-degeneracy and stationarity]\label{ass:nondeg}
We make the following assumptions:
\textnormal{(i)} $N' \ge \max(N,d)$ and the diffusion matrices $L^{X},L^{S}$ have full row rank; equivalently,  
$\Sigma^{X}:=L^{X}(L^{X})^{\top}$ and $\Sigma^{S}:=L^{S}(L^{S})^{\top}$ are positive definite;  
\textnormal{(ii)} All eigenvalues of the mean-reversion matrix $\Theta$ have strictly positive real parts;  
\textnormal{(iii)} The expert-view error covariance matrix $\Omega$ is positive definite.
\end{assumption}

\begin{remark}
(i) rules out redundant factors or assets and keeps instantaneous volatilities strictly positive,  
(ii) guarantees that the factor process $X$ is mean-reverting and moments stay finite, and  
(iii) reflects that every view carries some uncertainty and is not an exact linear combination of others.
\end{remark}

Throughout, $\mathcal F_t:=\sigma(W_s;\,s\le t)$ denotes the natural filtration generate by the Brownian motion \(W(\cdot)\) and $\mathcal{F}_t^Y := \sigma(\mathcal{F}_t \vee \sigma(Y(0,T)))$ the enlarged filtration that also contains the expert views.

\subsection{Graphical Representation}
\label{sec23}
Bayesian graphical models offer a straightforward and intuitive way to capture uncertainty, dependencies, and causal relationships among variables. In this section, we extend the model presented by \cite{Andrew2020} to a multi-period scenario where asset prices, denoted as $S$, are influenced by a set of risky factors, $X$. In our representation, random variables are depicted as nodes with circles for unobserved random variables (the vector of unseen asset prices and factors) and squares for ones that are observed (views). The edges in the graph indicate conditional dependencies between these variables.

Figure \ref{fig:figure1} shows a Bayesian network representation of the discrete time version of the problem where $t = 0, \dots, T$. At any time $\tau \in \{0, \dots, T\}$, the investor has access to past realizations of the risky factors and asset prices $\{X(0), S(0), \dots, X(\tau - 1), S(\tau - 1)\}$. Without views, the investor relies on the prior model, using \eqref{eq:dX} to predict the future evolution of the factors and \eqref{eq:PricesVec} for asset prices. With forward-looking predictions provided by \eqref{eq:Y_view}, the investor needs to update his prediction model.

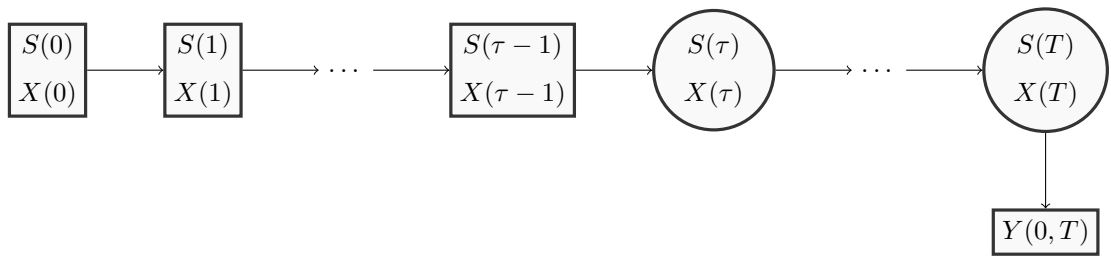
\begin{figure}[H]
\centering
\begin{tikzpicture}
[
roundnode/.style={circle, draw=black!80, fill=gray!5, very thick, minimum size=7mm, align=center},
squarednode/.style={rectangle, draw=black!80, fill=gray!5, very thick, minimum size=5mm, align=center},
]
\node[squarednode]      (observation)                              {$Y(0,T)$};
\node[roundnode]        (rT)       [above=of observation] { $S(T)$ \\ $X(T)$};
\node[]                 (dots)     [left=of rT] {\dots};

\node[roundnode]        (rtau)     [left=of dots] { $S(\tau)$ \\ $X(\tau)$};
\node[squarednode]        (rtau1)    [left=of rtau] { $S(\tau-1)$ \\ $X(\tau-1)$};
\node[]                 (dots2)    [left=of rtau1] {\dots};
\node[squarednode]        (r1)       [left=of dots2] { $S(1)$ \\ $X(1)$};
\node[squarednode]        (r0)       [left=of r1] { $S(0)$ \\ $X(0)$};

\draw[->] (rT.south) -- (observation.north);
\draw[<-] (rT.west) -- (dots.east);
\draw[<-] (dots.west) -- (rtau.east);
\draw[<-] (rtau.west) -- (rtau1.east);
\draw[<-] (rtau1.west) -- (dots2.east);
\draw[<-] (dots2.west) -- (r1.east);
\draw[<-] (r1.west) -- (r0.east);

\end{tikzpicture}
\captionsetup{format=hang, singlelinecheck=false, justification=raggedright}
\caption{Bayesian network of the model. \textmd{The figure shows a discrete time version of the problem where $t = 0, \dots, T$, and the noisy view $Y(0,T)$ of the factors $X(T)$ is revealed at $t = 0$. An investor at time $0 \leq \tau \leq T$  knows the realizations of the factors $\{X(0), \dots, X(\tau - 1)\}$, the asset prices $\{S(0), \dots, S(\tau - 1)\}$, and the forward-looking views $Y(0,T)$.}}
\label{fig:figure1}
\end{figure}

\section{Market Dynamics Conditional on Expert Views}
\label{sec3}
We now derive the dynamics of the factors $X(t)$ and asset prices $S(t)$ conditioned on the forward-looking views $Y(0,T) = y$. Let \(\mathbb Q := \mathbb P( \, \cdot \, | \, Y(0,T) = y)\) denote the conditional probability measure. We denote by $X^y(t): = X(t) \,|\, Y(0,T) = y$ and $S^y(t) := S(t) \,|\, Y(0,T) = y$ the conditional risk factors and price processes. The following proposition gives the conditional dynamics of the risk factors~$X^y(t)$.
\begin{proposition}
\label{prop:X^y}
    Suppose that the risk factors \(X(t)\) satisfy \eqref{eq:dX} and views $Y(0,T)$ satisfy \eqref{eq:Y_view}. Conditional on $Y(0,T) = y$, the vector of risk factor process $X^y(t)$ has dynamics
    \begin{equation}
    \label{eq:dX^y}
        dX^y(t) = \tilde{\Theta}(t) \big(\tilde{\mu}(t,y)- X^y(t)\big)  dt   + L^X dW^\mathbb{Q}(t)
    \end{equation}
    where 
    \begin{equation*}
        \tilde{\Theta}(t) = \Theta + L^X \eta(t) P e^{-\Theta (T-t)} \in \mathbb{R}^{d \times d}
    \end{equation*}
    is a time-dependent reversion rate, and 
    \begin{equation*}
        \tilde{\mu}(t,y) = \mu +  \tilde{\Theta}(t)^{-1} L^X \eta(t) \left(y - P \mu \right) \in \mathbb{R}^d
    \end{equation*}
    is the time-varying long-term mean. The matrix \(\eta(t) \in \mathbb{R}^{d \times K}\) is given by
    \begin{equation}
    \begin{split}
        \eta(t) &= (P e^{-\Theta (T-t)} L^X)^\top \big(P (\Sigma - e^{-\Theta(T-t)} \Sigma e^{-\Theta^\top (T-t)})P^\top + \Omega\big)^{-1}
        \label{eq:eta}
    \end{split}
    \end{equation}
    where $\Sigma \in \mathbb{R}^{d \times d}$ is the long-run covariance matrix of $X$ satisfying the Lyapunov equation
    \mbox{\(\Theta \Sigma + \Sigma \Theta^\top = \Sigma^X\)},
    and $W^\mathbb{Q}(t)$ is a standard $N'$-dimensional Brownian motion under the conditional measure \(\mathbb Q\), adapted to the filtration $\mathcal{F}_t^Y$. 
\end{proposition}
The proof is given in Appendix \ref{App:Proof_3.1}.


Equation \eqref{eq:dX^y} shows that after conditioning on the expert views, the risk factors \(X^y\) retain a mean-reverting structure. However, the dynamics are altered: the mean-reversion rate \(\tilde{\Theta}(t)\) and the long-term mean \(\tilde{\mu}(t,y)\) now depend on time and the views \(y\). Specifically, the mean is adjusted based on the ``innovation" term \(y - P\mu\); if the views \(y\) deviates from the prior long-run expectation \(P \mu\), then the factor's trajectory is adjusted accordingly.

\paragraph{Forward-looking Views and Change of Measure}
Given Figure \ref{fig:figure1}, \( Y(0,T) \) is a noisy observation of \( X(T) \). The distribution of \( X(t) \) given \( Y(0,T) = y \) is then similar to the Kalman smoother. Observing the views \( Y(0,T) = y \) modifies the investor's probability measure from \( \mathbb{P} \) to \( \mathbb{Q} = \mathbb{P}(\,\cdot \, \mid Y(0,T) = y) \), and hence modifies the investor's model of the risky factors from \eqref{eq:dX} to \eqref{eq:dX^y}. To formalize this, we can express the conditional dynamics \eqref{eq:dX^y} as an adjustment to the original drift
\begin{equation}
\label{eq:dX^y_2}
    dX^y(t) = \left(\Theta (\mu - X^y(t)) + L^X \eta(t) (y - P \mathbb{E}[X(T) \, | \, \mathcal{F}_t]) \right) dt + L^X dW^{\mathbb{Q}}(t),
\end{equation}
where 
\begin{equation*}
    \mathbb{E}[X(T) \,|\, \mathcal{F}_t] =  (I_d - e^{-\Theta(T-t)})\mu + e^{-\Theta (T-t)} X^y(t)
\end{equation*}
is the prior expectation of the terminal factors. The adjustment term \[    L^X \eta(t) (y - P \mathbb{E}[X(T) | X^y(t)])
\] is proportional to the difference between the expert view \(y\) and the investor's expectation of the view, scaled by the precision matrix \(L^X \eta(t)\).

It follows by Girsanov's Theorem (see \cite{GaussianBridges, Abdelhakmi2025_BL} for a discussion on the change of measure induced by forward-looking views in the case of a simple Brownian motion) that the relationship between the measures is defined by the Radon-Nikodym derivative
\begin{equation}
\label{eq:change_of_measure}
    \frac{d\mathbb{Q}}{d\mathbb{P}} = \exp\big(\int_0^T k(s)   dW(s) - \frac{1}{2}\int_0^T ||k(s)||^2 ds  \big),
\end{equation}
where
\begin{equation}
\label{eq:Radon-Nykodyn}
    k(t) = \eta(t) \left(y - P \mathbb{E}[X(T) | X(t)]\right), \quad t \in [0,T).
\end{equation}
We show in Appendix \ref{App:Novikov} that \(\mathbb{E}_\mathbb{P}[\frac{d\mathbb Q}{d \mathbb P}] = 1\). 
This implies that the Brownian motion $W(t)$ under the original measure can be expressed in the augmented filtration $\mathcal{F}_t^Y$ as
\begin{equation}
\label{eq:Girsanov}
\begin{split}
           dW(t) &= dW^\mathbb{Q}(t) + k(t) dt\\
           &= dW^\mathbb{Q}(t) +\eta(t) \left(y - P \mathbb{E}[X(T) \,|\, \mathcal{F}_t]\right) dt.
\end{split}
\end{equation}
The drift adjustment \(k(t)\) quantifies the impact of views on the noise dynamics. This change of measure framework directly impacts the asset price dynamics, as they are driven by the same Brownian motion \(W(t)\). The following corollary gives the conditional dynamics of the price process \(S^y(t)\).\footnote{The dynamics of $S^y(t)$ can also be derived by following the same steps as in the proof of Proposition \ref{prop:X^y}.}

\begin{corollary}
\label{corr:S^y}
    Suppose that the price process satisfies \eqref{eq:PricesVec} and views $Y(0,T)$ satisfy \eqref{eq:Y_view}. Conditional on $Y(0,T) = y$, the price process $S^y(t)$ has dynamics
    \begin{equation}
    \label{eq:dS^y}
        dS^y(t) = D(S^y(t))\left(\left(\tilde{\alpha}(t,y) + \tilde{\beta}(t) X^y(t) \right)dt + L^S dW^\mathbb{Q}(t)\right)
    \end{equation}
    where the drift \(\tilde{\alpha}(t,y) + \tilde{\beta}(t) X^y(t)\) is now an affine function of the risky factors $X^y(t)$ and the views $y$ with time-dependent coefficients
    \begin{equation}
    \label{eq:coeff_mu^S}
        \begin{cases}
            &\tilde{\alpha}(t,y) = \alpha + L^S \eta(t) \left( y -  P (I_d-e^{-\Theta (T-t)})\mu\right)  \\
            &\tilde{\beta}(t) = \beta -  L^S \eta(t) P e^{-\Theta (T-t)}
        \end{cases}
    \end{equation}
    and $\eta(t)$ is given by \eqref{eq:eta}.
\end{corollary}

From \eqref{eq:dS^y}, we see that forward-looking  views on factors influence the asset price process \(S\) in two ways. First, they modify the model's coefficients, transforming the constant intercept \(\alpha\) and coefficients \(\beta\) into time-varying functions \(\tilde{\alpha}(t,y)\) and \(\tilde{\beta}(t)\). Second, the view \(y\) defines a new covariate in the drift. The conditional price drift remains affine in the state variables \(x\) and the view \(y\).

\begin{remark} When the noise driving the risk factors \(X\) and asset prices \(S\) are uncorrelated (i.e., \(L^S (L^X)^\top = 0\)), the coefficients in \eqref{eq:coeff_mu^S} remain unchanged: \(\tilde{\alpha}(t,y) = \alpha \) and \(\tilde{\beta}(t) = \beta\). In this case, \(S^y(t)\) follows the original model structure \eqref{eq:PricesVec} and views change the factor dynamics but not the dependence structure of asset prices on those factors. In Section \ref{sec5}, we show that in this case, the investor's optimal policy is not affected by views.
\end{remark}

\begin{remark}
   The framework extends to general price dynamics. If asset prices satisfy
     \[
    dS(t)
      = b\bigl(t,S(t),X(t)\bigr)\,dt
      + \Xi\bigl(t,S(t),X(t)\bigr)\,dW(t)
  \]
    the conditional dynamics can be derived directly from \eqref{eq:Girsanov}
    \[
    dS^y(t)
      = \left(b\bigl(t,S^y(t),X^y(t)\bigr)
             + \Xi\bigl(t,S^y(t),X^y(t)\bigr)\,k(t)\right)\,dt
      + \Xi\bigl(t,S^y(t),X^y(t)\bigr)\,dW^{\mathbb Q}(t)
  \]
    where \(k(t)\) is the drift adjustment from \eqref{eq:Girsanov}.
\end{remark}

While the conditional dynamics \eqref{eq:dX^y} and \eqref{eq:dS^y} provide the complete model needed to derive an optimal allocation strategy, we now delve deeper into the structure of these conditional processes. To build intuition, the next section introduces the \emph{Mean-reverting Bridge (MrB)}, a novel stochastic process that provides insights into how forward-looking information shapes the path of mean-reverting factors. Following that, in Section \ref{sec5}, we will formulate and solve the investor's dynamic portfolio choice problem.

\section{Mean-Reverting Bridge (MrB)}
\label{sec4}

The concept of a stochastic process conditioned on a future value is well-studied for Brownian motion where it is known as a \emph{Brownian bridge} (e.g., \cite{BrownianBridgesPINSKY, JeanBlanc}). This idea was extended by \cite{Abdelhakmi2025_BL} to accommodate noisy information about the future value of Brownian motion. Here, we generalize this concept to mean-reverting processes given by~\eqref{eq:dX}.

We introduce the term \emph{Mean-Reverting Bridge (MrB)} to define a new general class of processes that extends the classical Ornstein-Uhlenbeck (OU) bridge in two critical directions: The observation of the terminal value is noisy, and the underlying mean-reverting process can be multi-dimensional. We proceed in three steps: (i)  we define and characterize the one-dimensional MrB with a noise-free future view in Section \ref{sec41}, (ii) extend this to noisy views in Section \ref{sec42}, and (iii) generalize to the multi-dimensional setting in Section \ref{sec43}. This progression provides a clear, intuitive interpretation of the conditional dynamics derived in Proposition~\ref{prop:X^y}.

\subsection{The MrB in One Dimension}
\label{sec41}


The one-dimensional version of the MrB, when conditioned on a known future value, corresponds to the classical Ornstein-Uhlenbeck (OU) bridge. This process  and its properties are well studied in the literature (e.g., \cite{OUB2008}, \cite{OUB2013}, \cite{OUB2017}). We summarize them here to establish a foundation for the extension to noisy views and multiple dimensions that follow.


Consider a standard mean-reverting process (an OU process with zero mean)
\begin{equation}
    \label{eq:OU}
    \begin{cases}
                dX(t) &= - \theta X(t) dt + dW(t)\\
                X(0) &= a,
    \end{cases}
    \end{equation}
where $\theta > 0$ is the mean-reversion rate and $W(t)$ is a standard Brownian motion. The solution is
\(X(t) = e^{- \theta t} a + \int_0^t e^{-\theta (t-s)} dW(s)\). We now condition this process on its value at a future time \(T\).

\begin{definition}[Mean-Reverting Bridge]
\label{def:MrB}
    Let $X(t) \in \mathbb{R}$ be a standard mean-reverting process satisfying \eqref{eq:OU}. Then the process $\{B(t)=(X(t)|X(T)=y)\}$ is called a mean-reverting Bridge (MrB) from $a$ to $y$ with mean-reversion rate $\theta$ and hitting time $T$.
\end{definition}

An MrB constrains a mean-reverting process to start at $a$ and terminate $y$ at time $T$, effectively “pinning” the process at both ends of the interval $[0, T]$. The next proposition summarizes the properties of this process. The proof is detailed in Appendix \ref{Appendix:Sec4}.


\begin{proposition}
\label{Prop:MrB_properties}
    A stochastic process $B(t) \in \mathbb{R}$ is a mean-reverting bridge (MrB) from $a$ to $y$ with rate $\theta > 0 $ and hitting time $T > 0$ if  and only if
   \begin{enumerate}
    \item $B(0) = a$ and $B(T) = y$ with probability 1.
    \item It is a Gaussian process.
    \item Its expectation is $\mathbb{E}[B(t)] = e^{-\theta t} a + \frac{e^{-\theta(T-t)} - e^{-\theta(T+t)}}{1 - e^{-2\theta T}}(y - e^{-\theta T} a)$.
    \item The covariance $\Cov(B(t), B(s)) = \frac{1 - e^{-2\theta(T-t)}}{1 - e^{-2\theta T}} \Cov(X(t), X(s))$ for $s \le t \le T$.
    \item Its sample paths $t \to B(t)$ are continuous on $[0,T]$ with probability 1.
\end{enumerate}
Furthermore, the MrB is the unique solution to the SDE
\begin{equation}
\label{eq:sde_1dmrb}
\begin{cases}
    dB(t) &= \tilde{\theta}(t)(\tilde{\mu}(t,y) - B(t))dt + dW^{\mathbb{Q}}(t) \\
    B(0) &= a\sout{}
\end{cases}
\end{equation}
where $W^\mathbb{Q}$ is a Brownian motion under the conditional measure $\mathbb{Q} = \mathbb{P}(\cdot | X(T)=y)$ and
\begin{equation}
\label{eq:coeff_MrB}
\tilde{\theta}(t)
\;:=\;
\frac{1+e^{-2\theta\,(T-t)}}{1-e^{-2\theta\,(T-t)}}\, \theta
\;=\;
\coth\bigl(\theta\,(T-t)\bigr)\, \theta,
\qquad
\tilde{\mu}(t,y)
\;:=\;
\frac{2\,e^{-\theta\,(T-t)}}{1+e^{-2\theta\,(T-t)}}\,y
\;=\;
\sech\bigl(\theta\,(T-t)\bigr)\, y.
\end{equation}
\end{proposition}

 The time-varying reversion rate \(\tilde{\theta}(t)\) increases as \(t \to T\), pulling the process ever more strongly towards its time-varying mean \(\tilde{\mu}(t,y)\). Simultaneously, \(\tilde{\mu}(t,y)\) converges to the target value \(y\). This dual effect ensures the bridge hits its target \(y\) at time \(T\) precisely, as illustrated in Figure \ref{fig:Standard_MrB}. When \(\theta \to 0\), the dynamics reduce to the standard Brownian bridge SDE. The mean-reverting bridge with noise-free view described in Proposition \ref{Prop:MrB_properties} is a special case of the general conditional factor dynamics in Proposition~\ref{prop:X^y}.


 \begin{figure}[ht]
 \begin{center}
       \includegraphics[width= 0.8 \textwidth]{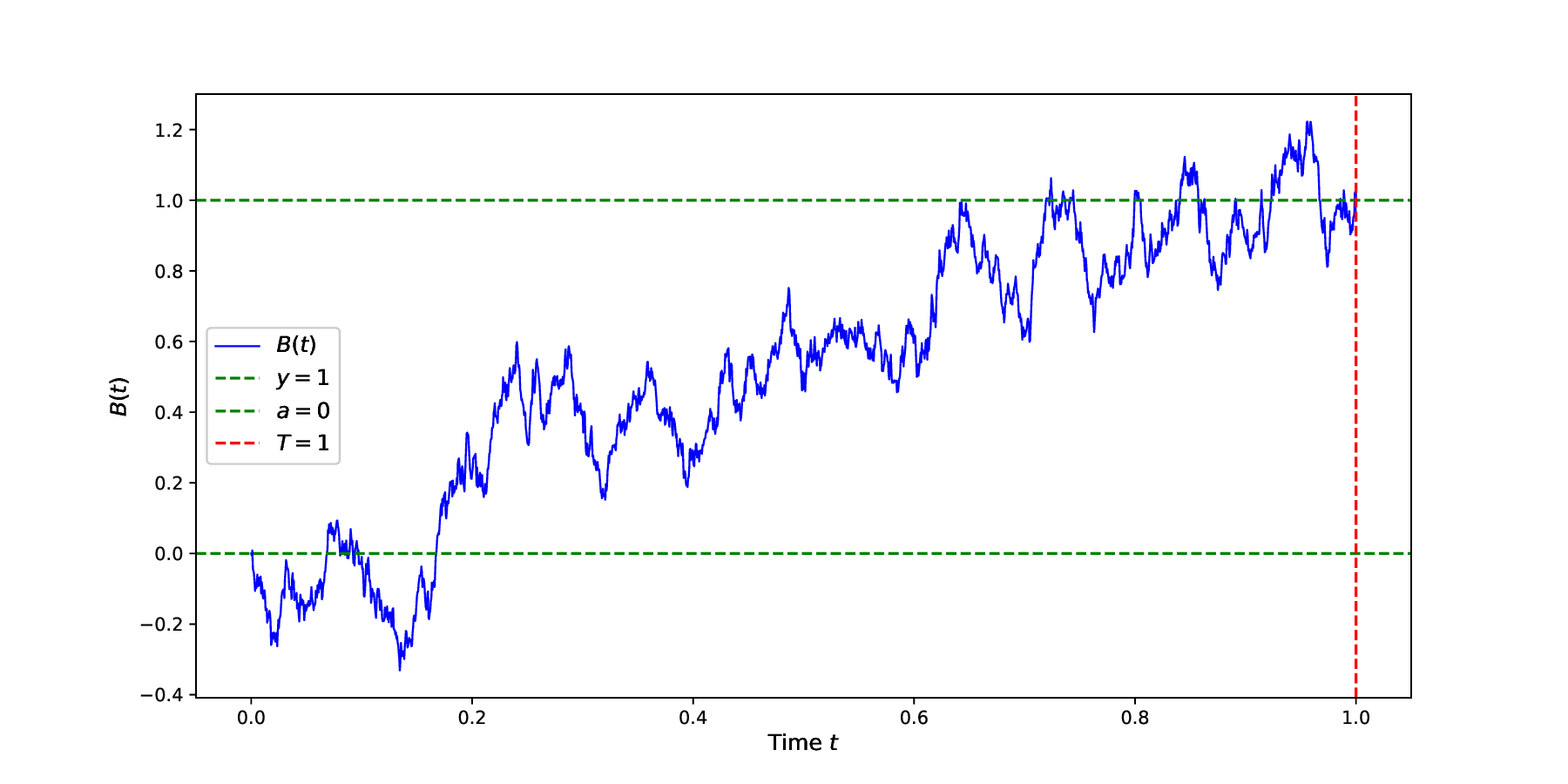}
      \caption{Simulation of the MrB process with initial condition  $a = 0$, terminal value $y = 1$, hitting time $T = 1$, and mean-reversion rate $\theta = 0.5$. \textmd{The process converges gradually to its terminal point $y$, and it reaches this point at $T$ with probability $1$.}}
      \label{fig:Standard_MrB}
 \end{center}
\end{figure}



\subsection{One-dimensional MrB with a Noisy View}
\label{sec42}

 We now consider the case when the view is noisy. We show that the conditional process is still an MrB but  with a modified hitting time and terminal value that are determined by the statistical properties of the view.



\begin{proposition}
\label{prop:1d_MrB}
Let $X(t)$ be a mean-reverting process from \eqref{eq:OU}. At $t=0$, we observe a noisy view of its terminal state, $Y(0,T) = X(T) + \epsilon$, where the noise $\epsilon \sim \mathcal{N}(0, \omega^2 / (2\theta))$ is independent of $X(t)$. Then the conditional process $\{B(t) = (X(t) \,|\, Y(0,T)=y), t \in [0,T]\}$ is the restriction to \([0,T]\) of a mean-reverting bridge from initial value $a$ to a  target $\tilde{y}$ and hitting time $\tilde{T}\geq T$, where
\begin{align*}
    \tilde{T} &= T + \delta \quad (\text{new hitting time}) \\
    \tilde{y} &= e^{-\theta\delta} y \quad (\text{new target value}).
\end{align*}
The time extension $\delta = \frac{\ln(1 + \omega^2)}{2\theta} \ge 0$ is increasing in the noise in the view \(\omega^2\).
\end{proposition}
The detailed proof of this proposition is given in Appendix \ref{Appendix:Proof_1d_MrB}.

This behavior highlights an important effect: Noise in the view \(\omega^2 > 0\) pushes the hitting time beyond the view's horizon \(T\) to a new time \(\tilde{T}\). The bridge now targets an adjusted value \(\tilde{y}\) at $\tilde{T}$, which is the original view \(y\) discounted back from the new hitting time \(\tilde{T}\) back to $T$. The greater the uncertainty \(\omega^2\), the larger the time extension \(\delta\), and the more the target \(\tilde{y}\) is discounted. If the view is certain (\(\omega^2 = 0\)), then \(\delta = 0\), and we recover the noise-free case. This is illustrated in Figure \ref{fig:Noisy_MrB}.


 \begin{figure}[ht]
 \begin{center}
       \includegraphics[width= 0.8\textwidth]{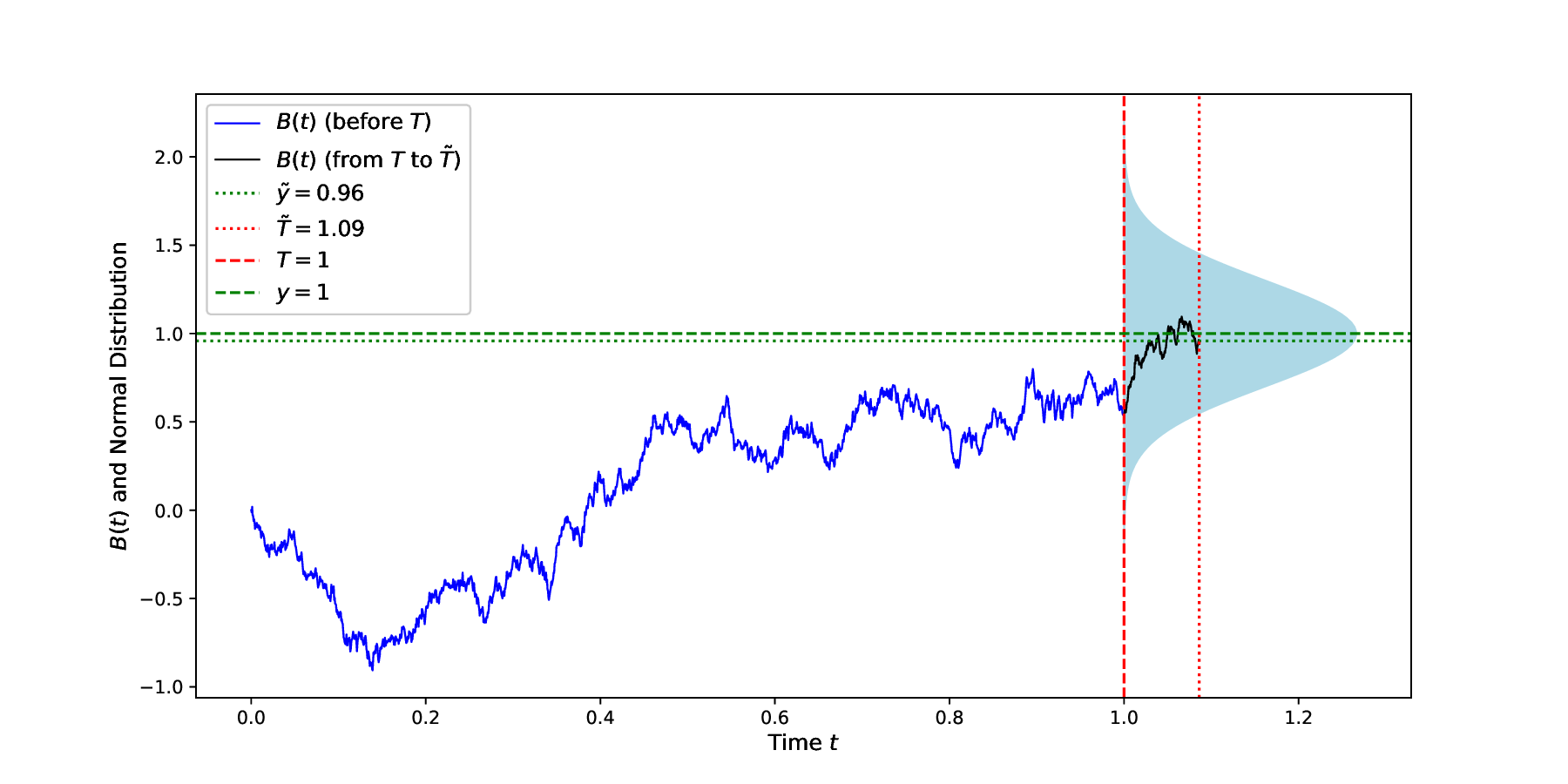}
     \caption{An MrB conditioned on a noisy view. Here $a=0, y=1, T=1, \theta=0.5, \omega=0.3$. The noise extends the hitting time to $\tilde{T} > T$ and discounts the target to $\tilde{y} < y$. The process (blue) evolves until $T=1$, with its future trajectory (black) converging to $\tilde{y}$ at time $\tilde{T}$.}
      \label{fig:Noisy_MrB}
 \end{center}
\end{figure}

\subsubsection{Application to the Factor Model}
This result connects directly to our asset pricing framework from Section \ref{sec3}. Consider a single risk factor \(X(t)\) with general mean reverting dynamics
\begin{equation*}
\begin{cases}
     dX(t) &= \theta (\mu - X(t)) dt + \sigma^X dW(t),\\
     X(0) &= a.
\end{cases}
\end{equation*}
This can be written as a scaled and shifted version of the standard process in \eqref{eq:OU}
\begin{equation*}
    \begin{split}
        X(t)  = (1 - e^{-\theta t})\mu + \sigma \bar{X}(t)
    \end{split}
\end{equation*}
where $\bar{X}$ is a mean-reverting process satisfying \eqref{eq:OU} with $\bar{X}(0) = a/\sigma$.
If an expert provides a noisy view $Y(0,T) = X(T) + \epsilon$ with $\epsilon \sim \mathcal{N}\big(0, \frac{\omega^2}{2\theta}\big)$,  observing \(X(T)\) with noise is equivalent to observing \(\bar{X}(T)\) with noise. From Proposition \ref{prop:1d_MrB}, we know that conditioning on this observation converts \(\bar{X}(t)\) into a MrB $B(t)$ from $a/\sigma$ with a terminal value
\begin{equation*}
    \tilde{y} = e^{-\theta \delta} \cdot \frac{y - (1-e^{-\theta T})\mu}{\sigma}
\end{equation*}
at the extended time $\tilde{T} = T + \delta$ where
\(\delta = \frac{1}{2\theta} \ln(1 + \omega^2 / \sigma^2)\). The conditional factor process \(X^y(t)\) is simply a scaled and shifted MrB
\begin{equation*}
    \begin{split}
        X^y(t) =(1 - e^{-\theta t})\mu + \sigma B(t).
    \end{split}
\end{equation*}
Thus, the MrB framework provides a complete and intuitive characterization of the factor's behavior under a forward-looking expert view.

\subsection{The MrB in a Multi-Dimensional Setting}
\label{sec43}
Extending the MrB concept to a multi-dimensional setting introduces significant complication. When factors are correlated, information about the future value of one factor provides information about the others. This creates complex cross-dependencies, and the intuitive ``pinning'' behavior of the one-dimensional bridge does not always carry over when views are noisy.

This can be observed even in the simpler case of conditioning one component of the factors on the terminal value of another component, as illustrated in the following example.

\begin{example}
Consider a two-dimensional mean-reverting process $X(t) \in \mathbb{R}^2$ with dynamics
\begin{eqnarray*}
\begin{cases}
dX_i(t) = -\theta_i X_i(t) dt + dW(t), & \quad i \in {1,2}, \\
X_i(0) = 0, & \quad i \in {1,2},
\end{cases}
\end{eqnarray*}
where both components $X_1(t)$ and $X_2(t)$ are driven by the same Brownian motion, but with distinct mean-reversion rates $\theta_1 \neq \theta_2$. Suppose we observe the noisy terminal state $Y(0,T) = X_2(T) + \epsilon$ with $\epsilon \sim \mathcal{N}(0, \omega^2)$, and define the conditional process
\(B(t) = \left(X(t)\,|\,Y(0,T)=y\right), \, t \in [0,T].
\)
It follows from Proposition \ref{prop:1d_MrB} that $B_2(t)$ is the restriction to $[0,T]$ of a MrB $(X_2(t) | X_2(\tilde{T}_2) = \tilde{y}_2)$ with terminal value $\tilde{y}_2 = e^{-\theta_2 \delta_2} y$ at $\tilde{T}_2 = T + \delta_2$, with $\delta_2 = \dfrac{1}{2\theta_2}\ln(1 + 2 \theta_2 \omega^2)$. However, one can verify there is no suitable terminal point \(\tilde{y}_1\) and hitting time \(\tilde{T}_1\) for $B_1(t)$ that satisfies \(B_1(t) \stackrel{(d)}{=} (X_1(t) \, | \, X_1(\tilde{T}_1) = \tilde{y}_1), \, t \in [0.T] \). 
\end{example}

This example shows that a formal definition is required to specify what constitutes a multidimensional bridge and its characteristics. 

\begin{definition}[Multidimensional MrB]
\label{def:m-MrB}
Let $X(t) \in \mathbb{R}^d$ be a mean-reverting process satisfying 
\begin{equation}
\label{eq:multi_OU}
        dX(t) = -\Theta X(t) dt + L^X dW(t),\quad
        X(0) = a.
\end{equation}
Let  $\tilde{T} = (\tilde T_1, \dots, \tilde T_d)^\top$ be a vector of hitting times. The process
\begin{equation*}
    B(t) = \left( X(t) \mid X_1(\tilde T_1)=y_1, \dots, X_d(\tilde T_d)=y_d \right)
\end{equation*}
is called a multidimensional-MrB (m-MrB) from $a$ to $y = (y_1, \dots, y_d)^\top$ with mean-reversion rate $\Theta$ and hitting times $\tilde{T}$.
\end{definition}

An m-MrB is then a process where each component is individually pinned to a specific target \(y_i\) at a specific time \(\tilde T_i\). The properties of such a process are a natural extension of the one-dimensional case (see Proposition \ref{prop:m-MrB_Properties} in the Appendix). 

We now establish conditions under which a noisy view \(Y(0,T) = PX(T) + \epsilon\) transforms a mean-reverting process into such an m-MrB. This extends  Proposition \ref{prop:1d_MrB} to the multi-dimensional setting and provides a characterization of the conditional factors dynamics derived in Proposition \ref{prop:X^y} in terms of m-MrB. The key insight is that this occurs when the \emph{information} from the noisy view is equivalent to the information from an exact, but time-extended, future observation. To formalize this, we first characterize the gain in precision of an m-MrB.

\begin{proposition}
\label{prop:precision_gain}
Let $\delta = (\delta_1,\cdots\delta_d)^\top$ where $\delta_i> 0$, \(\tilde{T} = T \, \mathbf{1}_d + \delta\) be a time extension, and    \begin{equation}
\label{eq:precision_gain1}
        \mathcal{P}(\delta; \Theta, \Sigma^X) :=  \mathbb{V}[X(T) \, | \, X(\tilde{T})]^{-1} - \mathbb{V}[X(T)\, | \, X(0)]^{-1} \in \mathbb{R}^{d \times d}.
    \end{equation}
    be the gain in the precision of $X(T)$ after conditioning on $X(\tilde{T})$.

    When \(X(t)\) satisfies \eqref{eq:multi_OU}, the gain in precision is 
    \begin{equation}\label{eq:precision_gain2}
    \mathcal{P}(\delta; \Theta, \Sigma^X) = F(\delta)^\top C(\delta)^{-1} F(\delta)\end{equation}
where \(
    C(\delta) := \mathbb{V}[X(\delta) \, | \, X(0)] \in \mathbb{R}^{d \times d}\) is the covariance of \(X\) at time \(\delta\) with elements  \[C(\delta)_{ij}
\;=\;\int_0^{\min\{\delta_i,\delta_j\}}
e_i^\top e^{-\Theta(\delta_i-u)}\,\Sigma^X\,e^{-\Theta^\top(\delta_j-u)}e_j\,du,
\]
and  \(F(\delta) \in \mathbb{R}^{d \times d}\) with \(i^{th}\) column \((F(\delta))_i = e^{-\Theta \delta_i}e_i\)
  where $e_i$ the $i^{th}$ canonical basis vector in $\mathbb{R}^d$.
\end{proposition}
Several observations to note: (i) the expression of $\mathcal P(\delta; \Theta, \Sigma^X)$ does not depend on the actual realization of $X(\tilde{T})$, since \(X\) is Gaussian; (ii) and the gain in precision depends only on the magnitude of the extension $\delta$ because \(X\) is Markovian. The proof of this result is given in Appendix \ref{EC:Proof_PrecisionMatrices}.

The following result shows that the mean-reverting process conditional on noisy views is an m-MrB  if and only if the precision gained from the expert view, \(P^\top \Omega^{-1} P\), aligns with the gain in precision \eqref{eq:precision_gain2} from known $X(\tilde{T})$ for some time extension \(\delta\).
\begin{theorem}
\label{thm:m-MrB}
Let $X(t)$ be a mean-reverting process satisfying \eqref{eq:multi_OU} and $Y(0,T)$ be a noisy view satisfying \eqref{eq:Y_view}. The conditional process $B(t) := (X(t) \, | \, Y(0,T) = y)$ is an m-MrB if and only if there exists a time-extension vector $\delta \in (0, \infty)^d$ such that
\begin{equation}
\label{eq:alignment_condition}
    P^\top \Omega^{-1} P = \mathcal{P}(\delta; \Theta, \Sigma^X).
\end{equation}
Furthermore, when this alignment condition holds:
\begin{enumerate}
    \item The centered process $B(t) - \mathbb{E}[B(t)]$ is an m-MrB from $\mathbf{0}$ to $\mathbf{0}$ with hitting times $\tilde{T} = T \mathbf{1}_d + \delta$.
    \item If, in addition, $P$ has full column rank, then $B(t)$ is an m-MrB from $a$ to a unique target $\tilde{y} \in \mathbb{R}^d$ with hitting times $\tilde{T}$, where the target is given by
    \begin{equation*}
        \tilde{y} = C(\delta) F(\delta)^{-\top} P^\top\Omega^{-1} y.
    \end{equation*}
\end{enumerate}
\end{theorem}
\begin{proof}
See Appendix \ref{Appendix:secC4} for the full proof.
\end{proof}
To understand the condition \eqref{eq:alignment_condition}, notice that \[P^\top\Omega^{-1} P = \mathbb{V}[X(T) \, | \, Y(0,T) = y]^{-1} - \mathbb{V}[X(T)\, | \, X(0)]^{-1}\]
is the gain in precision of \(X(T)\) when observing the views \(Y(0,T) = y\). Thus, condition \eqref{eq:alignment_condition} requires that the investor gain the same amount of precision from the views \(Y(0,T) =y\) as she would have gained from knowing the true value of \(X\) at some future time \(\tilde{T}\). When \(P\) is not full rank, the views do not contain enough information to uniquely determine the target vector \(\tilde y\), which is why only the centered process \(B(t) - \mathbb{E}[B(t)]\) is guaranteed to be an m-MrB.

This alignment condition becomes clearer in specific cases. For instance, if the factors are dynamically independent (i.e., \(\Theta\) and \(\Sigma^X\) are diagonal), a bridge forms only if the view provides decoupled information about each factor, which requires \(P^\top \Omega^{-1} P\) to also be diagonal. 

In the single factor case \((d=1)\), the precision-gain operation simplifies to 
\[\mathcal{P}(\delta; \theta, \sigma_X) = \frac{2\theta}{\sigma_X^2}(e^{2\theta \delta} - 1)^{-1}.\]
with \(\sigma_X^2 = L^X (L^X)^\top\). When \(K\) views are given on a single factor \(X(t) \in \mathbb{R}\) with \(P = \mathbf{1}_K \in \mathbb{R}^{K \times 1}\) and \(\epsilon \sim \mathcal{N}\left(0, \Omega\right) \in \mathbb{R}^K\), then~\eqref{eq:alignment_condition} always holds and the conditional process \((X(t) \, | \, Y(0,T) = y)\) is a MrB with time extension
\[\delta = \frac{1}{2\theta} \ln\left(1 + (P^\top \Omega^{-1} P)^{-1} \frac{2 \theta}{\sigma_X^2}\right).\]
Since \(P\) has full column rank, the target value of the MrB implied by Theorem \ref{thm:m-MrB} is
\[\tilde{y} = (P^\top \Omega^{-1}P)^{-1} e^{-\theta \delta} P^\top \Omega^{-1} y.\]

\begin{remark}
    While the alignment condition \eqref{eq:alignment_condition} provides helpful intuition, it is \emph{not} required for solving the control problem. Regardless of whether \eqref{eq:alignment_condition} holds, the conditional process \(\{X(t) \, | \, Y(0,T) = y\}\) is well-defined and satisfies the SDE in Proposition \ref{prop:X^y}. This is sufficient for the portfolio choice results in Section \ref{sec5}.
\end{remark}

\section{Control Problem}
\label{sec5}
We now solve the dynamic portfolio problem for an investor seeking to maximize the expected utility of terminal wealth. The investor begins with access to expert views \(Y(0,T) = y\) and operates within the conditional market dynamics derived in Section \ref{sec3}. The asset prices and risk factors evolve according to 
\begin{equation*}
    \begin{split}
        dS^y(t) &= D(S^y(t))\left(\left(\tilde{\alpha}(t,y) + \tilde{\beta}(t) X^y(t) \right)dt + L^S dW^\mathbb{Q}(t)\right),\\
        dX^y(t) &= \tilde{\Theta}(t) \left(\tilde{\mu}(t,y) - X^y(t)\right)  dt   + L^X dW^\mathbb{Q}(t).
    \end{split}
\end{equation*}
Let $\pi_i(t)$ be the fraction of wealth invested in the $i^{th}$ risky asset at time $t$. This implies that $1 - \pi(t)^\top \mathbf{1}_N$ is the fraction of wealth invested in the risk-free asset with fixed return $r_f$. We assume that $\pi(t) \in \Pi$ where the class of admissible policies
\begin{equation*}
    \Pi = \left\{ \pi: [0,T] \to \mathbb{R}^N, \pi \text{ is adapted to } \{\mathcal{F}_t^Y\}_{t \in [0,T]}, \int_{0}^T |\pi(t)|^2 dt \leq \infty\right\}.
\end{equation*}

We denote by $Z(t)$ the wealth process and assume that the portfolio is self-financing, the investor's wealth thus satisfies
\begin{equation}
\label{eq:dZ}
    dZ(t) = Z(t)\bigg(r_f dt + \pi(t)^\top \big(\mu^S(t,X^y(t),y) - r_f \mathbf{1}_N\big) dt + \pi(t)^\top L^S dW^\mathbb{Q}(t)\bigg).
\end{equation}
The investor's objective is to maximize the expected isoelastic (CRRA) utility of terminal wealth 
\begin{equation*}
    U(Z) = \frac{z^{1-\gamma}}{1-\gamma}
\end{equation*}
where $\gamma$ is the coefficient of relative risk aversion. Her value function is
\begin{equation*}
    V(t,z,x) = \max_{\pi \in \Pi} \mathbb{E}_\mathbb{Q}\bigl[U(Z(T)) | Z(t) = z, X^y(t) = x\bigr].
\end{equation*}
Throughout the analysis, we assume the investor is risk-averse, with \(\gamma > 1\).

\subsection{Value Function and Optimal Policy}
\label{sec51}
The value function \(V(t,z,x)\) satisfies the Hamilton-Jacobi-Bellman (HJB) equation
\begin{equation}
\label{eq:HJB}
\begin{split}
    \max_{\pi} \Big\{&\frac{\partial V}{\partial t} +  z \left( r_f + \pi(t)^\top\left(\tilde{\alpha}(t,y) + \tilde{\beta}(t) x - r_f \mathbf{1}_N \right) \right) \nabla_z V
  +  \left(\tilde{\Theta}(t)(\tilde{\mu}(t,y) - x)\right)^\top \nabla_x V  \\ &   + \frac{1}{2} z^2 \pi(t)^\top \Sigma^S \pi(t)  \nabla^2_z V + \frac{1}{2} \Tr(\Sigma^X \nabla_x^2 V)  + z \pi(t)^\top \Sigma^{S,X} \nabla^2_{x,z} V\Big\} = 0
\end{split}
\end{equation}
with terminal condition \(V(T,z,x) = U(z)\). Here,  $\Sigma^{S,X} = L^S(L^X)^\top$ is the covariance matrix between asset and factor innovations. The first-order condition yields the optimal policy's structure
\begin{equation*}
    \pi^*(t) = \underbrace{- (\Sigma^S)^{-1} \frac{\nabla_z V}{ z \nabla_z^2 V} (\tilde{\alpha}(t,y) + \tilde{\beta}(t) x - r_f \mathbf{1}_N \big)}_{\text{Mean-Variance Holding}} \underbrace{-(\Sigma^S)^{-1}\Sigma^{S,X} \frac{ \nabla^2_{x,z}V}{ z \nabla_z^2 V}}_{\text{Hedging}}.
\end{equation*}
The policy consists of a \emph{mean-variance} component, which is optimal for a myopic mean-variance problem, and an \emph{intertemporal hedging demand}, which hedges changes in investment opportunities (i.e., movements in the factors \(X^y(t)\)); see \cite{ChackoViceira2005}. 

The investor's problem \eqref{eq:HJB} can be solved by standard HJB methods for CRRA utility (see, e.g., \cite{MERTON1971373, oksendal2003}). For \(\gamma > 1\), the value function is given by 
\begin{equation*}
    V(t,z,x) = \frac{z^{1-\gamma}}{1-\gamma}e^{g(t,x)}
\end{equation*}
where
\begin{equation*}
    g(t,x) = \frac{1}{2}x^\top A(t) x + x^\top b(t) + c(t).
\end{equation*}
The matrix $A(t) \in \mathbb{R}^{d \times d}$ is symmetric negative semi-definite and satisfies the Riccati equation
\begin{equation}
\label{eq:Riccati_views}
    \begin{cases}
        A'(t) &+ \frac{1-\gamma}{\gamma} \tilde{\beta}(t)^\top (\Sigma^S)^{-1} \tilde{\beta}(t) + A(t) \big(\Sigma^X + \frac{1-\gamma}{\gamma} (\Sigma^{S,X})^\top (\Sigma^S)^{-1} \Sigma^{S,X}\big) A(t) \\
        &+ A(t) \big(\frac{1-\gamma}{\gamma}(\Sigma^{S,X})^\top (\Sigma^S)^{-1} \tilde{\beta}(t) - \tilde{\Theta}(t)\big) + \big(\frac{1-\gamma}{\gamma}(\Sigma^{S,X})^\top (\Sigma^S)^{-1} \tilde{\beta}(t) - \tilde{\Theta}(t)\big)^\top A(t) = 0,\\
        A(T) &= 0,
    \end{cases}
\end{equation}
where
\begin{equation*}
    \begin{cases}
        \eta(t) &= (P e^{-\Theta (T-t)} L^X)^\top \big(P (\Sigma - e^{-\Theta(T-t)} \Sigma e^{-\Theta^\top (T-t)})P^\top + \Omega\big)^{-1},\\
        \tilde{\Theta}(t) &= \Theta + L^X \eta(t) P e^{-\Theta (T-t)},\\
        \tilde{\beta}(t) &= \beta -  L^S \eta(t) P e^{-\Theta (T-t)}    .
    \end{cases}
\end{equation*}
The coefficient vector $b(t) \in \mathbb{R}^d$ solves the ODE system
\begin{equation}
\label{eq:ODEsys_views}
    \begin{cases}
        b'(t) &+ \frac{1-\gamma}{\gamma} \big( \tilde{\beta}(t)^\top + A(t) (\Sigma^{S,X})^\top\big) (\Sigma^S)^{-1} \big( \Sigma^{S,X} b(t) + \tilde{\alpha}(t,y) - r_f \mathbf{1}_N \big) \\ &+ \big(A(t) \Sigma^X - \tilde{\Theta} (t)^\top\big) b(t) + A(t) \tilde{\Theta}(t)\tilde{\mu}(t,y) = 0, \\
        b(T) &= 0,
    \end{cases}
\end{equation}
where 
\begin{equation*}
\begin{cases}
     \tilde{\mu}(t,y) &= \mu +  \tilde{\Theta}(t)^{-1} L^X \eta(t) \left(y - P \mu \right)\\
    \tilde{\alpha}(t,y) &= \alpha + L^S \eta(t) \left(y - P (1-e^{-\Theta (T-t)})\mu \right).
    \end{cases}
\end{equation*}
The scalar  $c(t) \in \mathbb{R}$ is found by integrating the following ODE
\begin{equation}
\label{eq:Int_C_views}
    \begin{cases}
        c'(t) &+(1-\gamma)r_f + \frac{1}{2}\Tr(\Sigma^X A(t)) + \tilde{\mu}(t,y)^\top b(t) + \frac{1}{2}b(t)^\top \Sigma^X b(t)\\
        &+ \frac{1-\gamma}{2\gamma}\big(\Sigma^{S,X} b(t) + \tilde{\alpha}(t,y)   - r_f \mathbf{1}_N \big)^\top (\Sigma^S)^{-1}\big(\Sigma^{S,X} b(t) + \tilde{\alpha}(t,y) - r_f \mathbf{1}_N \big) = 0,\\
        c(T) &= 0.
    \end{cases}
\end{equation}
The unique optimal allocation strategy is
\begin{equation}
\label{eq:optimalPolicy}
    \pi^*(t) = \frac{1}{\gamma} (\Sigma^S)^{-1} \left(\tilde{\alpha}(t,y) + \tilde{\beta}(t) x - r_f \mathbf{1}_N\right) + \frac{1}{\gamma}  (\Sigma^S)^{-1}\Sigma^{S,X}\frac{\partial g}{\partial x}(t,x)
\end{equation}
where
\begin{equation*}
    \frac{\partial g}{\partial x}(t,x) = A(t) x + b(t).
\end{equation*}

A key consequence of this framework is that access to expert views is, on average, always beneficial.
\begin{proposition}
\label{prop:Views_are_good}
     Let $V_0(t,z,x)$ be the value function of an investor without access to views. Then
     \begin{equation*}
         \mathbb{E}_y [V(t,z,x) \,|\, X^y(t) = x, Z(t) = z] \geq V_0(t,z,x).
     \end{equation*}
 \end{proposition}
 The proof of the results in this section are given in Appendix  \ref{App:ProofHJB} and \ref{App:Proof-prop-views-good}.


\subsection{Optimal Policy Simplification}
From \eqref{eq:optimalPolicy}, the optimal policy \(\pi^*(t)\) consists of a mean-variance component 
 \begin{eqnarray}
 \label{eq:MV_term}
     \frac{1}{\gamma} (\Sigma^S)^{-1} \left(\tilde{\alpha}(t,y) + \tilde{\beta}(t) x - r_f \mathbf{1}_N\right),
 \end{eqnarray}
 and an intertemporal hedging demand 
  \begin{eqnarray}
 \label{eq:Hedging_term}
     \frac{1}{\gamma}  (\Sigma^S)^{-1}\Sigma^{S,X}\frac{\partial g}{\partial x}(t,x) =  \frac{1}{\gamma}  (\Sigma^S)^{-1}\left(\Sigma^{S,X} A(t) x + b(t)\right).
 \end{eqnarray}
The properties  $(P, \Omega)$  of the views and the realization $y$ affect the optimal policy by modifying the dynamics of the ODEs \eqref{eq:Riccati_views} and \eqref{eq:ODEsys_views}. These modifications are complicated and it is unclear (qualitatively) how they affect the optimal policy \eqref{eq:optimalPolicy}.

The following result provides a simpler expression for the optimal policy that isolates the impact of the views. In particular, views change the terminal conditions of ODEs that characterize the optimal portfolio where there are no views. 


\begin{theorem}
\label{thm:Optimal_policy}
The optimal policy with views \eqref{eq:optimalPolicy} can be written as
\begin{equation}
\label{eq:Optimal_Policy_simple}
    \pi^*(t) = \frac{1}{\gamma} (\Sigma^S)^{-1} \left( \alpha + \beta x - r_f \mathbf{1}_N\right) + \frac{1}{\gamma} (\Sigma^S)^{-1} \Sigma^{S,X} \left(A_1(t)x + b_1(t) \right)
\end{equation}
where $A_1(t) \in \mathbb{R}^{d \times d}$ solves the constant-coefficients Riccati 
\begin{equation}
\label{eq:Riccati_no_views}
    \begin{split}
        A_1'(t) &+ \frac{1-\gamma}{\gamma} \beta^\top (\Sigma^S)^{-1} \beta + A_1(t) \big(\Sigma^X + \frac{1-\gamma}{\gamma} (\Sigma^{S,X})^\top (\Sigma^S)^{-1} \Sigma^{S,X}\big) A_1(t) \\
        &+ A_1(t) \big(\frac{1-\gamma}{\gamma}(\Sigma^{S,X})^\top (\Sigma^S)^{-1} \beta - \Theta\big) + \big(\frac{1-\gamma}{\gamma}(\Sigma^{S,X})^\top (\Sigma^S)^{-1} \beta - \Theta\big)^\top A_1(t) = 0
    \end{split}
\end{equation}
and \(b_1(t) \in \mathbb{R}^d\) solve the ODEs system
\begin{equation}
\label{eq:ODEsys_no_views}
    \begin{split}
        b_1'(t) &+ \frac{1-\gamma}{\gamma} \big( \beta^\top + A_1(t) (\Sigma^{S,X})^\top\big) (\Sigma^S)^{-1} \big( \Sigma^{S,X} b_1(t) + \alpha - r_f \mathbf{1}_N \big) \\ &+ \big(A_1(t) \Sigma^X - \Theta^\top\big) b_1(t) + A_1(t) \Theta\mu = 0.
    \end{split}
\end{equation}
The information from the expert views is incorporated exclusively through the terminal conditions
\begin{equation*}
        A_1(T) = - P^\top \Omega^{-1} P, \quad \text{and}\quad
        b_1(T) = P^\top \Omega^{-1}y.
\end{equation*}
Furthermore, \(A_1(t)\) is negative semi-definite on \([0,T]\). If \(A_1^{(1)}(t)\) and \(A_1^{(2)}(t)\) are solutions to \eqref{eq:Riccati_no_views} corresponding to view covariance matrices \(\Omega^{(1)}\) and \(\Omega^{(2)}\), respectively, where   \((\Omega^{(1)})^{-1} \succeq (\Omega^{(2)})^{-1} \), then \(A_1^{(1)}(t) \preceq A_1^{(2)}(t)\). This means \(A_1(t)\) becomes more negative semi-definite as the views precision, represented through \(\Omega^{-1}\), increases.
\end{theorem}
The proof of this theorem is given in Appendix \ref{App:Proof-thm-sec5}.

The result enables us to compare the optimal policy with views to the one without. From Section \ref{sec51}, the optimal policy without views is
 \begin{equation}
 \label{eq:pi_0}
     \pi_0^*(t) = \frac{1}{\gamma} (\Sigma^S)^{-1} \big(\alpha + \beta x - r_f \mathbf{1}_N\big) + \frac{1}{\gamma}  (\Sigma^S)^{-1}\Sigma^{S,X}\left(A_0(t)x + b_0(t)\right)
 \end{equation}
where \((A_0(t), b_0(t))\) solve the same ODEs as \((A_1(t), b_1(t))\) but with zero terminal conditions, \(A_0(T) = 0\) and \(b_0(T) = 0\).\footnote{This can be directly derived from the optimal policy characterization in Section \ref{sec51} by letting the view uncertainty \(\Omega \to \infty\). In this case, \(\tilde{\alpha}(t,y) \to \alpha\), \(\tilde{\beta}(t) \to \beta\), \(\tilde{\Theta}(t) \to \Theta\), and \(\tilde{\mu}(t,y) \to \mu\).} The aggregate impact of views is to change the terminal conditions of the ODEs \eqref{eq:Riccati_views} and \eqref{eq:ODEsys_views} which affects the second term in \eqref{eq:pi_0}. 
The optimal policy when there are views can now be written in terms of the no-views policy:
  \begin{equation}
    \label{eq:OptimalPolicy=pi_0+hedge}
        \pi^*(t) = \pi_0^*(t) + H(t,x)
    \end{equation}
    where \[H(t,x) = \frac{1}{\gamma} (\Sigma^S)^{-1}\Sigma^{S,X} \big( Q(t) x + q(t) \big)\] is the adjustment to the policy induced by expert views with  \(Q(t) := A_1(t) - A_0(t)\) and \(q(t):= b_1(t) - b_0(t)\).  The sensitivity of this adjustment to factor movements is
    \begin{eqnarray*}
    \frac{\partial H}{\partial x}(t,x) = \frac{1}{\gamma} (\Sigma^S)^{-1}\Sigma^{S,X} Q(t).
\end{eqnarray*}
From Theorem \ref{thm:Optimal_policy}, the matrix \(Q(t)\) is negative semi-definite and becomes more negative as the view precision \(\Omega^{-1}\) increases. Consequently, the magnitude of the views-induced adjustment  \(||H(t,x)||_2\) increases with precision.
When shocks to factors and assets are independent (i.e., \(\Sigma^{S,X} = 0\)), the adjustment term vanishes and views have no effect on the optimal policy.

\section{Unknown Drift and Dynamic Learning}
\label{sec6}
The analysis so far assumes that the model parameters are fully known at the start of the investment horizon. We now address the case where the long-run mean \(\alpha \in \mathbb{R}^N\) is not known but can be learned online using market observations. The investor now continuously updates her beliefs about the value of~\(\alpha\) while optimizing her portfolio.


We assume price dynamics of the form
\[dS(t) = D(S(t)) \left( \left(\alpha + \beta X(t) \right) dt + L^S dW(t) \right).\]
We adopt a normal prior \[\alpha \sim \mathcal{N}\left(\alpha_0, \Gamma_0\right)\]
and assume that it is independent of the risk factors \(X\). It follows that it is also independent of the expert views \(Y(0,T)\) and the conditional asset price dynamics from Corollary \ref{corr:S^y} still hold, with \(\alpha\) now treated as an unobserved random variable
\begin{equation*}
    dS^y(t) = D(S^y(t))\big((\alpha + \tilde{\lambda}(t,y) + \tilde{\beta}(t) X^y(t)) dt + L^S dW^\mathbb{Q}(t)\big)
\end{equation*}
where \begin{equation}
\begin{split}
\label{eq:tilde_lambda_beta}
    \tilde{\lambda}(t,y) &:= \tilde{\alpha}(t,y) - \alpha = L^S \eta(t) \left(y - ( 1-e^{-\Theta (T-t)})\mu\right),\\
    \tilde{\beta}(t) &= \beta -  L^S \eta(t) P e^{-\Theta (T-t)}
    \end{split}
\end{equation} 
(see Equation \eqref{eq:coeff_mu^S}). The conditional factor dynamics for \(X^y(t)\) are given by Proposition \ref{prop:X^y}.
\subsection{Learning the Drift from Market Data}
\label{sec61}
To learn the value of the long-run mean of the drift \(\alpha\), the investor must filter it from incoming market data. The most direct source of information is the asset prices themselves. We can construct an observation process \(dN_1^y(t)\) from the vector of asset log-returns by stripping away the known component of the drift
\begin{equation}
\label{eq:Observations_dN}
    dN_1^y(t) := dR^y(t) -\left( \tilde{\lambda}(t,y) + \tilde{\beta}(t) X^y(t) - \frac{1}{2}\diag\left(L^S (L^S)^\top\right)\right) dt = \alpha dt + L^S dW^\mathbb{Q}(t),
\end{equation}
where 
\[
    R^y(t) :=  \left[\ln\left(\frac{S_1^y(t)}{S_1^y(0)}\right), \cdots,\ln\left(\frac{S_N^y(t)}{S_N^y(0)}\right)\right]^\top \in \mathbb{R}^N
\]
is the conditional log-returns vector. Intuitively, observing \(dN_1^y(t)\) allows the investor to estimate \(\alpha\), though the signal contains market noise \(L^S dW^\mathbb{Q}(t)\).

However, this is not the only source of information. When asset and factor innovations are correlated (\(\Sigma^{S,X} \neq 0\)), the factor dynamics contain information about the shared noise source \(dW^\mathbb{Q}(t)\). To exploit this, we define a second observation process from the factor innovations 
\begin{equation}
    \label{eq:Obs2}
    dN_2^y(t) := dX^y(t) - \tilde{\Theta}(t)(\tilde{\mu}(t,y) - X^y(t)) dt = L^X dW^\mathbb{Q}(t).
\end{equation}
By observing \(dN_2^y(t)\), the investor gains a partial view of the noise term \(dW^\mathbb{Q}(t)\), which reduces the noise in the returns signal \(dN_1^y(t)\). The amount of information gained depends on the correlation structure \(\Sigma^{S,X} = L^S (L^X)^\top\). 

Combining these into a single observation vector \(dN^y(t) = [dN_1^y(t)^\top, dN_2^y(t)^\top]^\top \in \mathbb{R}^{N + d}\) yields the joint observation model
\begin{equation}
\label{eq:Joint-Observation}
    dN^y(t) = H \alpha dt + G dW^\mathbb{Q}(t)
\end{equation}
where the observation matrix \(H\) and the joint diffusion matrix \(G\) are
\begin{equation*}
    H = \begin{pmatrix}
        I_N \\ \mathbf{0}_{d \times N}
    \end{pmatrix}, \quad G = \begin{pmatrix}
        L^S \\ L^X
    \end{pmatrix}.
\end{equation*}
Given the linear structure of the observations \eqref{eq:Joint-Observation}, the posterior distribution of \(\alpha\) is normal with mean and variance given by the Kalman filter (see, for example Lemma 6.2.2 of \cite{oksendal2003}).

\begin{proposition}\label{prop:KF}The estimated drift \[\hat{\alpha}(t) := \mathbb{E}\bigr[\alpha \, | \, \mathcal{F}_t^Y\bigl]\] evolves according to the Kalman filter equation
    \begin{equation}
    \label{eq:hat_alpha}
            d\hat{\alpha}(t) = K(t)(dN^y(t) - H \hat{\alpha}(t) dt),
    \end{equation}
    with initial condition \(\hat{\alpha}(0) = \alpha_0\).
    The Kalman Gain matrix \(K(t) \in \mathbb{R}^{N \times (N+d)}\) is given by
    \begin{equation} \label{eq:KF_Gain} K(t) = \Gamma(t) H^\top (GG^\top)^{-1}.\end{equation}
    The estimation error covariance \(\Gamma(t) = \mathbb E\big[(\alpha - \hat{\alpha}(t))(\alpha - \hat{\alpha}(t))^\top \, | \, \mathcal F_t^Y \big] \in \mathbb{R}^{N \times N}\) is deterministic and given explicitly by 
        \begin{equation}
    \label{eq:Gamma_t}
        \Gamma(t) =  \left(\Gamma_0^{-1} + t \, (\Sigma^S - \Sigma^{S,X} (\Sigma^{X})^{-1} (\Sigma^{S,X})^\top)^{-1}\right)^{-1}.
    \end{equation} 
    The innovation process \begin{equation}
        \label{eq:Innovations_KF}
        V^\mathbb{Q}(t) := N^y(t) - \int_0^t H \hat{\alpha}(s) ds
    \end{equation}
    is a Brownian motion with quadratic variation \(dV^\mathbb{Q}(t) (dV^\mathbb{Q}(t))^\top = GG^\top dt\).
\end{proposition}
The precision of the drift estimate \(\Gamma(t)^{-1}\) highlights the values of factors in refining the estimate \(\hat{\alpha}(t)\). From \eqref{eq:Gamma_t}, we can express the rate of change in the precision as
\begin{equation}
\label{eq:Precision}
\begin{split}
        \frac{d}{dt}\Gamma(t)^{-1} &= (\Sigma^S - \Sigma^{S,X} (\Sigma^{X})^{-1} (\Sigma^{S,X})^\top)^{-1}\\
        &= \underbrace{(\Sigma^S)^{-1}}_{\text{Learning from asset data}} + \underbrace{(\Sigma^S)^{-1}\Sigma^{S,X}\left(\Sigma^X - (\Sigma^{S,X})^{\top}(\Sigma^S)^{-1}\Sigma^{S,X}\right)^{-1} (\Sigma^{S,X})^\top (\Sigma^S)^{-1}}_{\text{Learning from factor movement}}
\end{split}
\end{equation}
The first term \((\Sigma^S)^{-1}\) represents the precision gain from observing asset returns. The second term, which is positive semi-definite, represents the \emph{additional} precision gained by observing the correlated factors. Notably, when shocks to assets and factors are independent (i.e., \(\Sigma^{S,X} = 0\)), the second term becomes zero.
Note too that the \emph{precision} of the drift estimation is independent of the views \(Y(0,T)\), meaning the  uncertainty in the investor's estimate of $\alpha$ declines at the same rate with and without views.

\subsection{Optimal Policy under Drift Uncertainty}
\label{subsec:HJB_KF}

The investor's belief about the unknown drift \(\hat{\alpha}(t)\) evolves according to the Kalman filter equations derived in Proposition \ref{prop:KF}. Because this estimate changes stochastically as new market data arrives, \(\hat{\alpha}(t)\) becomes an additional \emph{state variable} in the investor's control problem. 

To formalize this, observe from \eqref{eq:Innovations_KF} that the dynamics of the drift estimate \(\hat{\alpha}(t)\) can be expressed as
\begin{equation}
    d\hat{\alpha}(t) = K(t) dV^\mathbb{Q}(t)
\end{equation}
where \(V^\mathbb{Q}(t)\) is a Brownian motion with volatility \(G\), 
while the dynamics of the risk factors from \eqref{eq:dX^y} remain
\begin{equation*}
\begin{split}
        dX^y(t) &= \tilde{\Theta}(t) (\tilde{\mu}(t,y) - X^y(t)) \, dt + L^X dW^\mathbb{Q}(t).
\end{split}
\end{equation*}
The conditional dynamics of the asset prices is
\begin{equation}
\label{eq:Price-process-filtering}
    dS^y(t) = D(S^y(t)) \bigg( \left(\hat{\alpha}(t) + \tilde{\lambda}(t,y) + \tilde{\beta}(t)X^y(t) \right)dt + L^S dW^\mathbb{Q}(t) \bigg)
\end{equation}

To solve the control problem, we define an augmented state vector \(M(t) = [X^y(t)^\top, \hat{\alpha}(t)^\top]^\top \in \mathbb{R}^{d+N}\) that combines both the observable risk factors and the learned drift estimate. The joint dynamics of this augmented state can be written in a mean-reverting form
\begin{equation}
\label{eq:M(t)}
    dM(t) = \Theta^M(t) (\mu^M(t) - M(t)) dt + \begin{pmatrix}
        L^X & K(t)
    \end{pmatrix}^\top dW^{M,\mathbb{Q}}(t)
\end{equation}
where 
\begin{equation*}
    W^{M,\mathbb{Q}}(t) := [W^\mathbb{Q}(t)^\top, V^\mathbb{Q}(t)^\top]
\end{equation*}
is a Brownian motion and the mean-reverting and long-run mean coefficients are
\[\tilde{\Theta}^M(t) = \begin{pmatrix} \tilde{\Theta}(t) & \mathbf{0} \\ \mathbf{0} & \mathbf{0} \end{pmatrix}, \quad \tilde{\mu}^M(t, y) = \begin{pmatrix} \tilde{\mu}(t,y) \\ \mathbf{0} \end{pmatrix}.\]
Observe that 
\[dM(t) (dM(t))^\top = L_t^M (L_t^M)^\top\] with 
\begin{equation*}
    L^M_t := \begin{pmatrix} L^X \\ K(t) G \end{pmatrix} \in \mathbb{R}^{(d + N ) \times N'}.
\end{equation*}
The conditional asset prices \eqref{eq:Price-process-filtering}  can now be written as
\begin{equation*}
    dS^y(t) = D(S^y(t)) \left(\left(\tilde{\lambda}(t,y) + \tilde{\beta}^M(t) M(t) \right)dt + L^S dW^\mathbb{Q}(t) \right)
\end{equation*}
where
\(\tilde{\beta}^M(t) = \begin{pmatrix}
    \tilde{\beta}(t) & I_N
\end{pmatrix} \in \mathbb{R}^{N \times (d+N)}\) and the investor's wealth satisfies 
\begin{equation}
\label{eq:weath_filtering}
    dZ(t) = Z(t) \bigg( r_f dt + \pi(t)^\top \big(\tilde{\lambda}(t,y) + \tilde{\beta}^M(t) M(t) - r_f \mathbf{1}_N\big) dt + \pi(t)^\top L^S dW^\mathbb{Q}(t) \bigg).
\end{equation}

With this augmented formulation, the control problem now mirrors the structure  of the problem solved in Section \ref{sec5} but with \(M(t)\) now playing the role as the factor process. The investor's value function 
\begin{equation*}
    V(t,z,m) = \max_{\pi \in \Pi} \mathbb{E}_{\mathbb{Q}}\bigl[U(Z(T)) \,|\, Z(t) = z, M(t) = m \bigr]
\end{equation*}
with $M(t)$ and $Z(t)$  given by \eqref{eq:M(t)}--\eqref{eq:weath_filtering} satisfies the HJB equation
\begin{align}
\max_{\pi}\Biggl\{\, & \frac{\partial V}{\partial t} + z\Bigl[r_f + \pi^\top\Bigl\{\tilde{\lambda}(t,y) + \tilde{\beta}^M(t) m - r_f\,\mathbf{1}_N\Bigr\}\Bigr]\nabla_z V
+ \Bigl(\tilde{\Theta}^M(t) (\tilde{\mu}^M(t,y) - m)\Bigr)^\top \nabla_m V \notag\\[1mm]
&  + \frac{1}{2} z^2 \pi(t)^\top \Sigma^S \pi(t)  \nabla^2_z V + \frac{1}{2} \Tr(\Sigma^M_t \nabla_m^2 V)  + z \pi(t)^\top \Sigma^{S,M}_t \nabla^2_{m,z} V\Big\} = 0,
\label{eq:HJB_Filtering}
\end{align}
where \(\Sigma^M_t = L^M_t (L^M_t)^\top\) and \(\Sigma^{S,M}_t = L^S (L^M_t)^\top\), with terminal conditional \(
V(T,z,m) = U(z) = \frac{z^{1-\gamma}}{1-\gamma}.
\) 
The value function and optimal policy can now be characterized using the results in Section \ref{sec51}. 

\begin{corollary}
    \label{prop:Policy_KF}
Let $\gamma > 1$. The solution to the HJB equation \eqref{eq:HJB_Filtering} is 
\[V(t, z, m) = \frac{z^{1-\gamma}}{1-\gamma}\exp\left(g(t, m)\right),\]
where the function $g$ is quadratic in the augmented state vector $m = [x^\top, \hat{\alpha}^\top]^\top$
\begin{equation}
\label{eq:g_filtering}
\begin{split}
    g(t,m) &= \frac{1}{2}m^\top A(t) m + m^\top b(t) + c(t)\\
   &=  \frac{1}{2}x^\top A^x(t)x  + \frac{1}{2}\hat{\alpha}^\top A^{\alpha}(t)\hat{\alpha} + x^\top A^{x,\alpha}(t)\hat{\alpha} + x^\top b^x(t) + \hat{\alpha}^\top b^\alpha(t)  + c(t)
    \end{split}
\end{equation}
The symmetric matrix \[A(t) = \begin{pmatrix}
    A^x(t) & A^{x,\alpha}(t)\\ A^{x,\alpha}(t)^\top & A^\alpha(t) 
\end{pmatrix} \in \mathbb{R}^{(d+N)\times(d+N)}\] is negative semi-definite and solves the matrix Riccati equation
\begin{equation*}
\begin{cases}
    A'(t) + \frac{1-\gamma}{\gamma} \tilde{\beta}^M(t)^\top (\Sigma^S)^{-1} \tilde{\beta}^M(t) + A(t) \left(\Sigma_t^M + \frac{1-\gamma}{\gamma} (\Sigma_t^{S,M})^\top (\Sigma^S)^{-1} \Sigma_t^{S,M}\right) A(t) \\
    \quad + A(t) \left(\frac{1-\gamma}{\gamma}(\Sigma_t^{S,M})^\top (\Sigma^S)^{-1} \tilde{\beta}^M(t) - \tilde{\Theta}^M(t)\right) \\
    \quad + \left(\frac{1-\gamma}{\gamma}(\Sigma_t^{S,M})^\top (\Sigma^S)^{-1} \tilde{\beta}^M(t) - \tilde{\Theta}^M(t)\right)^\top A(t) = 0,\\
    A(T) = 0.
\end{cases}
\end{equation*}
 The vector $b(t)  = [b^x(t)^\top, b^\alpha(t)^\top]^\top \in \mathbb{R}^{d+N}$ solves the ODE system
 \begin{equation*}
\begin{cases}
    b'(t) + \frac{1-\gamma}{\gamma} \left( \tilde{\beta}^M(t)^\top + A(t) (\Sigma_t^{S,M})^\top\right) (\Sigma^S)^{-1} \left( \Sigma_t^{S,M} b(t) + \tilde{\lambda}(t,y) - r_f \mathbf{1}_N \right) \\
    \quad + \left(A(t) \Sigma_t^M - \tilde{\Theta}^M(t)^\top\right) b(t) + A(t) \tilde{\mu}^M(t,y) = 0, \\
    b(T) = 0.
\end{cases}
\end{equation*}
 The scalar-valued function $c(t)$ is given in \eqref{eq:EC_c_filtering} in Online Appendix \ref{Appsec:KF}.
The optimal investment policy 
\begin{equation}
\pi^*(t,x,\hat{\alpha}) = \underbrace{\frac{1}{\gamma}(\Sigma^S)^{-1}\left(\hat{\alpha}(t) + \tilde{\lambda}(t,y) + \tilde{\beta}(t)x - r_f\mathbf{1}_N\right)}_{\text{Mean-Variance Holding}} + \underbrace{H
^x(t, x, \hat{\alpha})}_{\text{Factor Hedging}} + \underbrace{H^\alpha(t, x, \hat{\alpha}).}_{\text{Estimation Risk Hedging}}
\label{eq:opt_policy_KF}
\end{equation}
consists of a mean-variance holding and intertemporal hedging demands for shifts in the risk factors $X^y(t)$
$$H^x(t, x, \hat{\alpha}) = \frac{1}{\gamma}(\Sigma^S)^{-1}\Sigma^{S,X}\left(A^x(t)x + b^x(t) + A^{x,\alpha}\hat{\alpha}\right)$$
and changes in the estimate of the drift $\hat{\alpha}(t)$
$$H^\alpha(t, x, \hat{\alpha}) = \frac{1}{\gamma}(\Sigma^S)^{-1}\Gamma(t)\left(A^\alpha(t)\hat{\alpha} + b^\alpha(t) + (A^{x,\alpha})^\top x\right).$$
\end{corollary}

The detailed derivation of the optimal solution and the explicit ODEs are given in Appendix \ref{Appsec:KF}.

Corollary \ref{prop:Policy_KF} shows that the optimal policy now includes a hedging demand \(H_\alpha\) for changes in the posterior mean of \(\alpha\). The innovations in the asset prices that derive the learning process (i.e., the updates to \(\hat{\alpha}\)) also represent a source of risk to the investor's portfolio. The term \(H_\alpha\) is the demand for assets that will perform well when the drift estimate is revised downwards, thus providing a hedge. Notably, because the state variables \(x\) and \(\hat{\alpha}\) are correlated through the common Brownian motion \(W^\mathbb{Q}\), the hedging demands are coupled. The optimal hedge against factor risk \(H_x\) depends on the current estimate \(\hat{\alpha}\), and conversely, the hedge against estimation risk \(H_\alpha\) depends on the current state of the factors \(x\).

 \begin{remark}
    One can apply the decomposition from Theorem \ref{thm:Optimal_policy}
    to this augmented problem to isolate the effect of expert views. 
    
    \end{remark}

\begin{remark}
\label{remark:alpha_mean-reverting}
    The model can be extended so that the \emph{true} \(\alpha\) follows mean-reverting dynamics and may be correlated with the risk factors.  
In that case, it follows from \eqref{eq:M(t)} that the optimal policy in \eqref{eq:opt_policy_KF} retains the same structural form, but the coefficients of the augmented state dynamics in \eqref{eq:M(t)} change to reflect the stochastic evolution~of~\(\alpha\).
\end{remark}




\section{Experiments}
\label{sec7}
In this section, we evaluate the impact of forward-looking  views on portfolio performance. We compare the optimal dynamic policy from Theorem \ref{thm:Optimal_policy} against two  benchmarks.
\begin{itemize}
    \item \textbf{Dynamic (No Views)}: This investor uses the optimal dynamic policy \eqref{eq:pi_0}. Comparing to this benchmark isolates the value of forward-looking information
    
    \item \textbf{Static (Black-Litterman Variant)}: This investor incorporates views using a one-step Bayesian update and rebalances myopically at each period. This follows the classic framework of \cite{Black_1991, Black_1992}. A full derivation is in Appendix \ref{App:BL}\footnote{Note that our implementation of the Black-Litterman approach uses our results for the conditional dynamics derived in Section \ref{sec3}.}. Comparing our main strategy to this benchmark isolates the value of dynamic intertemporal hedging over a static approach.

\end{itemize}

\subsubsection*{Data and Simulation Setup} Our emprical analysis uses five Exchange-Traded Funds (ETFs): SPY (S\&P 500), DBC (commodities), LQD (corporate bonds), VNQ (real estate), and TLT (long-term Treasury bonds). We use their respective dividend yields as predictive signals as they are economically interpretable and have been shown to have good predictive power in equity returns (see, for example, \cite{Xia,Cornell2012}). Parameters for factor dynamics, asset returns, and dividend yield predictors are calibrated using historical data, with details provided in Appendix~\ref{App:Calibration}.

We simulate a one-year investment horizon (\(T = 1\)) with monthly rebalancing.  In each simulation run, we generate factors, asset prices, and views according to \eqref{eq:dX}, \eqref{eq:PricesVec}, and \eqref{eq:Y_view}, respectively. We consider \mbox{\(K = 3\)} forward-looking on the factors realization at \(t = T\). We set the views matrix $P = (p_1, p_2, p_3)^\top \in \mathbb{R}^{K \times N}$ with $p_1 = (1, -1, 0 ,0 ,0)$, $p_2 = (0, 0, 1 ,0, 0)$, and $p_3 = (0, 1, 0, 0 ,-1)$. Views are given about the difference in dividend-yield between SPY and DBC, and between DBC and TLT, and about the dividend-yield of LQD, while no view is given about VNQ.

The confidence in these views is controlled by the parameter \(\tau\) in the view covariance matrix 
\begin{equation}
    \label{eq:Cov_Sim}
    \Omega = \tau \; P \mathbb{V}[X(T) \, | \, X(0)] P^\top
\end{equation}
following  \cite{Meucci2006, He2002}.



\subsection{Views vs. No views: Quantifying the Value of Information}
We first evaluate the direct impact of incorporating expert views into the dynamic model. Following \cite{IDZOREK200717}, we set the baseline confidence parameter in \eqref{eq:Cov_Sim} to \(\tau = 0.05\).

\begin{figure}[ht]
     \centering
     \begin{subfigure}[b]{0.49\textwidth}
         \centering
             \captionsetup{font=footnotesize} 
          \caption{Efficient Frontier}    
\includegraphics[width=\textwidth]{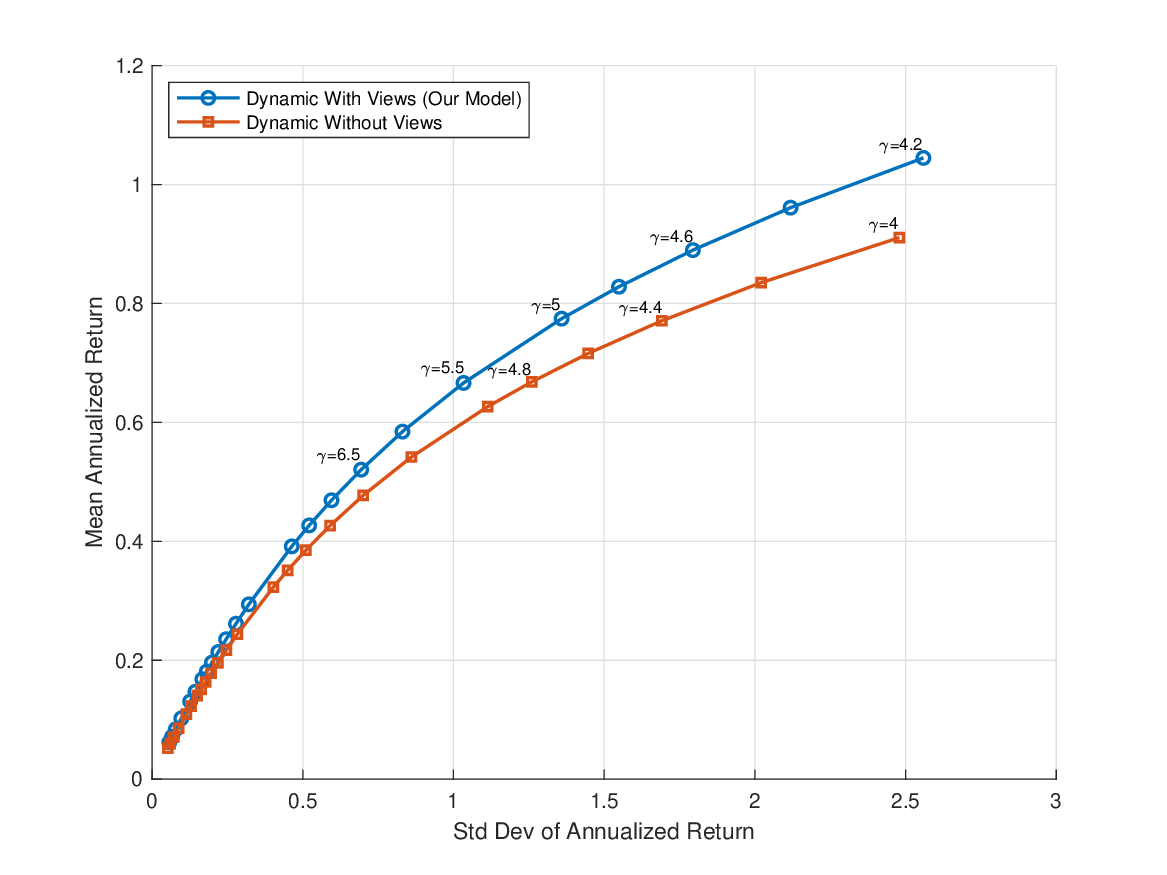}
         \label{fig:EF_CER_a}
     \end{subfigure}
     \begin{subfigure}[b]{0.49\textwidth}
         \centering
         \captionsetup{font=footnotesize}
          \caption{Certainty Equivalent Return (CER)}
         \includegraphics[width=\textwidth]{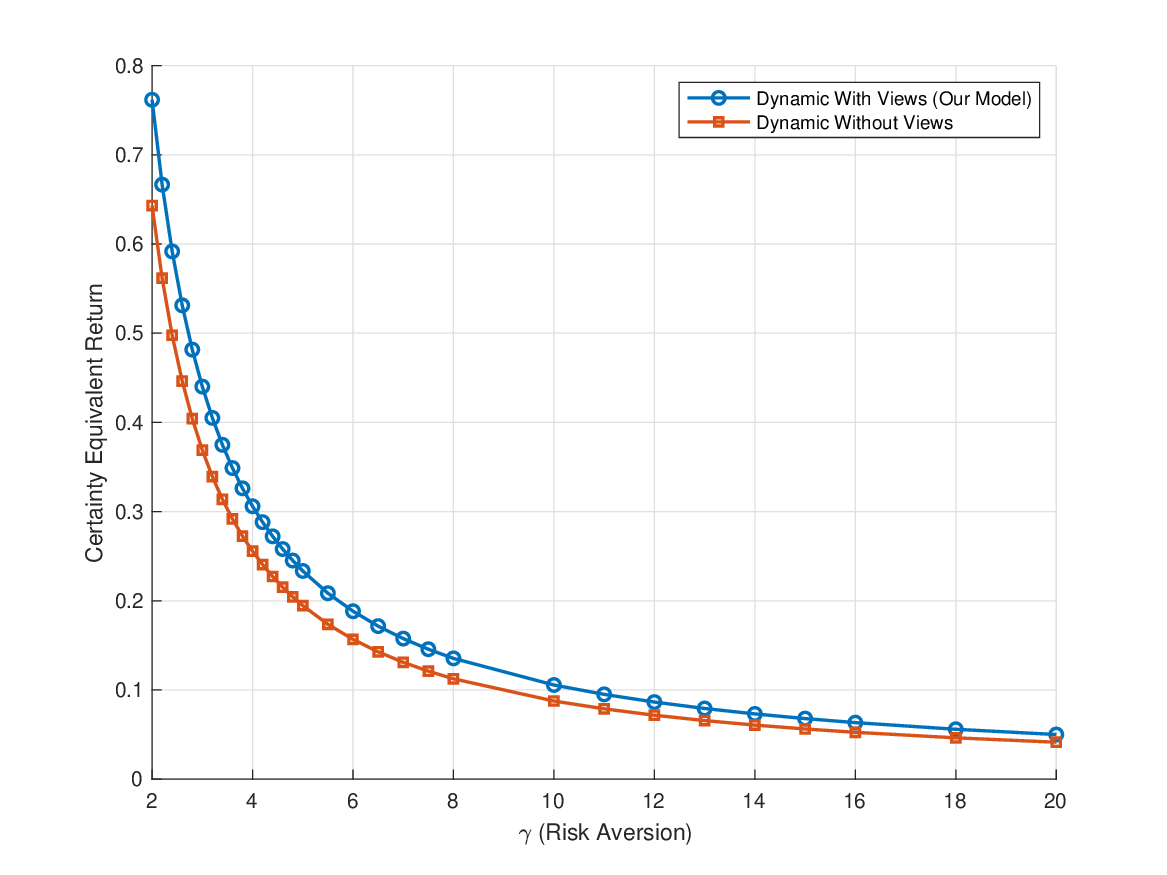}
         \label{fig:EF_CER_b}
     \end{subfigure}
        \caption{\textbf{Performance Comparison: With Views vs. No Views.} This figure illustrates the performance benefits of incorporating experts views into the dynamic model. Panel (a) shows that the efficient frontier for the portfolio with views consistently dominates the one without, leading to a higher expected return for any given level of risk. Panel (b) quantifies this by plotting the Certainty Equivalent Return (CER) against investor risk aversion.}
        \label{fig:EF_CER_dynamic}
\end{figure}

Figure \ref{fig:EF_CER_dynamic} shows that incorporating views yields clear performance benefits. Figure \ref{fig:EF_CER_a} shows that the efficient frontier for the strategy with views strictly dominates the no-views case, resulting in higher returns for any level of risk, with a more clear difference for less risk-averse investors. Figure \ref{fig:EF_CER_b} compares the Certainty Equivalent Return (CER) between the two policies, where the CER is defined as the constant rate of return \(r_c\) that provides  the same utility as the investment strategy. It satisfies 
\[U(z_0 e^{T r_{c}}) = \mathbb E\big[U(Z(T) \, | \, Z(0) = z_0, X(0)\big] \]
where \(U(z) = z^{1-\gamma}/(1-\gamma)\) is the investor's utility. The figure shows that the policy with views has a uniformly higher CER, and the performance gap is more significant for less risk-averse investors highlighting that they benefit most from the aggressive intertemporal hedging that the views enable.




To understand the drivers of this performance gain, we analyze its sensitivity to view precision and market structure.
\begin{figure}[ht]
    \centering
    \includegraphics[width=0.5\linewidth]{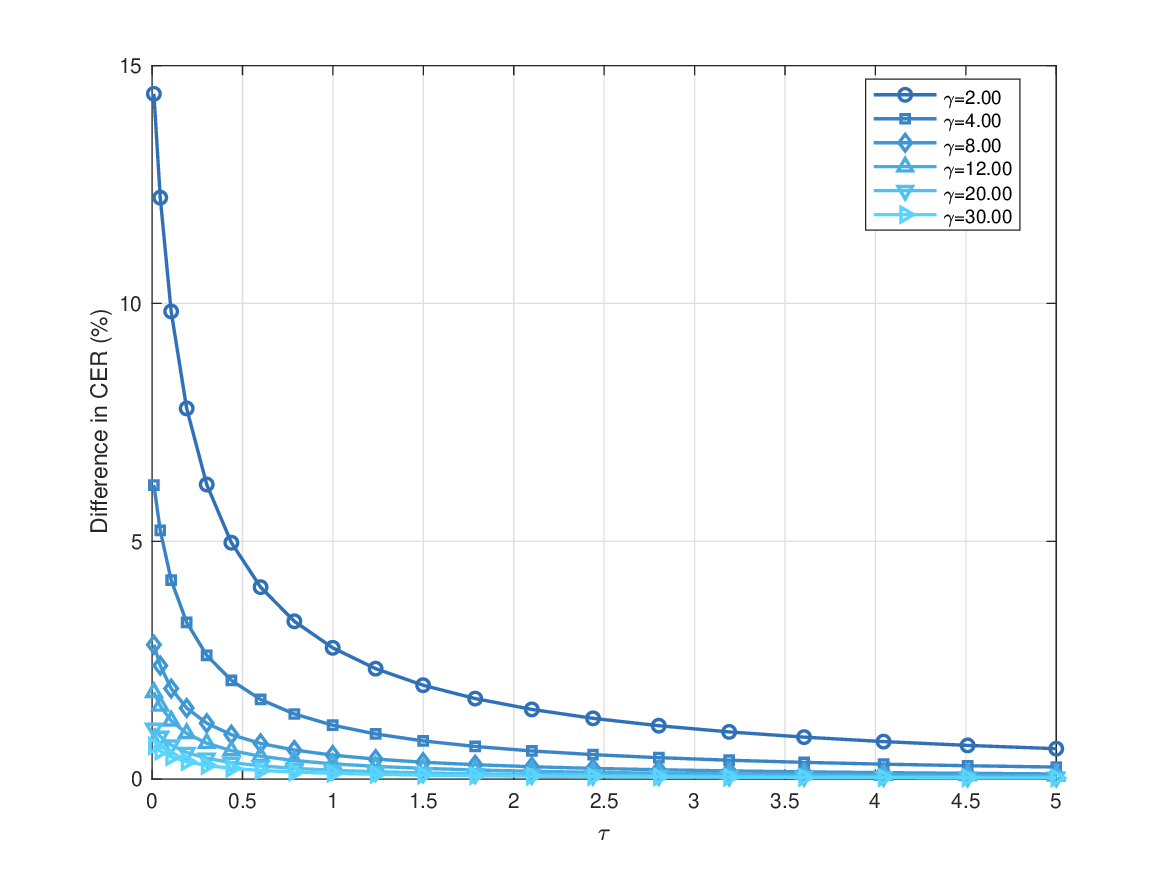}
    \caption{\textbf{Impact of View Precision on Performance Gains.} This figure displays the difference in CER between the dynamic strategies with and without views as a functions of the noise (\(\tau\)) in the views. The performance gain is positive for all noise levels and decreases as the views become less precise (i.e., as \(\tau\) increases).}
    \label{fig:DeltaCER_tau}
\end{figure}

Figure \ref{fig:DeltaCER_tau} compares the difference in CER between the policy with views and the one without. It shows that the benefit of the views is directly related to their precision. The performance gain is positive for all noise levels but decreases as the views become less precise (i.e., as \(\tau\) increases). This is intuitive: as confidence in the views diminishes, the optimal policy converges to the no-views case. Notably, significant performance gains persist even for relatively high levels of noise (\(\tau > 1\)).


To explore how market structure affects this performance gain, we examine the role of the correlation between asset and factor innovations. To do so, we modify the model to explicitly control the correlation between the Brownian motion driving the factors and assets
\begin{equation*}
\begin{split}
     dX(t) &= \Theta(\mu - X(t)) dt + L^X dW^X(t),\\
    dS(t) &= D(S(t)) \big((\alpha + \beta X(t)) dt + L^S dW^S(t)\big)
\end{split}
\end{equation*}
where we set \(dW^S(t) (dW^X(t))^\top = \rho \, dt \,  I_{N'}\). The instantaneous covariance between asset and factor shocks is now \(\rho \, L^S (L^X)^\top\). The parameter \(\rho \in [0,1]\) allows us to analyze the full spectrum from independent shocks \((\rho = 0)\) to the fully correlated shocks of our main model \((\rho = 1)\).


\begin{figure}[ht]
    \centering
    \includegraphics[width=0.5\linewidth]{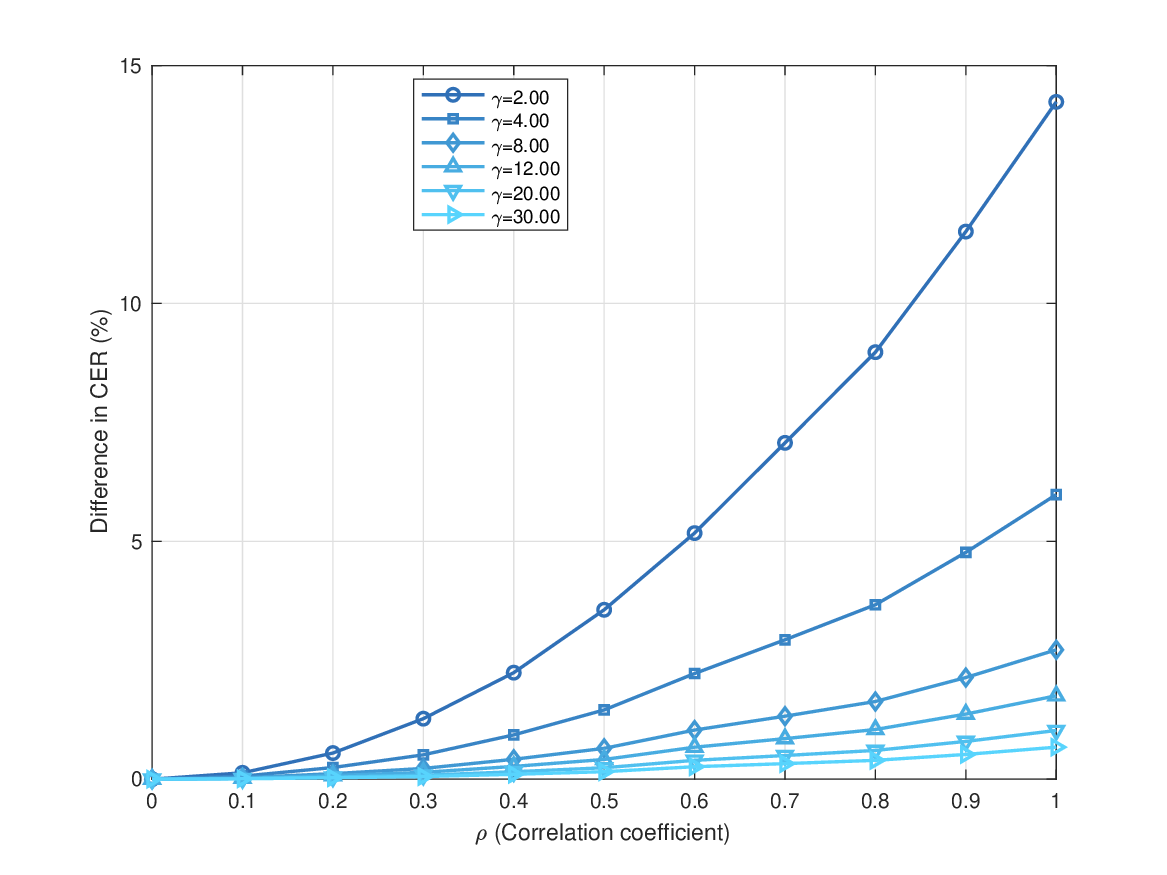}
    \caption{\textbf{Impact of Market Correlation on Performance Gains.} This figure plots the difference in CER as a function of the correlation \(\rho\) between the shocks driving asset prices and risk factors. The performance benefit from incorporating views increases with the correlation. This highlights that forward-looking information on factors is most valuable when the unpredictable shocks in those factors have a strong impact on asset returns.}
    \label{fig:DeltaCER_rho}
\end{figure}

Figure \ref{fig:DeltaCER_rho} shows that the value of information on factors is contingent on their connection to asset prices. The performance gain from views increases with the correlation \(\rho\). As prescribed by the optimal policy in Theorem \ref{thm:Optimal_policy}, if the factors are uncorrelated with assets \((\rho = 0)\), the hedging demand driven by the views disappears, and the performance gain vanishes.


\subsection{Dynamic vs. Static: Quantifying the Value of Intertemporal Hedging}

We now compare the dynamic optimal policies, both with and without views, to the static, periodically rebalanced Black-Litterman variant to isolate the value of the intertemporal hedging component. While both strategies use the same views and rebalancing frequency, the dynamic policy optimizes over the entire investment horizon, whereas the static policy is myopic.

Figure \ref{fig:Static_vs_Dynamic_a} shows that the dynamic strategies deliver a superior risk-return tradeoff. The efficient frontiers for the dynamic policies (both with and without views) strictly dominate the frontier for the static strategy. This shows the significant economic value of the intertemporal hedging demand, which allows the dynamic investor to manage future shifts in the risk factors. The efficiency gain is so large that the dynamic policy without views still outperforms the static policy with views.




\begin{figure}[ht]
     \centering
     \begin{subfigure}[b]{0.49\textwidth}
         \centering
             \captionsetup{font=footnotesize}
          \caption{Efficient Frontier}
 
\includegraphics[width=\textwidth]{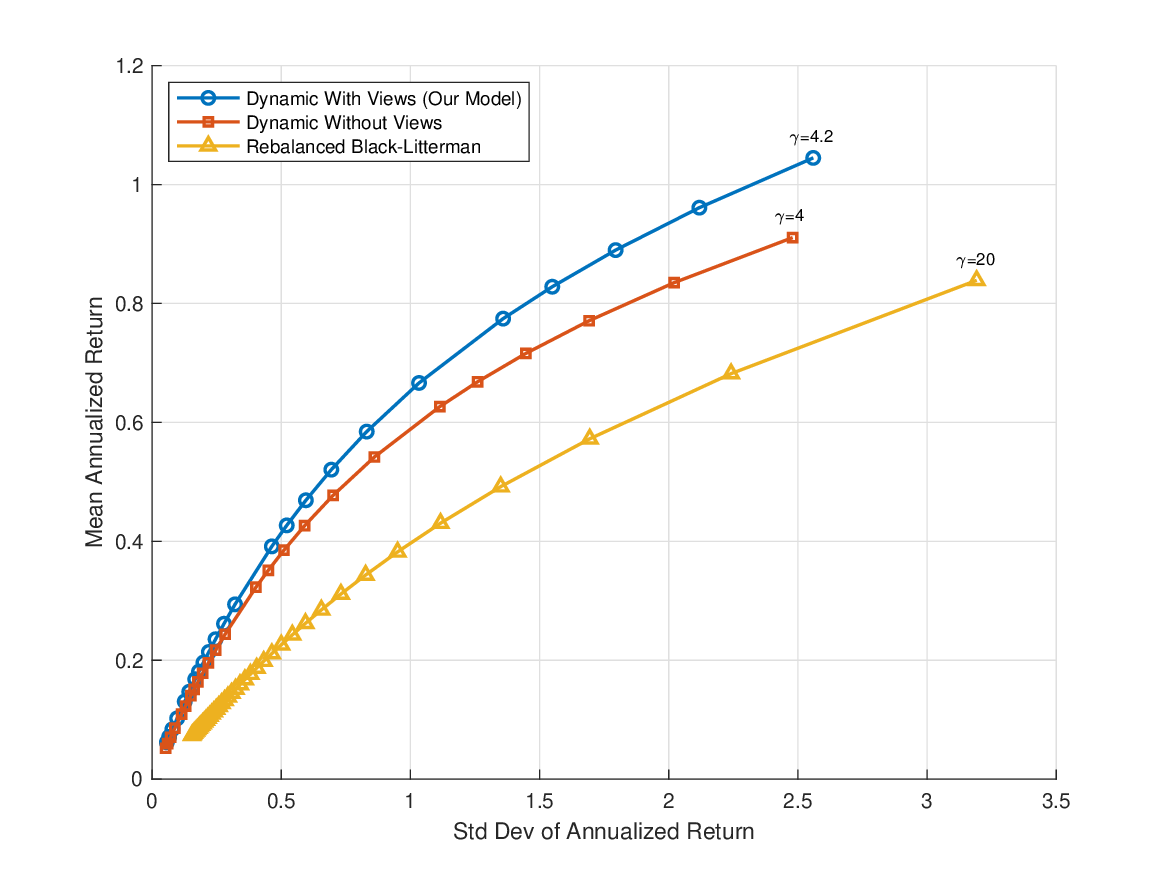}
         \label{fig:Static_vs_Dynamic_a}
     \end{subfigure}
     \begin{subfigure}[b]{0.49\textwidth}
         \centering
         \captionsetup{font=footnotesize}
          \caption{Turnover}
         \includegraphics[width=\textwidth]{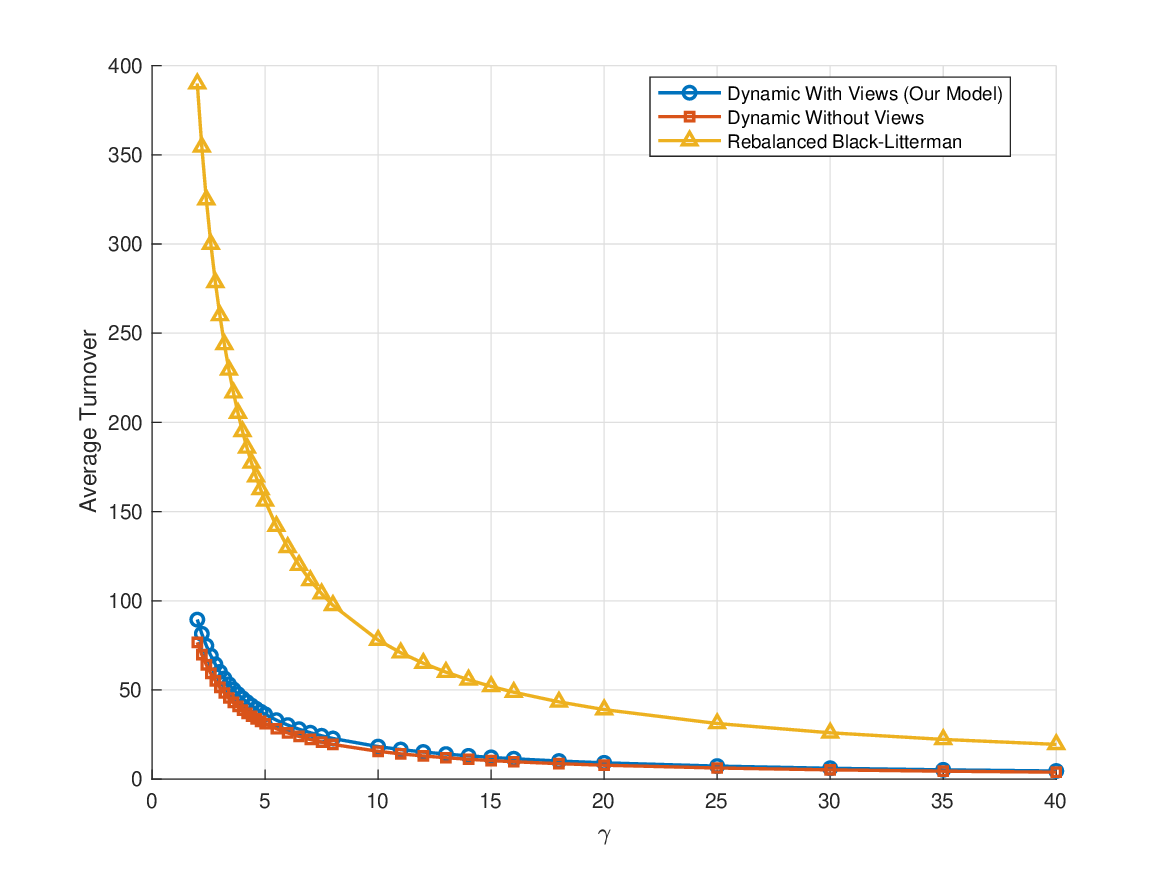}
         \label{fig:Static_vs_Dynamic_b}
     \end{subfigure}
        \caption{\textbf{Performance and Turnover Comparison: Dynamic vs. Static Strategies.} This figure compares the performance and trading activity of the proposed dynamic strategies against a static Black-Litterman variant. Panel (a) shows that the efficient frontier of the dynamic strategies (both with and without views) dominate that of the static approach, which is a result of the dynamic strategies' intertemporal hedging component. Panel (b) reveals that the dynamic strategies achieve this superior performance with significantly lower portfolio turnover, as the hedging component stabilizes the portfolio and reduces the need for frequent, aggressive rebalancing.}
        \label{fig:Static_vs_Dynamic}
\end{figure}

We now analyze the trading volume of these strategies by computing their turnover. Let \( n_i(t) = \pi_i(t) \frac{Z(t)}{S_i(t)}\)    
be the number of shares of asset $i \in [N]$ held at time $t$. We define turnover as the expected total absolute change in holdings during the investment horizon 
\[
\text{Turnover} = \sum_{i \in [N]} \mathbb{E}\left[ \sum_{t \in \mathcal{T}} \left| n_i(t+\Delta t) - n_i(t) \right| \right].
\]
A higher turnover indicates higher trading costs.

Figure \ref{fig:Static_vs_Dynamic_b} plots turnover against the risk aversion parameter \(\gamma\). When compared at the same level of \(\gamma\), the dynamic strategies exhibit dramatically lower turnover. This is a direct consequence of the hedging demand; by hedging against future factor movements, the dynamic portfolio remains more stable and requires fewer, less aggressive adjustments. The myopic static investor, by contrast, must react sharply to every new piece of information at each rebalancing date, incurring higher trading volume.

However, comparing turnover for the same \(\gamma\) is not an apples-to-apples comparison, as the resulting portfolios have different risk profiles. A more meaningful comparison is between portfolios with the same realized risk (i.e., the same standard deviation on the x-axis of Panel (a)). When viewed on this risk-adjusted basis, the turnover for the dynamic strategy is of a similar magnitude to the static one. However, for the same level of risk and turnover, the dynamic strategies achieve higher expected returns.

This is a notable result. Given that the dynamic policy maintains an active intertemporal hedge against future factor movements, one might expect it to require higher trading volumes. However, Figure \ref{fig:Static_vs_Dynamic} shows that it achieves a dominant risk-return profile without incurring additional trading costs.

Finally, consider the two dynamic variants. Although the policy with views hedges more aggressively than the one without views, Figure \ref{fig:Static_vs_Dynamic} shows their turnover levels are nearly identical, with the strategy with views having a better risk-return profile.

\section{Conclusion}
\label{sec8}
This paper addresses the challenge of integrating forward-looking expert forecasts into continuous-time dynamic factor models. We develop a continuous-time framework for modeling expert views and deriving the conditional dynamics of factors and assets based on these views. Our primarily theoretical contribution is deriving the complete conditional dynamics for both factors and asset prices. To offer deeper intuition into these conditional dynamics, we introduce and characterize the Mean-reverting Bridge (MrB), a stochastic process that generalizes the classical Brownian bridge to mean-reverting process. Additionally, we show that when factor and asset price innovations are correlated, expert views alter not only factor dynamics but the asset price process as well through a new time-dependent covariate that depends on the view.

We apply our framework to dynamic portfolio allocation. We derive the optimal investment policy and show that forward-looking views influence the optimal allocation strategy exclusively through the intertemporal hedging demand. Investors with views hedge more aggressively compared to those without. We further extend the core model to accommodate drift uncertainty, where decision-makers must infer drift parameters from observed data. Beyond the particular application considered, the developed framework provides a direct approach for modeling expert views and integrating them into dynamic factor models, with applications beyond portfolio allocation.

\section*{Acknowledgements}

Anas Abdelhakmi is partially supported by `Fondation Ibn Rochd Pour Les Sciences et L'Innovation (FIRSI)', under its 2021 SEEDS Initiative. Andrew Lim is supported by the  Ministry of Education, Singapore under its 2024 Academic Research Fund Tier 2 grant call (Award ref: MOE-T2EP20224-0018).

\bibliographystyle{plainnat} 
\bibliography{reference}

\begin{thebibliography}{36}
\providecommand{\natexlab}[1]{#1}
\providecommand{\url}[1]{\texttt{#1}}
\expandafter\ifx\csname urlstyle\endcsname\relax
  \providecommand{\doi}[1]{doi: #1}\else
  \providecommand{\doi}{doi: \begingroup \urlstyle{rm}\Url}\fi

\bibitem[Abdelhakmi and Lim(2025)]{Abdelhakmi2025_BL}
Anas Abdelhakmi and Andrew Lim.
\newblock Dynamic {B}lack-{L}itterman.
\newblock \emph{Operations Research}, 2025.
\newblock URL \url{https://doi.org/10.1287/opre.2024.1010}.
\newblock Articles in Advance.

\bibitem[Abou-Kandil et~al.(2003)Abou-Kandil, Freiling, Ionescu, and Jank]{Ricatti}
Hisham Abou-Kandil, Gerhard Freiling, Vlad Ionescu, and Gerhard Jank.
\newblock \emph{Hermitian Riccati differential equations}, pages 181--255.
\newblock Birkh{\"a}user Basel, Basel, 2003.
\newblock ISBN 978-3-0348-8081-7.
\newblock \doi{10.1007/978-3-0348-8081-7_4}.
\newblock URL \url{https://doi.org/10.1007/978-3-0348-8081-7_4}.

\bibitem[Aksamit and Jeanblanc(2017)]{JeanBlanc}
Anna Aksamit and Monique Jeanblanc.
\newblock \emph{Enlargement of Filtration with Finance in View}.
\newblock Springer Cham, 01 2017.
\newblock ISBN 978-3-319-41254-2.
\newblock \doi{10.1007/978-3-319-41255-9}.

\bibitem[Bertsimas et~al.(2012)Bertsimas, Gupta, and Paschalidis]{BertsimasBL}
Dimitris Bertsimas, Vishal Gupta, and Ioannis~Ch. Paschalidis.
\newblock Inverse {O}ptimization: {A} {N}ew {P}erspective on the {B}lack-{L}itterman {M}odel.
\newblock \emph{Operations Research}, 60\penalty0 (6):\penalty0 1389--1403, 2012.
\newblock \doi{10.1287/opre.1120.1115}.
\newblock URL \url{https://doi.org/10.1287/opre.1120.1115}.

\bibitem[Black and Litterman(1992)]{Black_1992}
Fischer Black and Robert Litterman.
\newblock Global portfolio optimization.
\newblock \emph{Financial Analysts Journal}, 48\penalty0 (5):\penalty0 28--43, 1992.
\newblock ISSN 0015198X.
\newblock URL \url{http://www.jstor.org/stable/4479577}.

\bibitem[Black and Litterman(1991)]{Black_1991}
Fischer Black and Robert~B Litterman.
\newblock Asset allocation.
\newblock \emph{The Journal of Fixed Income}, 1\penalty0 (2):\penalty0 7--18, sep 1991.
\newblock \doi{10.3905/jfi.1991.408013}.
\newblock URL \url{https://doi.org/10.3905\%2Fjfi.1991.408013}.

\bibitem[Campbell and Shiller(1988)]{Campbell1988}
John~Y. Campbell and Robert~J. Shiller.
\newblock The dividend-price ratio and expectations of future dividends and discount factors.
\newblock \emph{The Review of Financial Studies}, 1\penalty0 (3):\penalty0 195--228, 1988.
\newblock ISSN 08939454, 14657368.
\newblock URL \url{http://www.jstor.org/stable/2961997}.

\bibitem[Chacko and Viceira(2005)]{ChackoViceira2005}
George Chacko and Luis~M. Viceira.
\newblock Dynamic consumption and portfolio choice with stochastic volatility in incomplete markets.
\newblock \emph{The Review of Financial Studies}, 18\penalty0 (4):\penalty0 1369--1402, 2005.

\bibitem[Chen and Lim(2020)]{Andrew2020}
Shea Chen and Andrew Lim.
\newblock A {G}eneralized {B}lack–{L}itterman {M}odel.
\newblock \emph{Operations Research}, 01 2020.
\newblock \doi{10.1287/opre.2019.1893}.

\bibitem[Cochrane(2007)]{Chochrane2008}
John~H. Cochrane.
\newblock The {D}og {T}hat {D}id {N}ot {B}ark: A {D}efense of {R}eturn {P}redictability.
\newblock \emph{The Review of Financial Studies}, 21\penalty0 (4):\penalty0 1533--1575, 09 2007.
\newblock ISSN 0893-9454.
\newblock \doi{10.1093/rfs/hhm046}.
\newblock URL \url{https://doi.org/10.1093/rfs/hhm046}.

\bibitem[Corlay(2013)]{OUB2013}
Sylvain Corlay.
\newblock {Properties of the Ornstein-Uhlenbeck bridge}.
\newblock October 2013.
\newblock URL \url{https://hal.science/hal-00875342}.

\bibitem[Cornell(2012)]{Cornell2012}
Bradford Cornell.
\newblock Dividend-price ratios and stock returns: Another look at the history.
\newblock \emph{The Journal of Investing}, 22, 05 2012.
\newblock \doi{10.2139/ssrn.2071014}.

\bibitem[Davis and Lleo(2013)]{Davis2013}
Mark Davis and Sebastien Lleo.
\newblock Black-{L}itterman in {C}ontinuous {T}ime: The {C}ase for {F}iltering.
\newblock \emph{Quantitative Finance Letters}, 1, 04 2013.
\newblock \doi{10.1080/21649502.2013.803794}.

\bibitem[Davis and Lleo(2020)]{Davis2020}
Mark Davis and Sebastien Lleo.
\newblock Debiased {E}xpert {F}orecasts in {C}ontinuous-{T}ime {A}sset {A}llocation.
\newblock \emph{Journal of Banking and Finance}, 113, 04 2020.
\newblock URL \url{https://doi.org/10.1016/j.jbankfin.2020.105759}.

\bibitem[Davis and Lleo(2021)]{Davis2021}
Mark Davis and Sebastien Lleo.
\newblock Risk‐sensitive benchmarked asset management with expert forecasts.
\newblock \emph{Mathematical Finance}, 31(4):\penalty0 1162--1189, 2021.
\newblock URL \url{https://doi.org/10.1111/mafi.12310}.

\bibitem[Davis and Lleo(2022)]{Davis2022}
Mark Davis and Sebastien Lleo.
\newblock Jump-{D}iffusion {R}isk-{S}ensitive {B}enchmarked {A}sset {M}anagement with {T}raditional and {A}lternative {D}ata.
\newblock \emph{Annals of Operations Research}, 2022.
\newblock URL \url{https://doi.org/10.1007/s10479-022-05130-3}.

\bibitem[Durrett(1996)]{Durret}
Richard Durrett.
\newblock \emph{Stochastic Calculus: A Practical Introduction}.
\newblock CRC Press, Boca Raton, 1st edition, 1996.

\bibitem[Fernique(1970)]{Fernique1970}
Xavier Fernique.
\newblock Int{\'e}grabilit{\'e} des vecteurs {G}aussiens.
\newblock \emph{Comptes Rendus de l'Acad{\'e}mie des Sciences.}, 270:\penalty0 1698--1699, 1970.

\bibitem[Frey et~al.(2012)Frey, Gabih, and Wunderlich]{Frey2012}
Rudiger Frey, Abdelali Gabih, and Ralf Wunderlich.
\newblock Portfolio optimization under partial information with expert opinions.
\newblock \emph{International Journal of Theoretical and Applied Finance}, 15(1), 2012.
\newblock URL \url{https://doi.org/10.1142/S0219024911006486}.

\bibitem[Gabih et~al.(2024)Gabih, Kondakji, and Wunderlich]{GabihKondakjiWunderlich2024}
Abdelali Gabih, Hakam Kondakji, and Ralf Wunderlich.
\newblock Power utility maximization with expert opinions at fixed arrival times in a market with hidden gaussian drift.
\newblock \emph{Annals of Operations Research}, 341:\penalty0 897--936, 2024.
\newblock \doi{10.1007/s10479-024-06172-5}.
\newblock URL \url{https://doi.org/10.1007/s10479-024-06172-5}.

\bibitem[Gasbarra et~al.(2007)Gasbarra, Sottinen, and Valkeila]{GaussianBridges}
Dario Gasbarra, Tommi Sottinen, and Esko Valkeila.
\newblock Gaussian bridges.
\newblock In Fred~Espen Benth, Giulia Di~Nunno, Tom Lindstr{\o}m, Bernt {\O}ksendal, and Tusheng Zhang, editors, \emph{Stochastic Analysis and Applications}, pages 361--382. Springer Berlin Heidelberg, Berlin, Heidelberg, 2007.

\bibitem[Goldys and Maslowski(2008)]{OUB2008}
B.~Goldys and B.~Maslowski.
\newblock The {O}rnstein–{U}hlenbeck bridge and applications to {M}arkov semigroups.
\newblock \emph{Stochastic Processes and their Applications}, 118\penalty0 (10):\penalty0 1738--1767, 2008.
\newblock ISSN 0304-4149.
\newblock \doi{https://doi.org/10.1016/j.spa.2007.10.010}.
\newblock URL \url{https://www.sciencedirect.com/science/article/pii/S0304414907001822}.

\bibitem[He and Litterman(2002)]{He2002}
Guangliang He and Robert Litterman.
\newblock The {I}ntuition {B}ehind {B}lack-{L}itterman {M}odel {P}ortfolios.
\newblock \emph{Investment Management Research}, 10 2002.
\newblock \doi{10.2139/ssrn.334304}.

\bibitem[Idzorek(2007)]{IDZOREK200717}
Thomas Idzorek.
\newblock A step-by-step guide to the {B}lack-{L}itterman model: Incorporating user-specified confidence levels.
\newblock In Stephen Satchell, editor, \emph{Forecasting Expected Returns in the Financial Markets}, pages 17--38. Academic Press, Oxford, 2007.
\newblock URL \url{https://www.sciencedirect.com/science/article/pii/B9780750683210500030}.

\bibitem[Ledoit and Wolf(2020)]{Omega_estimation}
Olivier Ledoit and Michael Wolf.
\newblock {The Power of (Non-)Linear Shrinking: A Review and Guide to Covariance Matrix Estimation}.
\newblock \emph{Journal of Financial Econometrics}, 20\penalty0 (1):\penalty0 187--218, 06 2020.
\newblock ISSN 1479-8409.
\newblock \doi{10.1093/jjfinec/nbaa007}.
\newblock URL \url{https://doi.org/10.1093/jjfinec/nbaa007}.

\bibitem[Mazzolo(2017)]{OUB2017}
Alain Mazzolo.
\newblock Constraint {O}rnstein-{U}hlenbeck bridges.
\newblock \emph{Journal of Mathematical Physics}, 58, 04 2017.
\newblock \doi{10.1063/1.5000077}.

\bibitem[Merton(1971)]{MERTON1971373}
Robert~C Merton.
\newblock Optimum consumption and portfolio rules in a continuous-time model.
\newblock \emph{Journal of Economic Theory}, 3\penalty0 (4):\penalty0 373--413, 1971.
\newblock ISSN 0022-0531.
\newblock \doi{https://doi.org/10.1016/0022-0531(71)90038-X}.
\newblock URL \url{https://www.sciencedirect.com/science/article/pii/002205317190038X}.

\bibitem[Meucci(2006)]{Meucci2006}
Attilio Meucci.
\newblock Beyond {Black-Litterman} in {P}ractice: A {F}ive-{S}tep {R}ecipe to {I}nput {V}iews on {N}on-{N}ormal {M}arkets.
\newblock \emph{SSRN Electronic Journal}, 19, 05 2006.
\newblock \doi{10.2139/ssrn.872577}.

\bibitem[Nielsen et~al.(2000)Nielsen, Madsen, and Young]{NIELSEN200083}
Jan~Nygaard Nielsen, Henrik Madsen, and Peter~C. Young.
\newblock Parameter estimation in stochastic differential equations: An overview.
\newblock \emph{Annual Reviews in Control}, 24:\penalty0 83--94, 2000.
\newblock ISSN 1367-5788.
\newblock \doi{https://doi.org/10.1016/S1367-5788(00)90017-8}.
\newblock URL \url{https://www.sciencedirect.com/science/article/pii/S1367578800900178}.

\bibitem[Pastor and Stambaugh(2009)]{PastorStambaugh2009}
Lubos Pastor and Robert~F. Stambaugh.
\newblock Predictive system: Living with imperfect predictors.
\newblock \emph{Journal of Finance}, 64:\penalty0 1583--1628, 2009.

\bibitem[Pinsky and Karlin(2011)]{BrownianBridgesPINSKY}
Mark~A. Pinsky and Samuel Karlin.
\newblock {B}rownian {M}otion and {R}elated {P}rocesses.
\newblock In Mark~A. Pinsky and Samuel Karlin, editors, \emph{An {I}ntroduction to {S}tochastic {M}odeling}, pages 391--446. Academic Press, Boston, fourth edition, 2011.
\newblock URL \url{https://www.sciencedirect.com/science/article/pii/B9780123814166000083}.

\bibitem[Sass et~al.(2017)Sass, Westphal, and Wunderlich]{Sass2017}
Jorn Sass, Dorothee Westphal, and Ralf Wunderlich.
\newblock Expert opinions and logarithmic utility maximization in a market with {G}aussian drift.
\newblock \emph{Communications on Stochastic Analysis}, 8(1):\penalty0 27--47, 2017.
\newblock URL \url{https://doi.org/10.1142/S0219024917500224}.

\bibitem[Sass et~al.(2023)Sass, Westphal, and Wunderlich]{Sass03042023}
Jörn Sass, Dorothee Westphal, and Ralf Wunderlich.
\newblock Diffusion approximations for periodically arriving expert opinions in a financial market with gaussian drift.
\newblock \emph{Stochastic Models}, 39\penalty0 (2):\penalty0 323--362, 2023.
\newblock \doi{10.1080/15326349.2022.2100423}.
\newblock URL \url{https://doi.org/10.1080/15326349.2022.2100423}.

\bibitem[Walters(2013)]{walters2013}
Jay Walters.
\newblock The {F}actor {T}au in the {B}lack-{L}itterman {M}odel.
\newblock \emph{SSRN}, 2013.
\newblock URL \url{http://dx.doi.org/10.2139/ssrn.1701467}.

\bibitem[Xia(2001)]{Xia}
Yihong Xia.
\newblock Learning about predictability: The effects of parameter uncertainty on dynamic asset allocation.
\newblock \emph{The Journal of Finance}, 56\penalty0 (1):\penalty0 205--246, 2001.
\newblock ISSN 00221082, 15406261.
\newblock URL \url{http://www.jstor.org/stable/222467}.

\bibitem[Øksendal(2003)]{oksendal2003}
Bernt Øksendal.
\newblock \emph{Stochastic Differential Equations: An Introduction with Applications}.
\newblock Springer, Berlin, Heidelberg, 6th edition, 2003.
\newblock Chapter 3.

\end{thebibliography}

\newpage

\begin{appendices}
\renewcommand{\theequation}{EC.\arabic{equation}}
\setcounter{equation}{0}

\setcounter{page}{1}
\renewcommand{\thepage}{ec\arabic{page}}

\section{Experiments Calibration}
\label{App:Calibration}

For the empirical analysis, we select five ETFs: a broad equity market ETF (SPY), a commodity ETF (DBC), a corporate bond ETF (LQD), a real estate ETF (VNQ), and a long-term treasury bond ETF (TLT). The risk factors, $X(t)$, are the corresponding trailing 12-month dividend yields for each of these ETFs.

The parameters for the continuous-time model described in Section \ref{sec21} are calibrated using monthly historical data from January 2005 to December 2019 ($\Delta t = 1/12$), which was retrieved from Yahoo finance. The estimation procedure is detailed below.

\subsection{Asset Price Dynamics $(\alpha, \beta)$}
The SDE for the price of asset $i$, as given by \eqref{eq:PricesVec}, is
\begin{equation*}
 \frac{dS_i(t)}{S_i(t)} = \left(\alpha_i + \sum_{j=1}^{d} \beta_{ij} X_{j}(t)\right) dt + L^S_i dW(t)
\end{equation*}
We assume the matrix $\beta$ is diagonal as economic evidence shows an asset's own valuation ratio (e.g., dividend yield) is the strongest predictor of its returns, while cross-asset predictability is weak and unstable (see, for example \cite{Campbell1988,Chochrane2008}). This simplifies the model without losing explanatory power, as off-diagonal terms were statistically insignificant in our preliminary tests.

Using It\^o's lemma, the corresponding dynamics for the log-price are
\begin{equation*}
 d(\log S_i(t)) = \left(\alpha_i + \sum_{j=1}^{d} \beta_{ij} X_{j}(t) - \frac{1}{2}\Sigma^S_{ii}\right) dt + L^S_i dW(t)
\end{equation*}
where $\Sigma^S_{ii} = (L^S (L^S)^T)_{ii}$ is the $i$-th diagonal element of the asset covariance matrix. A first-order discretization of this process gives the following linear relationship for the log-returns
\begin{equation*}
 \log\left(\frac{S_{i,t+\Delta t}}{S_{i,t}}\right) \approx \left(\alpha_i - \frac{1}{2}\Sigma^S_{ii}\right)\Delta t + \left(\sum_{j=1}^{d} \beta_{ij} X_{j,t}\right) \Delta t + \epsilon_{S_i, t+\Delta t}
\end{equation*}
For each asset $i=1, \dots, 5$, we estimate the parameters by running a regression on the discrete-time model
\begin{equation*}
 \log\left(\frac{S_{i,t+\Delta t}}{S_{i,t}}\right) = c_i + b_{ii} X_{i,t} + \epsilon_{S_i, t+\Delta t}
\end{equation*}
While we considered regularized methods such as Ridge and Lasso, Ordinary Least Squares (OLS) regression yielded more stable results. Given that each regression involves only a single predictor variable, the complexity that Ridge and Lasso are designed to manage is not present, making OLS the most direct and effective estimation method. The continuous-time parameters are then recovered from the estimated regression coefficients $\hat{c}_i$ and $\hat{b}_{ii}$
\begin{align*}
 \hat{\beta}_{ii} &= \frac{\hat{b}_{ii}}{\Delta t} \\
 \hat{\alpha}_i &= \frac{\hat{c}_i}{\Delta t} + \frac{1}{2}\hat{\Sigma}^S_{ii}
\end{align*}
where $\hat{\Sigma}^S_{ii}$ is the annualized sample variance of the asset's log-returns. The calibrated parameters are
\begin{eqnarray*}
\begin{split}
    \alpha &= \begin{pmatrix} -0.2044 & -0.0320 & -0.0589 & -0.1824 & -0.1501 \end{pmatrix}^\top, \\
\beta &= \mathrm{diag}(14.1731, 3.9467, 1.7666, 5.2718, 5.6344).
\end{split}
\end{eqnarray*}

\subsection{Factor Dynamics $(\Theta, \mu)$}
The factors are the dividend yields, which we model as Ornstein-Uhlenbeck processes. Assuming a diagonal mean-reversion matrix $\Theta$, the SDE for each factor $j$ is
\begin{equation*}
 dX_j(t) = \Theta_{jj}(\mu_j - X_j(t))dt + L^X_j dW(t)
\end{equation*}
We assume a diagonal mean-reversion matrix $\Theta$ because each factor is constructed to represent a distinct economic signal (e.g., equity, commodity, credit yields). Critically, specifying a diagonal $\Theta$ does not imply that the factors are uncorrelated; rather, the significant cross-sectional dependence between factors is captured more directly through the off-diagonal elements of the joint diffusion covariance matrix, $\hat{\Sigma}^X$.

The exact discrete-time solution to this SDE is a vector AR(1) process
\begin{equation*}
 X_{j,t+\Delta t} = \mu_j(1 - e^{-\Theta_{jj}\Delta t}) + e^{-\Theta_{jj}\Delta t} X_{j,t} + \epsilon_{X_j, t+\Delta t}
\end{equation*}
We estimate this by fitting the AR(1) model $X_{j,t+\Delta t} = a_j + b_j X_{j,t} + \epsilon_{X_j, t+\Delta t}$ via OLS for each factor $j=1,\dots,5$. The continuous-time parameters are recovered from the estimated AR(1) coefficients $\hat{a}_j$ and $\hat{b}_j$ as follows
\begin{align*}
 \hat{\Theta}_{jj} &= -\frac{\log(\hat{b}_j)}{\Delta t} \\
 \hat{\mu}_j &= \frac{\hat{a}_j}{1 - \hat{b}_j}
\end{align*}
The calibrated parameters for the factor dynamics are
\begin{eqnarray*}
\begin{split}
\mu &= \begin{pmatrix} 0.0200 & 0.0068 & 0.0265 & 0.0429 & 0.0235 \end{pmatrix}^\top,\\
\Theta &= \mathrm{diag}(0.7412, 0.6080, 0.0677, 0.7872, 0.1751).
\end{split}
\end{eqnarray*}

\subsection{Diffusion and Covariance Structure $(L^X, L^S)$}
To ensure the model reproduces the complete correlation structure of the market, we calibrate the diffusion matrices $L^X$ and $L^S$ jointly from the data. 
Let $S_t$ be the vector of log-returns at time $t$ and $X_t$ be the vector of dividend-yield factors. We first compute their sample means, $\bar{S}$ and $\bar{X}$. Then, the monthly sample covariance matrices are calculated as
\begin{align*}
\hat{\Sigma}^S_{\text{monthly}} &= \frac{1}{T-1} \sum_{t=1}^{T} (S_t - \bar{S})(S_t - \bar{S})^T \\
\hat{\Sigma}^X_{\text{monthly}} &= \frac{1}{T-1} \sum_{t=1}^{T} (X_t - \bar{X})(X_t - \bar{X})^T \\
\hat{\Sigma}^{XS}_{\text{monthly}} &= \frac{1}{T-1} \sum_{t=1}^{T} (X_t - \bar{X})(S_t - \bar{S})^T
\end{align*}
These estimates are then annualized by dividing by the time step $\Delta t = 1/12$
\begin{equation*}
\hat{\Sigma}^S = \frac{\hat{\Sigma}^S_{\text{monthly}}}{\Delta t}, \quad \hat{\Sigma}^X = \frac{\hat{\Sigma}^X_{\text{monthly}}}{\Delta t}, \quad \hat{\Sigma}^{XS} = \frac{\hat{\Sigma}^{XS}_{\text{monthly}}}{\Delta t}
\end{equation*}
With these annualized components, we construct the full $(d+N) \times (d+N)$ joint empirical covariance matrix
\begin{equation*}
 \hat{\Sigma} =
 \begin{pmatrix}
 \hat{\Sigma}^X & \hat{\Sigma}^{XS} \\
 (\hat{\Sigma}^{XS})^T & \hat{\Sigma}^S
 \end{pmatrix}
\end{equation*}
This joint matrix fully captures the empirically observed variance and covariance structure between all asset returns and factors. The joint diffusion matrix for the entire system, $\hat{L} \in \mathbb{R}^{(d+N) \times (d+N)}$, is then calculated by taking the Cholesky decomposition of this empirical covariance matrix
\begin{equation*}
 \hat{L} \hat{L}^T = \hat{\Sigma}
\end{equation*}
This procedure determines both the structure of $\hat{L}$ and the required dimension of the common Brownian motion, $N' = d+N=10$. The resulting matrix $\hat{L}$ is then partitioned to obtain the final calibrated diffusion matrices
\begin{equation*}
 \hat{L} =
 \begin{pmatrix}
 \hat{L}^X \\
 \hat{L}^S
 \end{pmatrix}
\end{equation*}
where $\hat{L}^X$ consists of the first $d=5$ rows of $\hat{L}$, and $\hat{L}^S$ consists of the final $N=5$ rows. 
We find 


\begin{align*}
\hat L^{X} &=
\begin{pmatrix}
 1.05\times10^{-2} & 0 & 0 & 0 & 0 & 0 & 0 & 0 & 0 & 0 \\
 8.52\times10^{-3} & 2.92\times10^{-2} & 0 & 0 & 0 & 0 & 0 & 0 & 0 & 0 \\
 1.07\times10^{-2} & 1.73\times10^{-2} & 2.39\times10^{-2} & 0 & 0 & 0 & 0 & 0 & 0 & 0 \\
 4.08\times10^{-2} & 1.29\times10^{-2} & 4.94\times10^{-3} & 2.64\times10^{-2} & 0 & 0 & 0 & 0 & 0 & 0 \\
 8.31\times10^{-4} & 1.63\times10^{-2} & 2.36\times10^{-2} & 3.63\times10^{-3} & 7.60\times10^{-3} & 0 & 0 & 0 & 0 & 0
\end{pmatrix},
\\[8pt]
\hat L^{S} &= \resizebox{\matrixwidth}{!}{$\displaystyle
\begin{pmatrix}
  9.31\times10^{-3} & -2.65\times10^{-2} & -1.38\times10^{-2} &  7.39\times10^{-3} & -2.03\times10^{-3} &  1.41\times10^{-1} & 0 & 0 & 0 & 0 \\
  2.99\times10^{-3} &  8.44\times10^{-3} & -2.70\times10^{-3} &  1.74\times10^{-2} & -2.90\times10^{-2} &  1.04\times10^{-1} & 1.56\times10^{-1} & 0 & 0 & 0 \\
  4.22\times10^{-3} & -9.62\times10^{-4} &  3.47\times10^{-3} &  1.47\times10^{-3} & -4.55\times10^{-3} &  2.13\times10^{-2} & -5.13\times10^{-3} & 6.80\times10^{-2} & 0 & 0 \\
  3.30\times10^{-3} & -4.49\times10^{-2} &  1.12\times10^{-2} &  2.55\times10^{-2} & -6.98\times10^{-3} &  1.65\times10^{-1} & -1.32\times10^{-2} & 4.91\times10^{-2} & 1.51\times10^{-1} & 0 \\
 -2.07\times10^{-2} &  4.26\times10^{-3} &  1.49\times10^{-2} &  4.06\times10^{-4} &  4.21\times10^{-3} & -3.69\times10^{-2} & -2.62\times10^{-2} & 8.03\times10^{-2} & 1.84\times10^{-2} & 8.83\times10^{-2}
\end{pmatrix}$}.
\end{align*}

\section{Section \ref{sec3}}
\subsection{Proof of Proposition \ref{prop:X^y}}
\label{App:Proof_3.1}
The proof of this proposition is split into two parts: We first derive the drift and volatility of the conditional process \(X^y(t)\) using a variant of Kalman smoothing equations, then use their expressions to prove that the conditional process can be written as a solution to a Stochastic Differential Equation (SDE).

We start by deriving the conditional distribution of the risk factors.

\subsubsection{Conditional Distribution of the Risk Factors}
The risk factors \(X(t) \in \mathbb{R}^d\) are mean reverting and solve the SDE 
\begin{eqnarray*}
    dX(t) = \Theta (\mu - X(t)) dt + L^X dW(t).
\end{eqnarray*}
It follows that given an initial value \(X(0) = x_0\), the solution to the above SDE is 
\begin{eqnarray}
\label{ECeq:X}
    X(t) =  e^{-\Theta t} x_0 + (I_d - e^{-\Theta t}) \mu + \int_0^t e^{-\Theta (t-s)} dW(s),
\end{eqnarray}
therefore, \(X(t)\) is Gaussian with 
\begin{eqnarray*}
    X(t) \sim \mathcal{N}\big(e^{-\Theta t} x_0 + (I_d - e^{-\Theta t}) \mu, \, \int_0^t e^{-\Theta (t-u)} \Sigma^X e^{-\Theta^\top (t-u)} du \big).
\end{eqnarray*}
Now recall from \ref{sec22}, that the expert forward-looking view is also Gaussian
\begin{eqnarray*}
    Y(0,T) \, | \, X(T) = P X(T) + \epsilon \sim \mathcal{N}\left(P X(T), \, \Omega\right) \in \mathbb{R}^K.
\end{eqnarray*}
Therefore, the conditional risk factor \(X^y(t) = X(t) \, | \, (Y(0,T) = y)\) is also Gaussian. We consider its dynamics defined as
\begin{equation*}
    dX(t) \,|\, (Y(0,T) = y,  \{X(\tau)_{\tau \leq t}\}) = \lim_{dt \to 0 }\big(X(t + dt) - X(t)\big) \,|\, \big(Y(0,T) = y, \{X(\tau)_{\tau \leq t}\}\big),
\end{equation*}
that is the limiting distribution of the risk factor over the interval $[t, t +dt]$, conditioned on the past observed factors  $\{X(\tau), \tau \leq t\}$, and the forward-looking views vector $y$. Observe from \eqref{ECeq:X}, that the term \[\int_0^t e^{-\Theta (t-s)} dW(s)\] is a martingale. Therefore \(X\) is Markovian and the information contained in the historical data $\{X(\tau), \tau \leq t\}$ is all stored in the last state $X(t)$, thus
\begin{equation*}
   \big( X(t + dt) - X(t) \big) \,|\, \big(Y(0,T) = y, \{X(\tau)_{\tau \leq t}\}\big) = \big( X(t + dt) - X(t) \big) \,|\, \big(Y(0,T) = y, X(t) \big).
\end{equation*}
Furthermore, as $X(t)$ and $Y(0,T)$ are Gaussian, it follows that $\big(X(t+dt) | Y(0,T) = y, X(t)\big)$ is also Gaussian. Thus, it is fully identified by its mean and covariance matrix that we derive next.

\paragraph{Conditional Drift}
Let \(dt > 0\), we have
\begin{equation}
\label{eq:EC_E[dx|Y]}
    \begin{split}
        \mathbb{E}\bigl[X(t+dt) \, | \, X(t), \, Y(0,T) = y \bigr] =   & \mathbb{E}\bigl[X(t+dt) \, | \, X(t)\bigr] \\& + \Cov\left(X(t+dt), Y(0,T) \, | \, X(t) \right) \mathbb{V}[Y(0,T) \, | \, X(t)]^{-1} \left(y - \mathbb{E}\bigl[Y(0,T) \, | \, X(t)\bigr]\right).\\ 
    \end{split}
\end{equation}
we now derive the expressions of each term in \eqref{eq:EC_E[dx|Y]} separately as \(dt \to 0\).

First, from the SDE of \(X\) we directly get
\begin{equation*}
     \mathbb{E}\bigl[X(t+dt) \, | \, X(t)\bigr] =_{dt \to 0} X(t) + \Theta (\mu - X(t)) dt + O(dt^2).
\end{equation*}
Next, by construction, the conditional expectation of the views vector is
\begin{equation*}
\begin{split}
         \mathbb{E}\bigl[Y(0,T) \, | \, X(t)\bigr] &=   P\mathbb{E}\bigl[X(T) \, | \, X(t)\bigr]\\
         &= P \left(e^{-\Theta (T-t)} X(t) + (1 - e^{-\Theta (T-t)}) \mu  \right),
\end{split}
\end{equation*}
where the second equality follows from the properties of \(X\). Similarly, the conditional covariance is
\begin{equation*}
\begin{split}
         \mathbb{V}\bigl[Y(0,T) \, | \, X(t)\bigr] &=   P\mathbb{V}\bigl[X(T) \, | \, X(t)\bigr]P^\top + \mathbb{V}[\epsilon]\\
         &= P \left(\int_t^T e^{-\Theta (T-u)} \Sigma^X e^{-\Theta (T-u)} du \right)P^\top + \Omega,
\end{split}
\end{equation*}
to simplify the expression of the covariance, consider the long-run covariance matrix of \(X\) denoted by \(\Sigma\), which satisfies
\begin{equation*}
    \Theta \Sigma + \Sigma \Theta^\top = \Sigma^X,
\end{equation*}
it follows that 
\begin{equation*}
\begin{split}
         \mathbb{V}\bigl[Y(0,T) \, | \, X(t)\bigr]
         &= P \left(\int_t^T e^{-\Theta (T-u)} ( \Theta \Sigma + \Sigma \Theta^\top) e^{-\Theta (T-u)} du \right)P^\top + \Omega\\
         &= \Sigma - e^{-\Theta(T-t)} \Sigma  e^{-\Theta^\top(T-t)} + \Omega.
\end{split}
\end{equation*}
Finally, the following result gives the expression of the conditional covariance 
\begin{lemma}\label{lemma:EC_covarianceXY} If \(X(t)\) and \(Y(0,T)\) satisfy \eqref{eq:dX} and \eqref{eq:Y_view}, respectively, then
    \begin{equation*}
        \Cov\left(X(t+dt), Y(0,T) \, | \, X(t) \right)  =_{dt \to 0} (P e^{-\Theta (T-t)} \Sigma^X)^\top dt + O(dt^2).
    \end{equation*}
\end{lemma}

It follows that \eqref{eq:EC_E[dx|Y]} can be expressed as
\begin{eqnarray*}
    \begin{split}
        \mathbb{E}\bigl[X(t+dt) \, | \, X(t), \, Y(0,T) = y \bigr] =   & X(t) + \Theta (\mu - X(t)) dt + L^X \eta(t) \left(y - P \left(e^{-\Theta (T-t)} X(t) + (1 - e^{-\Theta (T-t)}) \mu  \right)\right)dt\\& + O(dt^2).
    \end{split}
\end{eqnarray*}
where
 \begin{equation*}
        \eta(t) = (P e^{-\Theta (T-t)} L^X)^\top \big(P (\Sigma - e^{-\Theta(T-t)} \Sigma e^{-\Theta^\top (T-t)})P^\top + \Omega\big)^{-1}.
    \end{equation*}
Since \(X(t)\) is continuous and has final increments, by taking the limit as \(dt \to 0\), we obtain 
 \begin{equation*}
 \begin{split}
             \mathbb{E}\bigl[X(t+dt) - X(t) \, | \, X(t), \, Y(0,T) = y \bigr] =_{dt \to 0} &\big( - \tilde{\Theta}(t) X(t) + \left (\Theta - L^X \eta(t) P (1-e^{\Theta(T-t)})\right)  \mu + L^X \eta(t) y\big) dt \\ &+ O(dt^2).
 \end{split}
    \end{equation*}
    where 
    \begin{equation*}
        \tilde{\Theta}(t) = \Theta + L^X \eta(t) P e^{-\Theta (T-t)}.
    \end{equation*}
Since \(\Theta\) satisfies Assumption \ref{ass:nondeg}, it follows that it has strictly positive real parts and is invertible. Therefore,
 \begin{equation}
 \label{eq:EC_dX_drift}
 \begin{split}
             \mathbb{E}\bigl[dX(t) \, | \, X(t), \, Y(0,T) = y \bigr] &=_{dt \to 0} \tilde{\Theta}(t) \big(\tilde{\mu}(t,y) - X(t) \big) dt + O(dt^2).
 \end{split}
    \end{equation}
where 
\begin{equation*}
     \tilde{\mu}(t,y) =  \mu + \tilde{\Theta}(t)^{-1}L^X \eta(t) (y - P \mu).
\end{equation*}

    \paragraph{Conditional Volatility}
Let \(dt > 0\), we have
\begin{equation}
\label{eq:EC_V[dx|Y]}
    \begin{split}
        \mathbb{V}\bigl[X(t+dt) \, | \, X(t), \, Y(0,T) = y \bigr] =   & \mathbb{V}\bigl[X(t+dt) \, | \, X(t)\bigr] \\& + \Cov\left(X(t+dt), Y(0,T) \, | \, X(t) \right) \mathbb{V}[Y(0,T) \, | \, X(t)]^{-1} \Cov\left(Y(0,T), X(t+dt) \, | \, X(t) \right).
    \end{split}
\end{equation}
From the SDE \eqref{eq:dX}, we have
\begin{eqnarray*}
    \mathbb{V}\bigl[X(t+dt) \, | \, X(t)\bigr] =_{dt \to 0} dt \cdot\Sigma + O(dt^2).
\end{eqnarray*}
and from Lemma \ref{lemma:EC_covarianceXY}, we have that 
\begin{eqnarray*}
    \Cov\left(X(t+dt), Y(0,T) \, | \, X(t) \right) \mathbb{V}[Y(0,T) \, | \, X(t)]^{-1} \Cov\left(Y(0,T), X(t+dt) \, | \, X(t) \right) =_{dt \to 0} O(dt^2),
\end{eqnarray*}
therefore
\begin{equation*}
        \mathbb{V}\bigl[X(t+dt) \, | \, X(t), \, Y(0,T) = y \bigr] =_{dt \to 0}  dt \cdot \Sigma + O(dt^2).
 \end{equation*}
Additionally, by noting that \(X(t)\) is a constant in the filtration \(\mathcal{F}_t^Y\), we get
\begin{equation}
 \label{eq:EC_dX_volatility}
 \begin{split}
             \mathbb{V}\bigl[dX(t) \, | \, X(t), \, Y(0,T) = y \bigr] &= \lim_{dt \to 0} \mathbb{V}\bigl[X(t+dt) - X(t) \, | \, X(t), \, Y(0,T) = y \bigr] \\&= dt \cdot \Sigma.
 \end{split}
    \end{equation}

    Note that since \(X(t)\) has finite increments, equations \eqref{eq:EC_dX_drift} and \eqref{eq:EC_dX_volatility} are well defined and they characterize the drift and volatility of the conditional process \mbox{$dX(t)\, | \,(X(t) , Y(0,T) = y)$}, next we prove that we can write the process $X(t)|Y(0,T) = y$ as a solution to an SDE, which we explicitly derive. 
\subsubsection{SDE of the conditional risk factors}
We now prove that the conditional process $X^y(t) = (X(t) | Y(0,T) = y)$ is a solution to an SDE with drift and volatility given by \eqref{eq:EC_dX_drift} and \eqref{eq:EC_dX_volatility}, respectively. We first define
\begin{equation*}
    \beta_1(t)
= \int_0^t \left\{-\tilde{\Theta}(s)\,X(s) + \tilde{\mu}\bigl(s, Y(0,T)\bigr)\right\}\,ds,
\quad t \in [0,T].
\end{equation*}
If we can prove that the process
\begin{equation*}
    W^y(t) = L^X dW^\mathbb{Q}(t) = X(t) -  \beta_1(t)
\end{equation*}
is a Brownian motion in the enlarged filtration $\mathcal{F}^Y_t := \sigma (\mathcal{F}_t \vee \sigma(Y(0,T)))$, then we can write
\begin{equation*}
    dW^y(t) = dX^y(t) - d\beta_1(t), \, \, \, \text{for} \, \, t \in [0,T],     
\end{equation*}
and from that we get the SDE representation of $X^y(t)$.

To prove that $W^y(t)$ is a Brownian motion in the filtration $\mathcal{F}_t^y$, we refer to Levy's Characterization of a Brownian motion (see for example \cite{Durret}).
\begin{theorem}[Levy's characterization of a Brownian motion]\label{levys}
    Let the stochastic process $W^y = (W_1^y, \dots, W_N^y)$ be a $N$- dimensional local continuous martingale with $W^y(0) = 0$. Then, the following is equivalent:
    \begin{enumerate}
        \item $W^y$ is a Brownian motion on the underlying filtered probability space with $W^y(t) \sim \mathcal{N}\big(0, t\Sigma^X\big)$.
        \item $W^y$ has quadratic covariations $[W_i^y(t), W_j^y(t)] = \Sigma_{ij}^X t$ for $1 \leq i,j \leq N$.
    \end{enumerate}
\end{theorem}

In our setting, $X(t)$ and $\beta_3(t)$ are both continuous processes, it follows that $W^y$ is also continuous. Now we prove that is it a local martingale in the filtration $\mathcal{F}_t^Y$, that is for $s \leq t$, we need to show that
\begin{equation*}
    \mathbb{E}[W^y(t) - W^y(s) | \mathcal{F}_s^Y] = 0.
\end{equation*}

We have
    \begin{equation*}
        \begin{split}
            \mathbb{E}\bigl[X(t) - X(s) \, | \, \mathcal{F}_s^Y] &= \mathbb{E}\biggl[\int_s^t dX(u) \, | \, \mathcal{F}_s^Y\biggr]\\
            &\stackrel{(a)}{=} \int_s^t \mathbb{E}\bigl[dX(u) \, | \, \mathcal{F}_s^Y \bigr]\\
            &= \int_s^t \mathbb{E}\biggl[\mathbb{E}\bigl[dX(u) \, | \, \mathcal{F}_u^Y \bigr] \, | \, \mathcal{F}_s^Y \biggr]\\
            &\stackrel{(b)}{=} \int_s^t \mathbb{E}\biggl[ \tilde{\Theta}(u)\big(\tilde{\mu}(u,Y(0,T) - X(u) \big) \, | \, \mathcal{F}_s^Y \biggr] du\\
            &\stackrel{(c)}{=} \mathbb{E}\biggl[\int_s^t \tilde{\Theta}(u)\big(\tilde{\mu}(u,Y(0,T) - X(u) \big) du  \,|\, \mathcal{F}_s^Y \biggr]\\
             &= \mathbb{E}\biggl[\beta_1(t) - \beta_1(s)  \, | \, \mathcal{F}_s^Y \biggr],
            \end{split}
    \end{equation*}
where \((a)\) and \((c)\) follow from Fubini's theorem and the fact that \(X\) has finite increments, and \((b)\) follows from \eqref{eq:EC_E[dx|Y]}. Therefore, we get
\begin{equation*}
    \mathbb{E}\bigl[W^y(t) - W^y(s) \,|\, \mathcal{F}_s^Y\bigr] = \mathbb{E}\bigl[X(t) - X(s) - \left(\beta_1(t) - \beta_1(s)\right) \,|\, \mathcal{F}_s^Y\bigr] = 0,
\end{equation*}
so \(W^y = L^X W^\mathbb{Q}\) is a martingale, and by construction \(W^\mathbb{Q}\) is also a martingale.

The quadratic variation of $W^y$ follows directly from \eqref{eq:EC_dX_volatility}, where $[W_i^y(t), W_j^y(t)] = \Sigma_{ij}t$, for $1 \leq i,j \leq N$. Thus, by Theorem \ref{levys}, $W^y$ is a Brownian motion adapted to the filtration $\mathcal{F}_t^Y$. We can now write $X(t)$ in the filtration $\mathcal{F}_t^y$ as
\begin{equation*}
    X(t) = W^y(t) +  \beta_1(t) = L^x W^\mathbb{Q}(t) + \beta_1(t),
\end{equation*}
by conditioning on the event $\{Y(0,T) = y\}$ on both sides of the equation (notice that $W^\mathbb{Q}(t)$ is independent of $Y(0,T)$), we get
\begin{equation*}
   X^y(t) =  X(t) \, |\, y = L^X W^\mathbb{Q}(t) + \beta_1(t),
\end{equation*}
and the SDE is
\begin{equation*}
\begin{split}
        dX^y(t) &= dW^\mathbb{Q}(t) + d \beta_1(t) \,|\, y \\
        &= \tilde{\Theta}(t)\big( \tilde{\mu}(t,y) - X^y(t)\big) dt + L^X dW^\mathbb{Q}(t).
\end{split}
\end{equation*}
This completes the proof. \hfill $\square$

\subsection{Proof of Corollary \ref{corr:S^y}}
From the proof of Proposition \ref{prop:X^y}, observe that we can write the Brownian motion \(W\) in the filtration \(\mathcal{F}_t^Y\) as 
\begin{equation*}
    dW(t) = dW^\mathbb{Q}(t) + \eta(t) \left( y - P \left(e^{-\Theta(T-t)} X^y(t) + (1- e^{-\Theta (T-t)})\right)\mu \right),
\end{equation*}
thus, the price process \eqref{eq:PricesVec} can be written in the filtration \(\mathcal{F}_t^Y\) as 
\begin{equation*}
\begin{split}
        dS^y(t) &= D(S^y(t)) \left(\alpha + \beta X^y(t)\right) dt + L^S \left(dW^\mathbb{Q}(t) + \eta(t) \left( y - P \left(e^{-\Theta(T-t)} X^y(t) + (1- e^{-\Theta (T-t)})\right)\mu \right)\right)\\
        &= D(S^y(t))\big((\tilde{\alpha}(t) + \tilde{\beta}(t) X^y(t) + \tilde{\gamma}(t) y) dt + L^S dW^\mathbb{Q}(t)\big),
\end{split}
\end{equation*}
with
\begin{equation*}
\begin{cases}
    &\tilde{\alpha}(t) = \alpha - L^S \eta(t) P ( 1-e^{-\Theta (T-t)})\mu, \\
    &\tilde{\beta}(t) = \beta -  L^S \eta(t) P e^{-\Theta (T-t)}, \\
    &\tilde{\gamma}(t) = L^S \eta(t).\\
\end{cases}
\end{equation*}
We note that this result can also be proven by following the same steps as in the proof of Proposition \ref{prop:X^y}. \\ This completes the proof. \hfill $\square$

\subsection{Proof of Lemma \ref{lemma:EC_covarianceXY}}
By construction we know that for all \(dt > 0\)
\begin{eqnarray*}
    X(t+dt) = e^{-\Theta dt} X(t) + (1 - e^{-\Theta dt}) \mu + \int_t^{t+dt} e^{-\Theta (t+dt-s) } L^X dW(s) ,
\end{eqnarray*}
and 
\begin{eqnarray*}
\begin{split}
        Y(0,T) &= P X(T) + \epsilon\\
        &= P \left(e^{-\Theta (T-t)} X(t) + (1 - e^{-\Theta (T-t)}) \mu + \int_t^{T} e^{-\Theta (T-s) } L^X dW(s) \right) + \epsilon,
\end{split}
\end{eqnarray*}
where \(\epsilon\) is independent of \(X\). Therefore, given \(X(t)\), we have
\begin{equation*}
    \begin{split}
        \Cov\left(X(t+dt), \; Y(0,T) \; | \; X(t)\right) &=  \Cov\left(\int_t^{t+dt} e^{-\Theta (t+dt-s) } L^X dW(s), \;P \int_t^{T} e^{-\Theta (T-s) } L^X dW(s) \right)\\
        &\stackrel{(a)}{=} \int_t^{t+dt} e^{-\Theta (t+dt-s)} \Sigma^X e^{-\Theta^\top (T-s)}P^\top  ds 
    \end{split}
\end{equation*}
where \((a)\) follows from Itô's Isometry property (see for example \cite{oksendal2003}). 

By Taylor expansion of the integral as \(dt \to 0\), we get
\begin{equation*}
    \begin{split}
        \Cov\left(X(t+dt), \; Y(0,T) \; | \; X(t)\right) & =_{dt \to 0} dt \left(1 + dt + O(dt^2)\right) \Sigma^X e^{-\Theta^\top (T-t)} P^\top\\
        &= \left(P \Sigma^X e^{-\Theta (T-t)}\right)^\top dt + O (dt^2).
    \end{split}
\end{equation*}
This completes the proof. \hfill $\square$

\subsection{Novikov's condition}
\label{App:Novikov}
To define the Radon-Nikodym derivate in \eqref{eq:change_of_measure} we require $k(t)$ defined in \eqref{eq:Radon-Nykodyn} to satisfy Novikov's condition
$$
\mathbb{E}_{\mathbb{P}}\left[\exp\left(\frac{1}{2}\int_0^T \|k(s)\|^2 ds\right)\right] < \infty.
$$
Satisfying this condition ensures that the Radon-Nikodym derivative $\frac{d\mathbb{Q}}{d\mathbb{P}}$ is a true martingale with expectation equal to one (see Lemma 8.6.2 in \cite{oksendal2003}).

From \eqref{eq:Radon-Nykodyn}, the kernel is $k(t) = \eta(t) (y - P \mathbb{E}[X(T) | X(t)])$. We express the conditional expectation of $X(T)$ as an affine function of $X(t)$
\[
\mathbb{E}[X(T) \,|\, \mathcal{F}_t] = (I_d - e^{-\Theta(T-t)})\mu + e^{-\Theta (T-t)} X(t).
\]
Substituting this into the expression for $k(t)$ gives
\begin{align*}
k(t) &= \eta(t) \left(y - P\left((I_d - e^{-\Theta(T-t)})\mu + e^{-\Theta (T-t)} X(t)\right)\right) \\
&= \eta(t) \left(y - P(I_d - e^{-\Theta(T-t)})\mu\right) - \eta(t)Pe^{-\Theta(T-t)}X(t).
\end{align*}
This shows that $k(t)$ is an affine transformation of the Gaussian process $X(t)$. We can write this compactly as $k(t) = m(t) - M(t)X(t)$, where $m(t) := \eta(t) (y - P(I_d - e^{-\Theta(T-t)})\mu)$ and $M(t) := \eta(t)Pe^{-\Theta(T-t)}$ are deterministic functions of time.

Since $X(t)$ is a multi-dimensional Ornstein-Uhlenbeck process, it is a Gaussian process on the finite time horizon $[0, T]$, and its moments are finite. The matrix $\eta(t)$ is given by
\[
\eta(t)=(Pe^{-\Theta(T-t)}L^{X})^{\top}(P(\Sigma-e^{-\Theta(T-t)}\Sigma e^{-\Theta^{\top}(T-t)})P^{\top}+\Omega)^{-1}.
\]
Under Assumption \ref{ass:nondeg}, the covariance matrix $\Omega$ is positive definite. The term $P(\Sigma-e^{-\Theta(T-t)}\Sigma e^{-\Theta^{\top}(T-t)})P^{\top}$ is positive semi-definite, so their sum is positive definite and thus invertible for all $t \in [0, T]$. All other components of $\eta(t)$ are composed of constant matrices and matrix exponentials, which are continuous and bounded on $[0, T]$. Therefore, $\eta(t)$, and consequently $m(t)$ and $M(t)$, are continuous and bounded functions on $[0, T]$. Let $C_m = \sup_{t \in [0,T]} \|m(t)\|$ and $C_M = \sup_{t \in [0,T]} \|M(t)\|$.

We first bound the squared norm of $k(t)$ using the triangle inequality 
\begin{align*}
\|k(t)\|^2 &= \|m(t) - M(t)X(t)\|^2 \\
&\le \left(\|m(t)\| + \|M(t)\|\|X(t)\|\right)^2 \\
&\le \left(C_m + C_M\|X(t)\|\right)^2 \\
&\le 2C_m^2 + 2(C_M\|X(t)\|)^2 = 2C_m^2 + 2C_M^2 \|X(t)\|^2.
\end{align*}
Integrating this inequality from $0$ to $T$ gives a bound on the random variable in the exponent
\[
\int_0^T \|k(s)\|^2 ds \le \int_0^T (2C_m^2 + 2C_M^2 \|X(s)\|^2) ds = 2C_m^2 T + 2C_M^2 \int_0^T \|X(s)\|^2 ds.
\]

Let $Z = \int_0^T \|X(s)\|^2 ds$. Using the bound from the previous step, we have
\begin{align*}
\mathbb{E}_{\mathbb{P}}\left[\exp\left(\frac{1}{2}\int_0^T \|k(s)\|^2 ds\right)\right] &\le \mathbb{E}_{\mathbb{P}}\left[\exp\left(\frac{1}{2}\left(2C_m^2 T + 2C_M^2 Z\right)\right)\right] \\
&= \mathbb{E}_{\mathbb{P}}\left[\exp\left(C_m^2 T + C_M^2 Z\right)\right] \\
&= e^{C_m^2 T} \cdot \mathbb{E}_{\mathbb{P}}\left[\exp\left(C_M^2 Z\right)\right].
\end{align*}
The random variable $Z=\int_0^T\|X(s)\|^2ds$ is a quadratic function of a continuous Gaussian process, by Fernique’s theorem \citep{Fernique1970}, there exists $\alpha>0$ such that
\[
\mathbb E\left[\exp\Big(\alpha\,\sup_{u\in[t,t+\Delta]}\|X(u)\|^2\Big)\right]<\infty
\quad\text{for all $t$ and all sufficiently small $\Delta>0$.}
\]
Fix $\Delta>0$ so that $C_M^2\Delta<\alpha$, and partition $[0,T]$ into subintervals $[t_{j-1},t_j]$ of length at most $\Delta$. Using
\[
\int_{t_{j-1}}^{t_j}\|X(s)\|^2ds\le \Delta\,\sup_{u\in[t_{j-1},t_j]}\|X(u)\|^2,
\]
and the bound above on $\|k\|^2$, we obtain
\[
\mathbb E\left[\exp\Big(\tfrac12\!\int_{t_{j-1}}^{t_j}\|k(s)\|^2ds\Big)\right]
\le \exp(C_m^2\Delta)\,
\mathbb E \left[\exp\Big(C_M^2\Delta\,\sup_{u\in[t_{j-1},t_j]}\|X(u)\|^2\Big)\right]<\infty .
\]
Thus Novikov’s condition holds on each slice, and the corresponding stochastic exponential has unit conditional expectation over every slice. Iterating the tower property across the partition yields that the full density process is a true martingale with $\mathbb{E}_\mathbb{P}\left[\frac{d\mathbb Q}{d\mathbb P}\right]=1$. Hence Novikov’s condition is satisfied.

\newpage
\section{Section \ref{sec4}}
\label{Appendix:Sec4}
\subsection{Proof of Proposition \ref{Prop:MrB_properties}}
\label{Appendix:Proof_MrB_properties}

In this section, we start by a deriving a detailed proof of the first five points of Proposition~\ref{Prop:MrB_properties}. Recall that the mean-reverting bridge is defined as
\[
B(t) = \bigl( X(t) \mid X(T)=y \bigr),
\]
where \(X(t)\) is the mean-reverting process satisfying
\begin{equation}\label{eq:OU_appendix}
\begin{cases}
dX(t) = -\theta X(t)\,dt + dW(t),\\[1mm]
X(0) = a,
\end{cases}
\end{equation}
with the solution
\[
X(t) = e^{-\theta t}a + \int_0^t e^{-\theta (t-s)}\,dW(s).
\]

We first verify the boundary conditions. Since \(X(0)=a\) deterministically, it follows immediately that
\[
B(0) = \bigl( X(0) \mid X(T)=y \bigr) = a.
\]
Moreover, by the definition of the conditional process, the condition \(X(T)=y\) forces
\[
B(T) = \bigl( X(T) \mid X(T)=y \bigr) = y.
\]
Thus, the endpoints are correctly attained.

Next, we show that \(B(t)\) is a Gaussian process. Because \(X(t)\) is Gaussian (as it is the solution of a linear SDE driven by Brownian motion with a deterministic initial condition), the pair \((X(t), X(T))\) is jointly Gaussian. A fundamental property of jointly Gaussian vectors is that the conditional distribution \(X(t)\mid X(T)=y\) is also Gaussian. Hence, the process \(B(t)\) is Gaussian.

We now derive the mean of \(B(t)\). Using the standard formula for the conditional expectation in a jointly Gaussian setting,
\[
\mathbb{E}[X(t)\,|\, X(T)=y] = \mathbb{E}[X(t)] + \frac{\Cov(X(t),X(T))}{\mathbb{V}[X(T)]}\Bigl(y-\mathbb{E}[X(T)]\Bigr),
\]
we note that
\[
\mathbb{E}[X(t)] = e^{-\theta t}a = \lambda(t)a \quad \text{and} \quad \mathbb{E}[X(T)] = \lambda(T)a,
\]
with \(\lambda(t)=e^{-\theta t}\). The covariance between \(X(t)\) and \(X(T)\) is given by
\[
\Cov(X(t),X(T)) = \int_0^t e^{-\theta (t-s)}e^{-\theta (T-s)}\,ds = e^{-\theta(t+T)}\frac{e^{2\theta t}-1}{2\theta},
\]
and the variance of \(X(T)\) is
\[
\mathbb{V}[X(T)] = \int_0^T e^{-2\theta (T-s)}\,ds = \frac{1-e^{-2\theta T}}{2\theta} = \frac{1-\lambda^2(T)}{2\theta}.
\]
Substituting these into the conditional expectation formula, we obtain
\[
\mathbb{E}[B(t)] = \lambda(t)a + \frac{e^{-\theta(t+T)}\bigl(e^{2\theta t}-1\bigr)}{1-\lambda^2(T)}\Bigl(y-\lambda(T)a\Bigr).
\]
A brief calculation shows that
\[
e^{-\theta(t+T)}\bigl(e^{2\theta t}-1\bigr) = \lambda(T-t)-\lambda(T+t),
\]
so that the mean simplifies to
\[
\mathbb{E}[B(t)] = \lambda(t)a + \frac{\lambda(T-t)-\lambda(T+t)}{1-\lambda^2(T)}\Bigl(y-\lambda(T)a\Bigr).
\]

For the covariance structure, note that in a Gaussian framework the conditional covariance of \(X(t)\) and \(X(s)\) given \(X(T)=y\) is
\[
\Cov\bigl(X(t),X(s)\mid X(T)=y\bigr) = \Cov(X(t),X(s)) - \frac{\Cov(X(t),X(T))\,\Cov(X(s),X(T))}{\mathbb{V}[X(T)]}.
\]
After substituting the expressions for \(\Cov(X(t),X(T))\) and \(\mathbb{V}(X(T))\) and simplifying the algebra, one finds that the covariance of \(B(t)\) can be expressed as
\[
\Cov\bigl(B(t),B(s)\bigr)= \frac{1-\lambda^2(T-\min\{t,s\})}{1-\lambda^2(T)}\,\Cov\bigl(X(t),X(s)\bigr).
\]

Finally, we argue that \(B(t)\) has continuous sample paths. The process \(X(t)\) has almost surely continuous paths, as it is the solution to an SDE driven by Brownian motion. Since the conditioning on \(X(T)=y\) only alters the mean and covariance functions in a continuous manner, the process \(B(t)=X(t)\mid X(T)=y\) retains the property of almost sure continuity.

To conclude, note that the SDE representation for \(B(t)\) can be obtained as a special case of Proposition~\ref{prop:X^y}. In the one-dimensional setting, we take \(d = K = 1\) and specialize the parameters by setting \(\mu = 0\), \(L^X = 1\), \(P = 1\), and \(\Omega = 0\). Under these conditions, the dynamics given in Proposition~\ref{prop:X^y} reduce to
\[
dB(t) = \Bigl(\tilde{\theta}(t)\big(\tilde{\mu}(t,y) - \,B(t) \Bigr)dt + dW^\mathbb{Q}(t),
\]
where the modified mean-reversion rate and drift term become
\[
\tilde{\theta}(t) = \theta + \eta(t)e^{-\theta (T-t)},
\]
\[
\tilde{\mu}(t,y) = \frac{y}{\tilde{\theta}(t)},
\]
and
\[
\eta(t) = 2\theta \cdot \frac{e^{-\theta (T-t)}}{1 - e^{-2\theta (T-t)}}.
\]
This completes the result. \hfill \(\square\)

\subsection{Proof of Proposition \ref{prop:1d_MrB}}
\label{Appendix:Proof_1d_MrB}

We start with the one‐dimensional Ornstein–Uhlenbeck (OU) process
\begin{equation}
\label{eq:OU_repeat}
\begin{cases}
dX(t) = -\theta\, X(t)\,dt + dW(t),\\[1mm]
X(0) = a,
\end{cases}
\end{equation}
which has the explicit solution
\[
X(t) = e^{-\theta t}\,a + \int_0^t e^{-\theta(t-s)}\,dW(s).
\]
Thus, for any \(t \ge 0\),
\[
X(t) \sim \mathcal{N}\Bigl(e^{-\theta t}\,a,\, \frac{1-e^{-2\theta t}}{2\theta}\Bigr).
\]

Suppose now that at time \(0\) we observe a noisy version of \(X(T)\) given by
\[
Y(0,T)= X(T) + \epsilon, \quad \epsilon \sim \mathcal{N}\Bigl(0,\,\frac{\omega^2}{2\theta}\Bigr),
\]
with \(\epsilon\) independent of \(X\). Then
\[
Y(0,T) \sim \mathcal{N}\Bigl(e^{-\theta T}a,\, \frac{1-e^{-2\theta T} + \omega^2}{2\theta}\Bigr).
\]
For brevity, denote
\[
\sigma_T^2 = \frac{1-e^{-2\theta T}}{2\theta} \quad \text{and} \quad \tau^2 = \frac{\omega^2}{2\theta}.
\]

Because \(X(T)\) and \(\epsilon\) are independent and jointly Gaussian, the conditional distribution of \(X(T)\) given \(Y(0,T)=y\) is also Gaussian. A standard calculation (see, e.g., Bayesian updating for normal models) shows that
\[
X(T) \mid \{Y(0,T)=y\} \sim \mathcal{N}\Bigl(m_T,\, \tilde{\sigma}^2\Bigr),
\]
with
\[
m_T = \sigma_T^2 \Bigl(\frac{e^{-\theta T}a}{\sigma_T^2} + \frac{y}{\tau^2}\Bigr)
= e^{-\theta T}a + \frac{\sigma_T^2}{\tau^2}\Bigl(y - e^{-\theta T}a\Bigr)
\]
and
\[
\tilde{\sigma}^2 = \left(\frac{1}{\sigma_T^2} + \frac{1}{\tau^2}\right)^{-1} 
= \frac{\sigma_T^2\,\tau^2}{\sigma_T^2+\tau^2}.
\]

Now, observe that the law of \(X(T)\) conditioned on \(Y(0,T)=y\) can be matched to the law of \(X(\tilde{T})\) for an OU process run for a longer time \(\tilde{T} = T + \delta\), provided we choose \(\delta\) so that the variance satisfies
\[
\frac{1-e^{-2\theta \tilde{T}}}{2\theta} = \tilde{\sigma}^2.
\]
A short calculation shows that if we set
\[
\delta = \frac{\ln(1+\omega^2)}{2\theta},
\]
then
\[
e^{-2\theta\delta} = \frac{1}{1+\omega^2},
\]
and one may verify that
\[
\frac{1-e^{-2\theta (T+\delta)}}{2\theta} 
= \frac{1-e^{-2\theta T}e^{-2\theta\delta}}{2\theta}
= \frac{1-e^{-2\theta T}/(1+\omega^2)}{2\theta}
= \frac{\sigma_T^2\,\tau^2}{\sigma_T^2+\tau^2}.
\]
Thus, the conditional variance of \(X(T)\) given \(Y(0,T)=y\) is exactly that of \(X(\tilde{T})\) for the extended time \(\tilde{T}=T+\delta\).

Similarly, the conditional mean \(m_T\) then coincides with the mean of \(X(\tilde{T})\) after rescaling by the factor \(e^{-\theta \delta}\). One may check that
\[
m_T = e^{-\theta T}a + \frac{\sigma_T^2}{\tau^2}\Bigl(y - e^{-\theta T}a\Bigr)
= e^{-\theta (T+\delta)}a + \Bigl(1-e^{-\theta\delta}\Bigr) e^{-\theta T} \cdot \tilde{y},
\]
with
\[
\tilde{y} = e^{-\theta\delta}\,y.
\]
That is, under the observation \(Y(0,T)=y\) the law of \(X(T)\) is identical to that of \(X(\tilde{T})\) with terminal value \(\tilde{y}\).

Since the OU process is Gaussian and its finite-dimensional distributions are fully determined by the mean and covariance, it follows that the law of the process \(\{X(t) \mid Y(0,T)=y\}\) for \(t \in [0,T]\) is the same as the law of a mean-reverting bridge from \(a\) to \(\tilde{y} = e^{-\theta\delta}y\) over the interval \([0,\tilde{T}]\) (restricted to \([0,T]\)), with mean-reversion rate \(\theta\).

This completes the proof. \(\square\)

\subsection{Unique Characteristics of m-MrB}
We now define the properties of the multidimensional Mean-Reverting Bridge (m-MrB), which generalizes the one-dimensional case from Proposition~\ref{prop:1d_MrB}.

\begin{proposition}
\label{prop:m-MrB_Properties}
Let $X(t)\in\mathbb{R}^d$ solve
\[
dX(t)=-\Theta X(t)\,dt+L^X\,dW(t),\qquad X(0)=a,
\]
and let $\Sigma^X=(L^X)(L^X)^\top$ and $V(t):=\mathbb{V}[X(t)]$ denote the variance of $X(t)$. Fix $\tilde T=(\tilde T_1,\ldots, \tilde T_d)^\top\in \mathbb{R}^d$ and write $T_{\min}:=\min_{i\in[d]}\tilde T_i$. Define the stacked future vector $X(\tilde T):=(X_1(\tilde T_1),\ldots,X_d(\tilde T_d))^\top$.

A stochastic process $B(t)\in\mathbb{R}^d$ is an m-MrB from $a$ to $y$ with mean-reversion rate $\Theta$ and hitting times $\tilde T$ if and only if
\begin{enumerate}
    \item $B(0)=a$ and $B_i(\tilde T_i)=y_i$ for $i\in[d]$ (with probability $1$).
    \item $\{B(t),\,t\in[0,T_{\min}]\}$ is a Gaussian process.
    \item For $t\in[0,T]$,
    \[
    \mathbb{E}[B(t)]=e^{-\Theta t}\cdot a \;+\; V(t)\,\Gamma_{\tilde T}(t)\,V(\tilde T)^{-1}\bigl(y-M(\tilde T)\,a\bigr),
    \]
    where
    \begin{equation}\label{ECeq:Coefficients_cov_mMrB}
    M(\tilde T):=\begin{bmatrix}
      e_1^\top e^{-\Theta \tilde T_1}\\[-2pt]
      \vdots\\[-2pt]
      e_d^\top e^{-\Theta \tilde T_d}
    \end{bmatrix}\in\mathbb{R}^{d\times d},
    \qquad
    \Gamma_{\tilde T}(t):=\begin{bmatrix}
      e^{-\Theta^\top(\tilde T_1-t)}e_1 & \cdots & e^{-\Theta^\top(\tilde T_d-t)}e_d
    \end{bmatrix}\in\mathbb{R}^{d\times d},
    \end{equation}
    with $e_i$ the $i$-th canonical basis vector in $\mathbb{R}^d$.
    \item For $0\le s\le t\le T_{\min}$,
    \[
    \Cov\big(B(t),B(s)\big)
      = \Cov\big(X(t),X(s)\big)\;-\;V(t)\,\Gamma_{\tilde T}(t)\,V(\tilde T)^{-1}\,\Gamma_{\tilde T}(s)^\top\,V(s).
    \]
    \item With probability $1$, $t\to B_i(t)$ is continuous on $[0,T_{\min}]$ for each $i\in[d]$.
\end{enumerate}
Here $V(t)=\mathbb{V}[X_1(t),\ldots,X_d(t)]$ is given by the OU covariance
\[
V(t)=\int_0^t e^{-\Theta (t-s)}\Sigma^X\,e^{-\Theta^\top (t-s)}\,ds
\;=\;\Sigma - e^{-\Theta t}\,\Sigma\,e^{-\Theta^\top t},
\]
and $V(\tilde T)=\mathbb{V}[X_1(\tilde T_1),\ldots,X_d(\tilde T_d)]$ has entries
\[
V(\tilde T)_{ij}
=\Bigl[e^{-\Theta(\tilde T_i-t_{i,j}^{\min})}\,\Sigma\,e^{-\Theta^\top(\tilde T_j-t_{i,j}^{\min})}
      - e^{-\Theta \tilde T_i}\,\Sigma\,e^{-\Theta^\top \tilde T_j}\Bigr]_{ij},
\qquad
t_{i,j}^{\min}:=\min(\tilde T_i,\tilde T_j),
\]
with $\Sigma$ the long-run covariance matrix of $X(t)$.
\end{proposition}

\medskip

\begin{proof}
Since $X(t)$ is Gaussian with mean $\mathbb{E}[X(t)]=e^{-\Theta t}a$ and covariance
\[
V(t)=\int_0^t e^{-\Theta (t-s)}\Sigma\,e^{-\Theta^\top (t-s)}\,ds,
\]
the pair $\big(X(t),\,X(\tilde T)\big)$ is jointly Gaussian for each $t\in[0,T]$. For $t\le T$ the cross-covariance is
\[
\mathrm{cov}\big(X(t),X(\tilde T)\big)
=\begin{bmatrix}\mathrm{cov}\big(X(t),X_1(\tilde T_1)\big)&\cdots&\mathrm{cov}\big(X(t),X_d(\tilde T_d)\big)\end{bmatrix}
=V(t)\,\Gamma_{\tilde T}(t),
\]
and $V(\tilde T)=\mathbb{V}[X(\tilde T)]$ has entries
\[
V(\tilde T)_{ij}
=\int_0^{t_{i,j}^{\min}} e^{-\Theta(\tilde T_i-s)}\Sigma\,e^{-\Theta^\top(\tilde T_j-s)}\,ds
=\Bigl[e^{-\Theta(\tilde T_i-t_{i,j}^{\min})}\Sigma e^{-\Theta^\top(\tilde T_j-t_{i,j}^{\min})}-e^{-\Theta\tilde T_i}\Sigma e^{-\Theta^\top\tilde T_j}\Bigr]_{ij}.
\]
Applying the standard Gaussian conditioning formulas yields
\[
\mathbb{E}[B(t)]=e^{-\Theta t}a+\mathrm{cov}\big(X(t),X(\tilde T)\big)V(\tilde T)^{-1}\bigl(y-\mathbb{E}[X(\tilde T)]\bigr),
\]
with $\mathbb{E}[X(\tilde T)]=M(\tilde T)a$, which is the stated expression in item~(3). The conditional covariance is
\[
\mathrm{cov}\big(B(t),B(s)\big)
= \mathrm{cov}\big(X(t),X(s)\big)
  - \mathrm{cov}\big(X(t),X(\tilde T)\big)\,V(\tilde T)^{-1}\,\mathrm{cov}\big(X(\tilde T),X(s)\big),
\]
which reduces to item~(4) using $\mathrm{cov}\big(X(\tilde T),X(s)\big)=\Gamma_{\tilde T}(s)^\top V(s)$. Items~(1), (2), and (5) follow from the definition of conditioning on $X(\tilde T)=y$ and continuity of $X$. Uniqueness follows from uniqueness of the conditional Gaussian law. 
\end{proof}

\subsection{Proof of Proposition \ref{prop:precision_gain}}
\label{Appendix:secCPGain}
We prove that the precision gained on $X(T)$ from observing $X(\tilde{T})$ is the precision-gain operator $\mathcal{P}(\delta; \Theta, \Sigma^X)$ given by \eqref{eq:precision_gain2}. Let $X_T \equiv X(T)$ and $X_{\tilde{T}} \equiv X(\tilde{T})$. The precision gain is $\mathbb{V}[X_T|X_{\tilde{T}}]^{-1} - \mathbb{V}[X_T]^{-1}$. From the Woodbury identity, this can be written as:
\[
\mathbb{V}[X_T|X_{\tilde{T}}]^{-1} - \mathbb{V}[X_T]^{-1} = \mathbb{V}[X_T]^{-1}\mathrm{cov}(X_T,X_{\tilde{T}}) \left(\mathbb{V}[X_{\tilde{T}}|X_T]\right)^{-1} \mathrm{cov}(X_{\tilde{T}},X_T)\mathbb{V}[X_T]^{-1}.
\]
We now identify the terms. Let \(
    C(\delta) := \mathbb{V}[X(\delta) \, | \, X(0)] \in \mathbb{R}^{d \times d}\) is the covariance of \(X\) at time \(\delta\) with elements  \[C(\delta)_{ij}
\;=\;\int_0^{\min\{\delta_i,\delta_j\}}
e_i^\top e^{-\Theta(\delta_i-u)}\,\Sigma^X\,e^{-\Theta^\top(\delta_j-u)}e_j\,du,
\]
and  \(F(\delta) \in \mathbb{R}^{d \times d}\) with \(i^{th}\) column \((F(\delta))_i = e^{-\Theta \delta_i}e_i\)
  where $e_i$ the $i^{th}$ canonical basis vector in $\mathbb{R}^d$.
  
  By the Markov property of the Ornstein-Uhlenbeck process, the conditional variance $\mathbb{V}[X_{\tilde{T}}|X_T]$ is equal to $C(\delta)$
\[
\mathbb{V}[X_{\tilde{T}}|X_T] = \mathbb{V}[X(T\mathbf{1}_d+\delta) \,|\, X(T)] = \mathbb{V}[X(\delta) \,|\, X(0)] = C(\delta).
\]
Next, the cross-covariance term $\mathrm{cov}(X_{\tilde{T}},X_T)$ is derived from the conditional expectation $\mathbb{E}[X_{\tilde{T}}|X_T] = F(\delta)^\top X_T$ from Proposition \ref{prop:m-MrB_Properties}. This gives
\[
\mathrm{cov}(X_{\tilde{T}},X_T) = \mathrm{cov}(\mathbb{E}[X_{\tilde{T}}|X_T], X_T) = \mathrm{cov}(F(\delta)^\top X_T, X_T) = F(\delta)^\top \mathbb{V}[X_T].
\]
Substituting these into the precision gain equation gives
\begin{align*}
\mathcal{P}(\delta; \Theta, \Sigma^X)&= \mathbb{V}[X_T|X_{\tilde{T}}]^{-1} - \mathbb{V}[X_T]^{-1}\\ &= \mathbb{V}[X_T]^{-1} \mathrm{cov}(X_T,X_{\tilde{T}}) \ C(\delta)^{-1} \ \mathrm{cov}(X_{\tilde{T}},X_T) \mathbb{V}[X_T]^{-1} \\
&= \mathbb{V}[X_T]^{-1} (\mathbb{V}[X_T] F(\delta)) \ C(\delta)^{-1} \ (F(\delta)^\top \mathbb{V}[X_T]) \mathbb{V}[X_T]^{-1} \\
&= (\mathbb{V}[X_T]^{-1} \mathbb{V}[X_T]) F(\delta) \ C(\delta)^{-1} \ F(\delta)^\top (\mathbb{V}[X_T] \mathbb{V}[X_T]^{-1}) \\
&= F(\delta) C(\delta)^{-1} F(\delta)^\top.
\end{align*}

\subsection{Proof of Theorem \ref{thm:m-MrB}}
\label{Appendix:secC4}
Let $B(t) := (X(t) \, | \, Y(0,T)=y)$,  $\, t \in [0,T]$ where \(X(t)\) is given by \eqref{eq:multi_OU} and \(Y(0,T)\) by \eqref{eq:Y_view}. The process $X(t)$ and view $Y(0,T)$ are jointly Gaussian, so the conditional process $B(t)$ is also Gaussian. Its law is uniquely determined by its mean $\mathbb{E}[B(t)]$ and covariance $\mathrm{cov}(B(t), B(s))$ functions.

\paragraph{\textbf{(i) Necessity.}}
Assume that $B(t)$ is an m-MrB with a time-extension vector $\delta \in [0, \infty)^d$ and hitting times $\tilde{T} = T\, \mathbf{1}_d + \delta$. By definition, the law of $B(t)$ must be identical to the law of the corresponding m-MrB process. Consequently, their covariance functions must be identical for all $s,t \in [0,T]$.

The identity of the covariance functions implies an identity in the information structure. Specifically, the precision gained on the factor vector $X(T)$ by conditioning on the available information must be the same for both processes. For the process $B(t)$, the information comes from the view $Y(0,T)$, it can be shown that the resulting precision gain on $X(T)$ is\footnote{See Subsection \ref{EC:Proof_PrecisionMatrices} for the detailed derivations of the precision gain matrix.} 
\[
\mathbb{V}[X(T) \, | \, Y(0,T)=y]^{-1} - \mathbb{V}[X(T)]^{-1} = P^\top \Omega^{-1} P.
\]
For the m-MrB, the process is defined by conditioning on the exact future vector $X(\tilde{T})$. The precision gained on $X(T)$ from this future observation is, by definition, the precision-gain operator
\[
\mathbb{V}[X(T) \, | \, X(\tilde{T})]^{-1} - \mathbb{V}[X(T)]^{-1} = \mathcal{P}(\delta; \Theta, \Sigma^X).
\]
where \(\mathcal{P}(\delta; \Theta, \Sigma^X)\) is given by \eqref{eq:precision_gain2}. 
Since the laws of the processes are assumed to be identical, their informational structures must coincide. Therefore, the alignment condition must hold:
\[
P^\top \Omega^{-1} P = \mathcal{P}(\delta; \Theta, \Sigma^X).
\]
This proves the condition is necessary.

\paragraph{\textbf{(ii) Sufficiency.}}
Conversely, assume there exists a vector $\delta \in [0, \infty)^d$ such that the alignment condition in \eqref{eq:alignment_condition} holds. Our goal is to show that the process $B(t)$ is an m-MrB by demonstrating that its law coincides with that of  $(X(t) \, | \, X(\tilde{T}) = \tilde{y})$, for a hitting times $\tilde{T} = T\, \mathbf{1}_d + \delta$ and a terminal point $\tilde{y}$.

Since both $B(t)$ and the conditioned process $X(t)$ are Gaussian, it suffices to show that their mean and covariance functions are identical. We will first establish the equality of the covariance functions and then determine the terminal point $\tilde{y}$ that aligns their mean functions.

\textbf{1. Covariance Equivalence}

First, we characterize the covariance of the target m-MrB. Let $\tilde{T} = T\mathbf{1}_d + \delta$. From Proposition \ref{prop:m-MrB_Properties}, the covariance of the m-MrB, conditioned on its value at $\tilde{T}$, is given by
$$
\Cov(X(t), X(s) \, | \, X(\tilde{T})) = \Cov(X(t), X(s)) - V(t) \Gamma_{\tilde{T}}(t) V(\tilde{T})^{-1} \Gamma_{\tilde{T}}(s)^\top V(s).
$$
Let $F(\delta) \in \mathbb{R}^{d \times d}$ be the matrix with the $i$-th column given by $(F(\delta))_i = e^{-\Theta \delta_i} e_i$. It follows from \eqref{ECeq:Coefficients_cov_mMrB} that $\Gamma_{\tilde{T}}(t) = e^{-\Theta^\top (T-t)} F(\delta)^\top$. The terminal variance matrix $V(\tilde{T})$ is given by the Markov property as
\begin{equation}
\label{ECeq:V_tildeT_revised}
    V(\tilde{T}) = V(T\mathbf{1}_d + \delta) = F(\delta) V(T) F(\delta)^{\top} + C(\delta),
\end{equation}
where $C(\delta) := \mathbb{V}[X(\delta) \, | \, X(0)]$ has elements
$$
C(\delta)_{ij} = \int_0^{\min\{\delta_i,\delta_j\}} e_i^\top e^{-\Theta(\delta_i-u)}\,\Sigma^X\,e^{-\Theta^\top(\delta_j-u)}e_j\,du.
$$
Substituting these expressions and simplifying yields the m-MrB covariance
\begin{equation}
\label{ECeq:Cov_m-MrB_simplified_revised}
    \Cov(X(t), X(s) \, | \, X(\tilde{T})) = \Cov(X(t), X(s)) - V(t) e^{-\Theta^\top(T-t)} \left(V(T)^{-1} + \mathcal{P}(\delta; \Theta, \Sigma^X) \right)^{-1} e^{-\Theta(T-s)} V(s),
\end{equation}
where $\mathcal{P}(\delta; \Sigma^X, \Theta) = F(\delta)^\top C(\delta)^{-1} F(\delta)$.

Next, we consider the process $B(t)$. By definition, its covariance is
\begin{equation}
\label{ECeq:Cov_B_m-MrB_simplified_revised}
\begin{split}
    \Cov(B(t), B(s)) &= \Cov(X(t), X(s) \, | \, Y(0,T) = y) \\
    &= \Cov(X(t), X(s)) - \Cov(X(t), Y(0,T)) \mathbb{V}[Y(0,T)]^{-1} \Cov(Y(0,T), X(s)) \\
    &= \Cov(X(t), X(s)) - \Cov(X(t), X(T)) P^\top \mathbb{V}[Y(0,T)]^{-1} P \Cov(X(T), X(s)) \\
    &\stackrel{(a)}{=} \Cov(X(t), X(s)) - V(t) e^{-\Theta^\top(T-t)} \left(V(T)^{-1} + P^\top\Omega^{-1}P \right)^{-1} e^{-\Theta(T-s)} V(s),
\end{split}
\end{equation}
where step (a) uses the expression for the precision gain matrix from Subsection \ref{EC:Proof_PrecisionMatrices}.

By our initial assumption, the alignment condition \eqref{eq:alignment_condition} holds, which states $P^\top \Omega^{-1}P = \mathcal{P}(\delta; \Sigma^X, \Theta)$. Comparing \eqref{ECeq:Cov_m-MrB_simplified_revised} and \eqref{ECeq:Cov_B_m-MrB_simplified_revised}, we see that the covariance functions are identical. This implies that the centered processes, $B(t) - \mathbb{E}[B(t)]$ and $(X(t) \, | \, X(\tilde{T})) - \mathbb{E}[X(t) \, | \, X(\tilde{T})]$, are equal in law. Since $B(t) - \mathbb{E}[B(t)]$  has zero mean, it follows by Proposition \ref{prop:m-MrB_Properties} that it is a m-MrB from \(\mathbf{0}\) to \(\mathbf{0}\) with hitting times \(\tilde{T}\).

\textbf{2. Mean Equivalence}

It remains to show that when \(P\) has full rank, there exist a unique terminal point $\tilde{y}$ such that the mean functions also coincide. From Proposition \ref{prop:m-MrB_Properties}, the mean of the m-MrB is
\begin{equation}
\label{ECeq:Expectation_m-MrB_simplified_revised}
    \mathbb{E}[X(t) \, | \, X(\tilde{T}) = \tilde y] = e^{-\Theta t}a + V(t) e^{-\Theta^\top (T-t)} \left(V(T)^{-1} + \mathcal{P}(\delta; \Sigma^X, \Theta) \right)^{-1} F(\delta)^\top C(\delta)^{-1} \left(\tilde y - F(\delta) \mathbb{E}[X(T)]\right).
\end{equation}
From Proposition \ref{prop:X^y}, it can be shown that the mean of $B(t)$ is
\begin{equation}
\label{ECeq:Expectation_B_m-MrB_simplified_revised}
    \mathbb{E}[B(t)] =  e^{-\Theta t}a + V(t) e^{-\Theta^\top (T-t)} \left(V(T)^{-1} + P^\top \Omega^{-1} P \right)^{-1} P^\top \Omega^{-1} \left(y - P\mathbb{E}[X(T)]\right).
\end{equation}
 The two expressions for the mean become identical if we choose $\tilde{y}$ such that
$$
F(\delta)^\top C(\delta)^{-1} \left(\tilde y - F(\delta) \mathbb{E}[X(T)]\right) = P^\top \Omega^{-1} \left(y - P\mathbb{E}[X(T)]\right).
$$
Given the alignment condition, we get
\begin{equation}
\label{ECeq:equalmeans_mMrB}
F(\delta)^\top C(\delta)^{-1} \tilde{y} = P^\top \Omega^{-1} y.
\end{equation}
When \(P\) has full rank and \(\Omega\ succ 0\), \(P \Omega^{-1} P \succ 0\). It follows from the alignment condition \eqref{eq:alignment_condition} that \(F(\delta)^{-1} C(\delta)^{-1} F(\delta) \succ 0 \). Therefore, \(F(\delta)\) is also invertible and there exists a unique target vector \(\tilde y\) that solves \eqref{ECeq:equalmeans_mMrB}, with
$$
\tilde{y} =  C(\delta)F(\delta)^{-\top}P^\top \Omega^{-1} y.
$$
With this choice of $\tilde{y}$, the mean functions of the two processes are identical.

Therefore, since both $B(t)$ and $(X(t) \, | \, X(\tilde{T}) = \tilde{y})$ are Gaussian processes, and we have shown that their covariance and mean functions are identical for the specified $\tilde{y}$ and $\tilde{T}$, the processes coincide in law:
$$
B(t) \stackrel{d}{=} (X(t) \, | \, X_1(T + \delta_1) = \tilde{y}_1, \ldots, X_d(T + \delta_d) = \tilde y_d).
$$
Therefore, $B(t)$ is an m-MrB from $a$ to $\tilde{y}$ with hitting times $\tilde{T}$, which completes the sufficiency argument.
\begin{remark}
    When $P$ does not have full rank, the matrix $P^\top \Omega^{-1}P$ is singular. While this may prevent the existence of a unique terminal point $\tilde{y}$ that aligns the mean functions for every given $y$, the alignment of the covariance structures in \eqref{ECeq:Cov_m-MrB_simplified_revised} and \eqref{ECeq:Cov_B_m-MrB_simplified_revised} is unaffected. Consequently, the centered process $B(t) - \mathbb{E}[B(t)]$ is still a zero-mean m-MrB (from $\mathbf{0}$ to $\mathbf{0}$) with hitting times $\tilde{T} = T\mathbf{1}_d + \delta$.
\end{remark}

\subsubsection{Proof of the Information Gain Matrix from Noisy Views}
\label{EC:Proof_PrecisionMatrices}
We prove that the precision gained on $X(T)$ from observing $Y(0,T)=y$ is given by $P^\top \Omega^{-1} P$. The precision gain is defined as $\mathbb{V}[X(T) \, | \, Y(0,T)]^{-1} - \mathbb{V}[X(T)]^{-1}$.

Let $X \equiv X(T)$ and $Y \equiv Y(0,T) = PX + \epsilon$. The vectors $X$ and $Y$ are jointly Gaussian. The conditional covariance matrix is given by
\begin{equation}
\label{eq:gauss_var}
\mathbb{V}[X|Y] = \mathbb{V}[X] - \mathrm{cov}(X,Y)\mathbb{V}[Y]^{-1}\mathrm{cov}(Y,X).
\end{equation}
Let $V_X := \mathbb{V}[X(T)]$. The covariance and variance terms are
\begin{align*}
\mathrm{cov}(X,Y) &= \mathrm{cov}(X, PX + \epsilon) = \mathrm{cov}(X,PX) + \mathrm{cov}(X,\epsilon) = V_X P^\top \\
\mathbb{V}[Y] &= \mathbb{V}(PX + \epsilon) = \mathbb{V}(PX) + \mathbb{V}(\epsilon) = P V_X P^\top + \Omega
\end{align*}
by substituting into \eqref{eq:gauss_var}, we get
\[
\mathbb{V}[X|Y] = V_X - (V_X P^\top)(P V_X P^\top + \Omega)^{-1}(P V_X).
\]
To find the inverse of this expression, we use the Woodbury matrix identity: $(A-BD^{-1}C)^{-1} = A^{-1} + A^{-1}B(D - CA^{-1}B)^{-1}CA^{-1}$. Let $A=V_X$, $B=V_X P^\top$, $C=P V_X$, and $D = P V_X P^\top + \Omega$. The inner term of the identity simplifies to
\[
D - C A^{-1} B = (P V_X P^\top + \Omega) - (P V_X)(V_X^{-1})(V_X P^\top) = \Omega.
\]
Applying the identity yields the conditional precision
\begin{align*}
\mathbb{V}[X|Y]^{-1} &= V_X^{-1} + V_X^{-1}(V_X P^\top)(\Omega)^{-1}(P V_X)V_X^{-1} \\
&= V_X^{-1} + P^\top \Omega^{-1} P.
\end{align*}
Rearranging gives the precision gain
\[
\mathbb{V}[X|Y]^{-1} - \mathbb{V}[X]^{-1} = P^\top \Omega^{-1} P.
\]

\newpage

\section{Section \ref{sec5}}

\subsection{Proof of the Value Function Characterization in \ref{sec51}}
\label{App:ProofHJB}

Given observed risk factors $x \in \mathbb{R}^d$, views vector \(y \in \mathbb{R}^K\), and  a level of wealth $z \in \mathbb{R}$ at time $t$, the investor's value function is
\begin{equation*}
    V(t,z,x) = \max_{\pi \in \mathcal{A}} \mathbb{E}_\mathbb{Q}\big[ U(Z(T)) \,|\, X^y(t) = x, \, Y(0,T) = y \big],
\end{equation*}
where \(\mathbb Q = \mathbb P( . \, | \, Y(0,T)  = y)\) is the posterior probability measure,
\begin{equation*}
    U(Z(T)) = \frac{Z(T)^{1 - \gamma}}{1 - \gamma}
\end{equation*}
is the investor's utility at the end of the investment horizon $T$, and $\gamma$ is her risk aversion ($\gamma \geq 0 $ and $\gamma \neq 1$). By Itô's Lemma, the dynamics of the value function are
\begin{equation}
\label{eq:EC_Itos_vf}
    \begin{split}
        dV(t,Z(t), X^y(t)) =& \frac{\partial V}{\partial t} dt + \nabla_z V dZ(t) + \frac{1}{2} \nabla_z^2 V (dZ(t))^2 + (\nabla_x V)^\top dX^y(t) +  \frac{1}{2} (dX^y(t))^\top (\nabla_x^2 V )dX^y(t)\\
        & + (dX^y(t))^\top \nabla_{z,x}^2 V dZ(t).
     \end{split}
\end{equation}
Under the assumption that the market is self-financing, the wealth process satisfies
\begin{equation*}
    dZ(t) = Z(t) \bigg( r_f dt + \pi(t)^\top \big(\alpha + \beta X^y(t) - r_f \mathbf{1}_N\big) dt + \pi(t)^\top L^S dW^\mathbb{Q}(t) \bigg),
\end{equation*}
it follows that
\begin{equation*}
    (dZ(t))^2 = \pi^\top(t) \Sigma^S \pi(t) Z^2(t) dt.
\end{equation*}
The conditional risk factors dynamics are given by \eqref{eq:dX^y} and can be  expressed as
\begin{equation*}
                dX^y(t) =\tilde{\Theta}(t)(\tilde{\mu}(t,y) - X^y(t)) dt + L^X dW^\mathbb{Q}(t).
\end{equation*}
Therefore, 
\begin{equation*}
    \begin{split}
        dX^y(t) (dX^y(t))^\top &= \Sigma^X dt,\\
        dZ(t)\, dX^y(t) &= Z(t) L^X (L^S)^\top \pi(t).
    \end{split}
\end{equation*}
Thus, by substituting the above dynamics in \eqref{eq:EC_Itos_vf}, we express the HJB as
\begin{equation}
\label{eq:EC_HJB}
\begin{split}
    \max_{\pi} \Big\{&\frac{\partial V}{\partial t} +  z \big( r_f + \pi(t)^\top(\alpha + \beta x - r_f \mathbf{1}_N \big) \nabla_z V
  +  \tilde{\Theta}(t)(\tilde{\mu}(t,y) - x)^\top \nabla_x V    + \frac{1}{2} z^2 \pi(t)^\top \Sigma^S \pi(t)  \nabla^2_z V \\ & + \frac{1}{2} \Tr(\Sigma^X \nabla_x^2 V)  + z \pi(t)^\top \Sigma^{S,X} \nabla^2_{x,z} V\Big\} = 0
\end{split}
\end{equation}
The HJB is concave in $\pi(t)$ for each $t \in [0,T]$, therefore, the optimal investment strategy $\pi^*(t)$ can be directly derived by taking the first order derivative with respect to $\pi$
\begin{equation*}
    (\mu^S(t,x,y) - r_f \mathbf{1}_N) \nabla_z V + z^2 \pi^*(t)\Sigma^S \nabla_z^2 V + z \Sigma^{S,X} \nabla_{x,z}^2 V = 0,
\end{equation*}
the optimal solution is then
\begin{equation}
\label{eq:EC_pi_ori}
\begin{split}
        \pi^*(t) &= - \frac{1}{z^2 \nabla_Z^2 V} (\Sigma^S)^{-1} \big(\nabla_z V   (\alpha + \beta x - r_f \mathbf{1}_N) + z \Sigma^{S,X} \nabla^2_{x,z}V\big)\\
        &= - \frac{\nabla_z V}{ z \nabla_z^2 V} (\Sigma^S)^{-1}(\alpha + \beta x - r_f \mathbf{1}_N \big) - (\Sigma^S)^{-1}\Sigma^{S,X}\frac{ \nabla^2_{x,z}V}{ z \nabla_z^2 V}.
\end{split}
\end{equation}
We now give an explicit characterization for the value function and the optimal policy. For a CRRA utility, the value function takes the form
\begin{equation}
\label{eq:EC_VF_form}
    V(t,z,x) = \frac{z^{1-\gamma}}{1-\gamma}\exp\left(\frac{1}{2}x^\top A(t) x + x^\top b(t) + c(t)\right),
\end{equation}
where $A(t) \in \mathbb{R}^{N \times N}$, $b(t) \in \mathbb{R}^{N}$, and $c(t) \in \mathbb{R}$ are derived from a system of ODEs, which we determine next. 

From the structure \eqref{eq:EC_VF_form}, we have
\begin{equation*}
    \begin{cases}
        \dfrac{\partial V}{\partial t} &=  (\dfrac{1}{2}x^\top A'(t) x + x^\top b'(t) + c'(t)) V,\\
        \nabla_z V &= \dfrac{1 - \gamma}{z}V,\\
        \nabla_x V &= (A(t) x + b(t)) V, \\
        \nabla_z^2V &= \dfrac{-\gamma(1 - \gamma)}{z^2}V,\\
        \nabla_x^2 V &= (A(t) + (A(t)x + b(t))(A(t)x+b(t))^\top)V,\\
        \nabla_{x,z}^2 V &=  \dfrac{1-\gamma}{z}(A(t)x + b(t))V,\\
    \end{cases}
\end{equation*}
Therefore, \eqref{eq:EC_pi_ori} becomes
\begin{equation*}
    \pi^*(t) = \frac{1}{\gamma} (\Sigma^S)^{-1} \left(\alpha + \beta x -r_f\mathbf{1}_N \right) + \frac{1}{\gamma}(\Sigma^S)^{-1}\Sigma^{S,X} \left(A(t) x + b(t) \right),
\end{equation*}
and the HJB \eqref{eq:EC_HJB} can be written as
\begin{equation}
\label{eq:EC_HJB2}
    \begin{split}
        &\frac{1}{2}x^\top A'(t) x + x^\top b'(t) + c'(t) + (1-\gamma)r_f \\+& \frac{1-\gamma}{\gamma} \big((\Sigma^S)^{-1} (\alpha + \beta x - r_f \mathbf{1}_N) + (\Sigma^S)^{-1} \Sigma^{S,X}(A(t)x + b(t)) \big)^\top  (\alpha + \beta x - r_f \mathbf{1}_N)\\
        + & \mu^X(t,x,y)^\top (A(t) x + b(t))\\
        - & \frac{1 - \gamma}{2\gamma}\big((\Sigma^S)^{-1} (\alpha + \beta x - r_f \mathbf{1}_N) + (\Sigma^S)^{-1}\Sigma^{S,X}( A(t)x + b(t)) \big)^\top \big((\tilde{\Theta}(t)\tilde{\mu}(t,y)  - r_f \mathbf{1}_N) + \Sigma^{S,X}(A(t)x + b(t)) \big)\\
        + & \frac{1-\gamma}{\gamma}\big((\Sigma^S)^{-1} (\tilde{\Theta}(t)\tilde{\mu}(t,y) - r_f \mathbf{1}_N) + (\Sigma^S)^{-1} \Sigma^{S,X}( A(t)x + b(t))\big)^\top\Sigma^{S,X} (A(t)x + b(t))\\
        + & \frac{1}{2} \Tr\big(A(t) \Sigma^X + (A(t)x + b(t))^\top \Sigma^X (A(t)x + b(t)) \big) = 0.
    \end{split}
\end{equation}
Now we derive the system of ODEs by separation of variables. Recall that
    \begin{equation*}
    \begin{split}
        &\tilde{\Theta}(t) = \Theta + L^X \eta(t) P e^{-\Theta (T-t)},\\
       & \tilde{\mu}(t,y) =  \mu + \tilde{\Theta}(t)^{-1} L^X \eta(t) y,\\
       & \eta(t) = (P e^{-\Theta (T-t)} L^X)^\top \big(P (\Sigma - e^{-\Theta(T-t)} \Sigma e^{-\Theta^\top (T-t)})P^\top + \Omega\big)^{-1},
            \end{split}
    \end{equation*}
and
 \begin{equation*}
        \begin{split}
            &\tilde{\alpha}(t,y) = \alpha + L^S \eta(t) (y- P ( 1-e^{-\Theta (T-t)})\mu_,  \\
            &\tilde{\beta}(t) = \beta -  L^S \eta(t) P e^{-\Theta (T-t)}, \\
        \end{split}
    \end{equation*}
    Thus, by separation of variables, the HJB \eqref{eq:EC_HJB2} can be split into the following ODEs:
    \begin{equation}
    \label{eq:EC_A_Ricatti}
    \begin{split}
                A'(t) &+ \frac{1-\gamma}{\gamma} \tilde{\beta}(t)^\top (\Sigma^S)^{-1} \tilde{\beta}(t) + A(t) \big(\Sigma^X + \frac{1-\gamma}{\gamma} (\Sigma^{S,X})^\top (\Sigma^S)^{-1} \Sigma^{S,X}\big) A(t) \\
        &+ A(t) \big(\frac{1-\gamma}{\gamma}(\Sigma^{S,X})^\top (\Sigma^S)^{-1} \tilde{\beta}(t) - \tilde{\Theta}(t)\big) + \big(\frac{1-\gamma}{\gamma}(\Sigma^{S,X})^\top (\Sigma^S)^{-1} \tilde{\beta}(t) - \tilde{\Theta}(t)\big)^\top A(t) = 0,
        \end{split}
\end{equation}
$b(t) \in \mathbb{R}^d$ solves a system of linear ODEs
\begin{equation*}
    \begin{split}
        b'(t) &+ \frac{1-\gamma}{\gamma} \big( \tilde{\beta}(t)^\top + A(t) (\Sigma^{S,X})^\top\big) (\Sigma^S)^{-1} \big( \Sigma^{S,X} b(t) + \tilde{\alpha}(t,y)  - r_f \mathbf{1}_N \big) \\ &+ \big(A(t) \Sigma^X - \tilde{\Theta} (t)^\top\big) b(t) + A(t)\tilde{\Theta}(t) \tilde{\mu}(t,y) = 0,
    \end{split}
\end{equation*}
and $c(t) \in \mathbb{R}$ is given by
\begin{equation*}
    \begin{split}
        c'(t) &+(1-\gamma)r_f + \frac{1}{2}\Tr(\Sigma^X A(t)) + \tilde{\Theta}(t)\tilde{\mu}(t,y)^\top b(t) + \frac{1}{2}b(t)^\top \Sigma^X b(t)\\
        &+ \frac{1-\gamma}{2\gamma}\big(\Sigma^{S,X} b(t) + \tilde{\alpha}(t,y)  - r_f \mathbf{1}_N \big)^\top (\Sigma^S)^{-1}\big(\Sigma^{S,X} b(t) + \tilde{\alpha}(t,y)- r_f \mathbf{1}_N \big) = 0.
    \end{split}
\end{equation*}
Additionally, at time $T$ we have
\begin{equation*}
    V(T,z,x) = U(z), \, \, \forall x
\end{equation*}
thus, the terminal conditions of the differential equations are $A(T) = 0$, $b(T) = 0$, and $c(T) = 0$.



Given the terminal condition \( A(T) = 0 \), we aim to prove that \( A(t) \), which satisfies \eqref{eq:EC_A_Ricatti}, is symmetric and negative semi-definite.  

Since \( A(T) = 0 \) is symmetric, and the matrices \( \Sigma^S \) and \( \tilde{\beta}(t)^\top (\Sigma^S)^{-1} \tilde{\beta}(t) \) are also symmetric, it follows that \( A(t) \) remains symmetric for all \( t \in [0,T] \).  

Furthermore, since \( \Sigma^S \) is positive definite, we have that \( \tilde{\beta}(t)^\top (\Sigma^S)^{-1} \tilde{\beta}(t) \) is positive semi-definite. Consequently,  
\begin{equation*}
    \frac{1-\gamma}{\gamma} \tilde{\beta}(t)^\top (\Sigma^S)^{-1} \tilde{\beta}(t)
\end{equation*}  
is negative semi-definite for \( \gamma > 1 \). Moreover
\begin{equation*}
\begin{split}
    \Sigma^X + \frac{1-\gamma}{\gamma} (\Sigma^{S,X})^\top (\Sigma^S)^{-1} \Sigma^{S,X} &= \Sigma^X + \frac{1-\gamma}{\gamma} L^X (L^S)^\top (L^S (L^S)^\top)^{-1} L^S (L^X)^\top  \\
    &= \Sigma^X + \frac{1-\gamma}{\gamma} L^X \Pi_S (L^X)^\top 
    \end{split}
\end{equation*}
where \(\Pi_S := (L^S)^\top (L^S (L^S)^\top)^{-1} L^S \). By definition \( 0 \preceq \Pi_S \preceq I_N\), it follows that for \(\gamma > 1\) (so \(\frac{1-\gamma}{\gamma} < 0\)) we have
\begin{equation}
\label{ECeq:nsd_proof}
     \Sigma^X + \frac{1-\gamma}{\gamma} (\Sigma^{S,X})^\top (\Sigma^S)^{-1} \Sigma^{S,X} \succeq \Sigma^X + \frac{1-\gamma}{\gamma} \Sigma^X  \succeq 0.
\end{equation}
Additionally, the terminal condition \( A(T) = 0 \) is itself negative semi-definite. By the result in \cite{Ricatti}, it follows that \( A(t) \) is negative semi-definite for \( \gamma > 1 \) (and positive semi-definite for \( \gamma < 1 \)).  

To make this more explicit, we can rewrite \eqref{eq:EC_A_Ricatti} as  
\begin{equation*}
    \lim_{dt \to 0} A(T) - A(T-dt) = - \frac{1-\gamma}{\gamma}\tilde{\beta}(t)^\top (\Sigma^S)^{-1} \tilde{\beta}(t) \succeq 0.
\end{equation*}  
Since \( A(T) = 0 \), it follows that \( A(T - dt) \) is negative semi-definite. Applying the same reasoning, we conclude that \( A(t) \) remains negative semi-definite for all \( t \in [0,T] \).  

This completes the proof. \hfill \( \square \)  

\subsection{Proof of Proposition \ref{prop:Views_are_good} }
\label{App:Proof-prop-views-good}
Let $\Pi_0$ be the class of admissible policies adapted to the natural filtration $\mathcal{F}_t = \sigma(W_s; s \le t)$, and let $\Pi$ be the class of admissible policies adapted to the enlarged filtration $\mathcal{F}_t^Y = \sigma(\mathcal{F}_t \vee \sigma(Y(0,T)))$. An essential observation is that any policy admissible in the no-views setting is also admissible in the with-views setting (i.e., $\Pi_0 \subset \Pi$).

The value function for an investor without views $V_0(t, z, x)$ is defined as the maximum expected utility achievable using policies from $\Pi_0$
$$
V_0(t, z, x) = \max_{\pi \in \Pi_0} \mathbb{E}_{\mathbb{P}}[U(Z(T)) | Z(t)=z, X(t)=x]
$$
Let $\pi_0^*(t)$ be the optimal policy for this problem. By definition, the value function is
$$
V_0(t, z, x) = \mathbb{E}_{\mathbb{P}}[U(Z^{\pi_0^*}(T)) | Z(t)=z, X(t)=x]
$$
where $Z^{\pi_0^*}$ denotes the wealth process when following the policy $\pi_0^*$.

The value function for an investor with a specific view realization $Y(0,T)=y$, which we here write as $V(t, z, x; y)$, is defined as
$$
V(t, z, x; y) = \max_{\pi \in \Pi} \mathbb{E}_{\mathbb{Q}}[U(Z(T)) \,|\, Z(t)=z, X^y(t)=x]
$$
where $\mathbb{Q} = \mathbb{P}(\cdot \,|\, Y(0,T)=y)$. Let $\pi^*(t;y)$ be the optimal policy for this problem.

Since the no-views optimal policy $\pi_0^*$ is adapted to $\mathcal{F}_t$, it is also adapted to the larger filtration $\mathcal{F}_t^Y$. Thus, $\pi_0^*$ is an admissible, though possibly suboptimal, policy for the investor who has observed the forward-looking views $Y(0,T) = y$. By the definition of $V(t, z, x; y)$, it must be greater than or equal to the value obtained from any other admissible policy, including $\pi_0^*$. For any given $y$, we have
$$
V(t, z, x; y) \ge \mathbb{E}_{\mathbb{Q}}[U(Z^{\pi_0^*}(T)) | Z(t)=z, X^y(t)=x]
$$
This inequality holds for every possible realization $y$ of the view vector $Y(0,T)$. We now take the expectation of both sides with respect to the distribution of $Y(0,T)$, conditional on the state at time $t$. Let $\mathbb{E}_y[\cdot | X(t)=x, Z(t)=z]$ denote this expectation.
$$
\mathbb{E}_y[V(t, z, x; Y) | X(t)=x, Z(t)=z] \ge \mathbb{E}_y \left[ \mathbb{E}_{\mathbb{Q}}[U(Z^{\pi_0^*}(T)) | Z(t)=z, X^y(t)=x] \Big| X(t)=x, Z(t)=z \right]
$$
The right-hand side is an iterated expectation. Using the law of total expectation, we can simplify its expression
\begin{align*}
\text{RHS} &= \mathbb{E}_y \left[ \mathbb{E}_{\mathbb{P}}[U(Z^{\pi_0^*}(T)) | Y(0,T)=y, Z(t)=z, X(t)=x] \Big| X(t)=x, Z(t)=z \right] \\
&= \mathbb{E}_{\mathbb{P}}[U(Z^{\pi_0^*}(T)) | Z(t)=z, X(t)=x]\\
&= V_0(t,z,x).
\end{align*}
The final expression follows from the definition of the value function without views. Therefore, we have shown
$$
\mathbb{E}_y[V(t, z, x; Y) | X(t)=x, Z(t)=z] \ge V_0(t, z, x)
$$
which completes the proof.
\hfill$\square$



\subsection{Proof of Theorem \ref{thm:Optimal_policy} }
\label{App:Proof-thm-sec5}
We start by introducing the following Lemma 
\begin{lemma}
\label{lemma:separation_ODEs}
     Consider the investor's hedging demand in \eqref{eq:optimalPolicy}, where $A(t)$ and $b(t)$ satisfy  \eqref{eq:Riccati_views} and \eqref{eq:ODEsys_views}, respectively. Then, we have
     \begin{equation*}
              \begin{cases}
         A(t) = A_1(t) + \hat{A}(t)\\
         b(t) = b_1(t) + \hat{b}(t),\\
     \end{cases}
     \end{equation*}
where $A_1(t)$ is a symmetric negative definite matrix satisfying the Riccati equation 
\begin{equation}
\label{eq:EC_Ricatti_1_ori}
\begin{split}
        A_1'(t) &+ \frac{1-\gamma}{\gamma} \beta^\top (\Sigma^S)^{-1} \beta(t) + A_1(t) \big(\Sigma^X + \frac{1-\gamma}{\gamma} (\Sigma^{S,X})^\top (\Sigma^S)^{-1} \Sigma^{S,X}\big) A_1(t) \\
        &+ A_1(t) \big(\frac{1-\gamma}{\gamma}(\Sigma^{S,X})^\top (\Sigma^S)^{-1} \beta - \Theta\big) + \big(\frac{1-\gamma}{\gamma}(\Sigma^{S,X})^\top (\Sigma^S)^{-1} \beta - \Theta\big)^\top A_1(t) = 0,
        \end{split}
\end{equation}
and $b_1(t)$ solves the ODEs system 
\begin{equation}
\label{eq:EC_ODE_1_ori}
\begin{split}
               b_1'(t) &+ \frac{1-\gamma}{\gamma} \big( \beta^\top + A_1(t) (\Sigma^{S,X})^\top\big) (\Sigma^S)^{-1} \big( \Sigma^{S,X} b_1(t) + \alpha - r_f \mathbf{1}_N \big) \\&+ \big(A_1(t) \Sigma^X - \Theta^\top\big) b_1(t) + A_1(t) \Theta \mu = 0,
        \end{split}
\end{equation} with terminal conditions
\begin{equation*}
     \begin{cases}
     A_1(T) = -P^\top \Omega^{-1} P\\
     b_1(T) = P^\top \Omega^{-1} y.\\
 \end{cases}
\end{equation*}
The matrix $\hat{A}(t)$ is symmetric positive semi-definite   
\begin{equation*}
    \hat{A}(t) =(P e^{-\Theta (T-t)})^\top \big( P (\Sigma - e^{-\Theta (T-t)} \Sigma e^{-\Theta^\top (T-t)}) P^\top + \Omega \big)^{-1} P e^{-\Theta (T-t)},
\end{equation*}
and $\hat{b}(t)$ can be expressed as
\begin{equation*}
        \hat{b}(t) = (P e^{-\Theta (T-t)})^\top \big( P (\Sigma - e^{-\Theta (T-t)} \Sigma e^{-\Theta^\top (T-t)}) P^\top + \Omega \big)^{-1} \big( P (1 - e^{-\Theta (T-t)} ) \mu - y \big).
\end{equation*}
\end{lemma}

From Lemma \ref{lemma:separation_ODEs}, we have that
\begin{equation*}
      \begin{cases}
         A(t) = A_1(t) + \hat{A}(t)\\
         b(t) = b_1(t) + \hat{b}(t),\\
     \end{cases}
     \end{equation*}
where $A_1(t)$ the Riccati equation \eqref{eq:EC_Ricatti_1_ori}, and $b_1(t)$ solves the ODEs system \eqref{eq:EC_ODE_1_ori}, with terminal conditions
\begin{equation*}
     \begin{cases}
     A_1(T) = -P^\top \Omega^{-1} P\\
     b_1(T) = P^\top \Omega^{-1} y.\\
 \end{cases}
\end{equation*}
The matrix $\hat{A}(t)$ is symmetric with
\begin{equation*}
    \hat{A}(t) =(P e^{-\Theta (T-t)})^\top \big( P (\Sigma - e^{-\Theta (T-t)} \Sigma e^{-\Theta^\top (T-t)}) P^\top + \Omega \big)^{-1} P e^{-\Theta (T-t)},
\end{equation*}
and $\hat{b}(t)$ can be expressed as
\begin{equation*}
        \hat{b}(t) = (P e^{-\Theta (T-t)})^\top \big( P (\Sigma - e^{-\Theta (T-t)} \Sigma e^{-\Theta^\top (T-t)}) P^\top + \Omega \big)^{-1} \big( P (1 - e^{-\Theta (T-t)} ) \mu - y \big).
\end{equation*}
It follows that the hedging demand in \eqref{eq:optimalPolicy} takes the form 
\begin{equation*}
\begin{split}
     \frac{1}{\gamma}  (\Sigma^S)^{-1}\Sigma^{S,X}\frac{\partial g}{\partial x}(t,x) &=  \frac{1}{\gamma}  (\Sigma^S)^{-1}\Sigma^{S,X} \left(A(t) x + b(t)\right)\\
     &= \frac{1}{\gamma}  (\Sigma^S)^{-1}\Sigma^{S,X} \left((A_1(t) + \hat{A}(t)) x + b_1(t) + \hat{b}(t)\right)\\
     &= \frac{1}{\gamma} (\Sigma^S)^{-1}\Sigma^{S,X} \left(\hat{A}(t) x + \hat{b}(t) \right) + \frac{1}{\gamma}  (\Sigma^S)^{-1}\Sigma^{S,X} \left(A_1(t) x + b_1(t)\right).
     \end{split}
\end{equation*}
We now prove that 
\begin{equation*}
    \frac{1}{\gamma} (\Sigma^S)^{-1}\Sigma^{S,X} \left(\hat{A}(t) x + \hat{b}(t) \right) = \frac{1}{\gamma} (\Sigma^S)^{-1} \left(\alpha + \beta x - (\tilde{\alpha}(t,y) + \tilde{\beta}(t) x)\right).
\end{equation*}

We have 
\begin{equation*}
    \begin{split}
        \frac{1}{\gamma} (\Sigma^S)^{-1}\Sigma^{S,X} \left(\hat{A}(t) x + \hat{b}(t) \right) &= \frac{1}{\gamma} (\Sigma^S)^{-1} L^S (L^X)^\top \left(\hat{A}(t) x + \hat{b}(t) \right)\\
        &= \frac{1}{\gamma} (\Sigma^S)^{-1} L^S \left( \eta(t) P e^{-\Theta (T-t)} x + \eta(t) \left(P (1 - e^{-\Theta (T-t)} \mu - y \right)\right),
    \end{split}
\end{equation*}
where the second equality follows from the fact that
\begin{equation*}
    \begin{split}
        (L^X)^\top \hat{A}(t) &= \eta(t) P e^{-\Theta (T-t)},\\
        (L^X)^\top \hat{b}(t) &= \eta(t) \left(P (1 - e^{-\Theta (T-t)}) \mu - y \right).
    \end{split}
\end{equation*}
By definition, we have
\begin{equation*}
        \begin{cases}
            &\tilde{\alpha}(t) = \alpha - L^S \eta(t) P ( 1-e^{-\Theta (T-t)})\mu  \\
            &\tilde{\beta}(t) = \beta -  L^S \eta(t) P e^{-\Theta (T-t)} \\
            &\tilde{\gamma}(t) = L^S \eta(t)\\
        \end{cases}
    \end{equation*}
Thus,
\begin{equation*}
    \eta(t) P e^{-\Theta (T-t)} x + \eta(t) \left(P (1 - e^{-\Theta (T-t)} \mu - y \right) = \alpha + \beta x - (\tilde{\alpha}(t,y) + \tilde{\beta}(t) x),
\end{equation*}
and
\begin{equation*}
        \frac{1}{\gamma} (\Sigma^S)^{-1}\Sigma^{S,X} \left(\hat{A}(t) x + \hat{b}(t) \right) = \frac{1}{\gamma} (\Sigma^S)^{-1} \left(\alpha + \beta x - (\tilde{\alpha}(t,y) + \tilde{\beta}(t) x)\right).
\end{equation*}
Therefore,
\begin{equation*}
    \begin{split}
        \frac{1}{\gamma}  (\Sigma^S)^{-1}\Sigma^{S,X}\frac{\partial g}{\partial x}(t,x) &=  \frac{1}{\gamma} (\Sigma^S)^{-1}\Sigma^{S,X} \left(\hat{A}(t) x + \hat{b}(t) \right) + \frac{1}{\gamma}  (\Sigma^S)^{-1}\Sigma^{S,X} \left(A_1(t) x + b_1(t)\right)\\
        &= \frac{1}{\gamma} (\Sigma^S)^{-1} \left(\alpha + \beta x - (\tilde{\alpha}(t,y) + \tilde{\beta}(t) x)\right) + \frac{1}{\gamma}  (\Sigma^S)^{-1}\Sigma^{S,X} \left(A_1(t) x + b_1(t)\right),
    \end{split}
\end{equation*}
and it follows that the optimal policy \eqref{eq:optimalPolicy} takes the form
\begin{equation*}
    \begin{split}
        \pi^*(t) &= \frac{1}{\gamma}(\tilde{\alpha}(t,y) + \tilde{\beta}(t) x - r_f \mathbf{1}_N) + \frac{1}{\gamma}  (\Sigma^S)^{-1}\Sigma^{S,X}\frac{\partial g}{\partial x}(t,x)\\
        &= \frac{1}{\gamma}(\alpha + \beta x - r_f \mathbf{1}_N) + \frac{1}{\gamma}  (\Sigma^S)^{-1}\Sigma^{S,X}(A_1(t)x + b_1(t)).
    \end{split}
\end{equation*}

We now prove that the matrix $Q(t) := A_1(t) - A_0(t)$ is negative semi-definite ($Q(t) \preceq 0$) and that it is a decreasing function of the view precision matrix $\Omega^{-1}$. Both $A_1(t)$ and $A_0(t)$ are solutions to the same constant-coefficient Riccati equation given in \eqref{eq:Riccati_no_views}, but with different terminal conditions.

\paragraph{1. Proof that $A_1(t) - A_0(t)$ is Negative Semi-Definite.}
Let $Q(t) = A_1(t) - A_0(t)$. We first analyze the terminal condition of $Q(t)$ at $t=T$.
\begin{itemize}
    \item The terminal condition for the no-views case is $A_0(T) = 0$.
    \item The terminal condition for the views case is $A_1(T) = -P^{\top}\Omega^{-1}P$.
\end{itemize}
Therefore, the terminal condition for their difference is:
\[
Q(T) = A_1(T) - A_0(T) = -P^{\top}\Omega^{-1}P.
\]
By Assumption \ref{ass:nondeg}, the view covariance matrix $\Omega$ is positive definite, which implies its inverse, the precision matrix $\Omega^{-1}$, is also positive definite. For any matrix $P$, the product $P^{\top}\Omega^{-1}P$ is positive semi-definite. Consequently, $Q(T) = -(P^{\top}\Omega^{-1}P)$ is negative semi-definite.

Additionally, by subtracting the Riccati equation for $A_0(t)$ from that of $A_1(t)$, we find that $Q(t)$ satisfies a homogeneous Riccati differential equation
\begin{equation}
\label{eq:EC_Q(t)_riccati}
\begin{cases}
        Q'(t) & + Q(t) \big(\Sigma^X + \frac{1-\gamma}{\gamma} (\Sigma^{S,X})^\top (\Sigma^S)^{-1} \Sigma^{S,X}\big) Q(t)
        + Q(t) \big(\frac{1-\gamma}{\gamma}(\Sigma^{S,X})^\top (\Sigma^S)^{-1} (\Sigma^{S,X} A_0(t) + \beta) +\Sigma^X A_0(t) - \Theta\big)\\ &+ \big(\frac{1-\gamma}{\gamma}(\Sigma^{S,X})^\top (\Sigma^S)^{-1} (\Sigma^{S,X} A_0(t) + \beta) +\Sigma^X A_0(t) - \Theta\big)^\top Q(t) = 0\\
        Q(T) &= - P^\top \Omega^{-1}P.
\end{cases}
\end{equation}
We showed in \eqref{ECeq:nsd_proof} that \[\Sigma^X + \frac{1-\gamma}{\gamma} (\Sigma^{S,X})^\top (\Sigma^S)^{-1} \Sigma^{S,X}\succeq 0.\] It follows then from \cite{Ricatti} that \(Q(t) \preceq 0\) for \(t \in [0,T]\). 
Thus, we conclude that
\[
A_1(t) - A_0(t) \preceq 0 \quad \forall t \in [0,T].
\]

\paragraph{2. Proof that $A_1(t) - A_0(t)$ is Decreasing in View Precision $\Omega^{-1}$.}
The proof of this result follows the same steps as the previous part. Let's consider two different sets of views characterized by their precision matrices, $\Pi_a = \Omega_a^{-1}$ and $\Pi_b = \Omega_b^{-1}$. Assume that the second view is more precise, such that $\Pi_b \succeq \Pi_a$ in the positive semi-definite sense.

This gives rise to two solutions, $A_{1,a}(t)$ and $A_{1,b}(t)$, and their corresponding differences from the no-views case, $Q_a(t) = A_{1,a}(t) - A_0(t)$ and $Q_b(t) = A_{1,b}(t) - A_0(t)$.

We compare their terminal conditions at $t=T$.
Since $\Pi_b \succeq \Pi_a$, it follows that $P^{\top}\Pi_b P \succeq P^{\top}\Pi_a P$. Giving $-P^{\top}\Pi_b P \preceq -P^{\top}\Pi_a P$. Therefore, at the terminal time, we have $Q_b(T) \preceq Q_a(T)$.

Since both $Q_a(t)$ and $Q_b(t)$ satisfy the identical Riccati differential equation \eqref{eq:EC_Q(t)_riccati} and $Q_b(T) \preceq Q_a(T)$, it can be shown by a similar argument as part 1, that
\[
Q_b(t) \preceq Q_a(t) \quad \forall t \in [0,T].
\]
This shows that as view precision $\Omega^{-1}$ increases, the matrix difference $A_1(t) - A_0(t)$ becomes more negative semi-definite.

\subsection{Proof of Lemma \ref{lemma:separation_ODEs}}
Let \begin{equation*}
    V(t,z,x) = \frac{z^{1-\gamma}}{1-\gamma}\exp\left(\frac{1}{2}x^\top A(t) x + x^\top b(t) + c(t)\right),
\end{equation*} be the investor's value function. From the optimal policy characterization in Section \ref{sec51},  \(A(t) \in \mathbb{R}^{N \times N}\) solves
\begin{equation}
\label{eq:EC_Ricatti_Lemma}
    \begin{cases}
        A'(t) &+ \frac{1-\gamma}{\gamma} \tilde{\beta}(t)^\top (\Sigma^S)^{-1} \tilde{\beta}(t) + A(t) \big(\Sigma^X + \frac{1-\gamma}{\gamma} (\Sigma^{S,X})^\top (\Sigma^S)^{-1} \Sigma^{S,X}\big) A(t) \\
        &+ A(t) \big(\frac{1-\gamma}{\gamma}(\Sigma^{S,X})^\top (\Sigma^S)^{-1} \tilde{\beta}(t) - \tilde{\Theta}(t)\big) + \big(\frac{1-\gamma}{\gamma}(\Sigma^{S,X})^\top (\Sigma^S)^{-1} \tilde{\beta}(t) - \tilde{\Theta}(t)\big)^\top A(t) = 0,\\
        A(T) &= 0,
    \end{cases}
\end{equation}
where
\begin{equation*}
    \begin{cases}
        \eta(t) &= (P e^{-\Theta (T-t)} L^X)^\top \big(P (\Sigma - e^{-\Theta(T-t)} \Sigma e^{-\Theta^\top (T-t)})P^\top + \Omega\big)^{-1},\\
        \tilde{\Theta}(t) &= \Theta + L^X \eta(t) P e^{-\Theta (T-t)},\\
        \tilde{\beta}(t) &= \beta -  L^S \eta(t) P e^{-\Theta (T-t)}    .
    \end{cases}
\end{equation*}
Now, let \(A_1(t), \hat{A}(t) \in \mathbb{R}^{N \times N}\) be two symmetric matrices where \(A_1(t)\) satisfies
\begin{equation}
\label{eq:EC_Ricatti_1_Lemma}
\begin{split}
        A_1'(t) &+ \frac{1-\gamma}{\gamma} \beta^\top (\Sigma^S)^{-1} \beta(t) + A_1(t) \big(\Sigma^X + \frac{1-\gamma}{\gamma} (\Sigma^{S,X})^\top (\Sigma^S)^{-1} \Sigma^{S,X}\big) A_1(t) \\
        &+ A_1(t) \big(\frac{1-\gamma}{\gamma}(\Sigma^{S,X})^\top (\Sigma^S)^{-1} \beta - \Theta\big) + \big(\frac{1-\gamma}{\gamma}(\Sigma^{S,X})^\top (\Sigma^S)^{-1} \beta - \Theta\big)^\top A_1(t) = 0,
        \end{split}
\end{equation}
with terminal condition \(A_1(t) = -P^\top \Omega^{-1} P\),
and \[\hat{A}(t) = A(t) - A_1(t).\]
By substituting \(A(t)\) by \(\hat{A}(t) + A_1(t)\) in \eqref{eq:EC_Ricatti_Lemma}, we get
\begin{equation*}
    \begin{cases}
        (\hat{A}'(t) + A_1'(t)) &+ \frac{1-\gamma}{\gamma} \tilde{\beta}(t)^\top (\Sigma^S)^{-1} \tilde{\beta}(t)\\ &+ (\hat{A}(t) + A_1(t)) \big(\Sigma^X + \frac{1-\gamma}{\gamma} (\Sigma^{S,X})^\top (\Sigma^S)^{-1} \Sigma^{S,X}\big) (\hat{A}(t) + A_1(t)) \\
        &+ (\hat{A}(t) + A_1(t)) \big(\frac{1-\gamma}{\gamma}(\Sigma^{S,X})^\top (\Sigma^S)^{-1} \tilde{\beta}(t) - \tilde{\Theta}(t)\big) \\ &+ \big(\frac{1-\gamma}{\gamma}(\Sigma^{S,X})^\top (\Sigma^S)^{-1} \tilde{\beta}(t) - \tilde{\Theta}(t)\big)^\top (\hat{A}(t) + A_1(t)) = 0,\\
        \hat{A}(T) + A_1(T) &= 0.
    \end{cases}
\end{equation*}
By noting that \(A_1(t)\) solves \eqref{eq:EC_Ricatti_1_Lemma}, it can be easily verified that \(\hat{A}(t)\) solves the Riccati
\begin{equation*}
    \hat{A}'(t) = \Theta^\top \hat{A}(t) + \hat{A}(t) \Theta + \hat{A}(t) \Sigma^X \hat{A}(t)
\end{equation*}
with terminal condition \(\hat{A}(T) = P^\top \Omega^{-1} P\). The solution to the above Riccati can be derived explicit and is 
\begin{equation*}
    \hat{A}(t) =(P e^{-\Theta (T-t)})^\top \big( P (\Sigma - e^{-\Theta (T-t)} \Sigma e^{-\Theta^\top (T-t)}) P^\top + \Omega \big)^{-1} P e^{-\Theta (T-t)}.
\end{equation*}

We apply a similar approach for \(b(t)\). From the results  in Section \ref{sec51}, \(b(t)\) solves
\begin{equation}
\label{eq:EC_ODE_Lemma}
    \begin{cases}
        b'(t) &+ \frac{1-\gamma}{\gamma} \big( \tilde{\beta}(t)^\top + A(t) (\Sigma^{S,X})^\top\big) (\Sigma^S)^{-1} \big( \Sigma^{S,X} b(t) + \tilde{\alpha}(t) + \tilde{\gamma}(t) y - r_f \mathbf{1}_N \big) \\ &+ \big(A(t) \Sigma^X - \tilde{\Theta} (t)^\top\big) b(t) + A(t) \tilde{\mu}(t,y) = 0, \\
        b(T) &= 0.
    \end{cases}
\end{equation}
Now let \(b_1(t), \hat{b}(t) \in \mathbb{R}^{N}\) where \(b_1(t)\) satisfies
\begin{equation}
\label{eq:EC_ODE_1_Lemma}
\begin{split}
               b_1'(t) &+ \frac{1-\gamma}{\gamma} \big( \beta^\top + A_1(t) (\Sigma^{S,X})^\top\big) (\Sigma^S)^{-1} \big( \Sigma^{S,X} b_1(t) + \alpha - r_f \mathbf{1}_N \big) \\&+ \big(A_1(t) \Sigma^X - \Theta^\top\big) b_1(t) + A_1(t) \Theta \mu = 0,
        \end{split}
\end{equation}
with terminal condition \(b_1(t) = -P^\top \Omega^{-1} P\),
and \[\hat{b}(t) = b(t) - b_1(t).\]
Following the same steps we did for \(\hat{A}(t)\), by substituting \(b(t)\) by \(\hat{b}(t) + b_1(t)\) in \eqref{eq:EC_Ricatti_Lemma}, we get that \(\hat{b}(t)\) solves the following ODE
\begin{equation}
\begin{split}
        \hat{b}'(t) &+ \frac{1-\gamma}{\gamma} \beta^\top (\Sigma^S)^{-1} \left(\Sigma^{S,X} \hat{b}(t) - \hat{\alpha}_1(t) \mu + \hat{\alpha}_2(t) y \right)\\
        & -\Theta^\top \hat{b}(t) + \hat{A}(t) \left( \Theta \mu - \hat{\mu}_1(t) \mu + \hat{\mu}^2(t) y \right) = 0,       \end{split}
\end{equation}
with terminal condition \(\hat{b}(T) = P^\top \Omega^{-1}P \), and 
\begin{equation*}
    \begin{split}
        \hat{\alpha}_1(t) &= L^S \eta(t) P (1 - e^{-\Theta (T-t)}), \\
        \hat{\alpha}_2(t) &= L^S \eta(t),\\
        \hat{\mu}_1(t) &= L^X \eta(t) P (1 - e^{-\Theta (T-t)}), \\
        \hat{\mu}_2(t) &= L^X \eta(t). \\
    \end{split}
\end{equation*}
It can be directly verified that the solution to the ODE system \eqref{eq:EC_ODE_1_Lemma} is 
\begin{equation*}
        \hat{b}(t) = (P e^{-\Theta (T-t)})^\top \big( P (\Sigma - e^{-\Theta (T-t)} \Sigma e^{-\Theta^\top (T-t)}) P^\top + \Omega \big)^{-1} \big( P (1 - e^{-\Theta (T-t)} ) \mu - y \big).
\end{equation*}
This completes the proof. \hfill $\square$

\newpage
\section{Section \ref{sec6}}
\label{Appsec:KF}

\subsection{Proof of Proposition \ref{prop:KF}}
Since $\alpha$ is assumed to be constant and normally distributed with initial mean $\alpha_0$ and covariance $\Gamma_0$, and the observation process is given by
\[
dN^y(t) = H \alpha \, dt + G \, L^S\,dW^\mathbb{Q}(t),
\]
where 
\begin{equation*}
    H = \begin{pmatrix}
        I_N \\ \mathbf{0}_{d \times N}
    \end{pmatrix}, \quad G = \begin{pmatrix}
        L^S \\ L^X
    \end{pmatrix}.
\end{equation*}
The problem falls into the linear Gaussian filtering framework. In this setting, the Kalman–Bucy filter provides that the best estimate (in the mean square error sense) of $\alpha$ given the observation history $\mathcal{F}_t^Y$ is
\[
\hat{\alpha}(t) = \mathbb{E}\bigl[\alpha \mid \mathcal{F}_t^Y\bigr].
\]

The derivation of the dynamics of \(\hat{\alpha}(t)\) follows Theorem 6.3.1 of \cite{oksendal2003}, and its dynamics are driven by the so-called \emph{innovation process}.. We define the innovation process by
\[
dI(t) = dN^y(t) - H \hat{\alpha}(t) \, dt.
\]
Then, by the Kalman–Bucy theory, the dynamics of $\hat{\alpha}(t)$ are given by
\begin{equation}
\label{ECeq:hat_alpha}
    d\hat{\alpha}(t) = K(t) dI(t),
\end{equation}
where \begin{equation}
    \label{ECeq:K_gain}
K(t) = \Gamma(t) H^\top (GG^\top)^{-1} \in \mathbb{R}^{N \times (N+d)}
\end{equation} 
is the Kalman Gain matrix, with estimation error covariance
\begin{equation*}
    \Gamma(t) = \mathbb{E}[(\alpha - \hat{\alpha}(t))(\alpha - \hat{\alpha}(t))^\top | \mathcal{F}_t^Y] \in \mathbb{R}^{N \times N}.
\end{equation*}
\(\Gamma(t)\) satisfies the Riccati equation 
\[
\frac{d}{dt}\Gamma(t) = -\Gamma(t) \bigl(\Sigma^S - \Sigma^{S,X} (\Sigma^X)^{-1} (\Sigma^{S,X})^\top \bigr)^{-1}\Gamma(t),
\]
with initial condition \(\Gamma(0) = \Gamma_0\). It can be easily verified that the solution is 
\begin{equation}
    \label{ECeq:Gamma_t}
    \Gamma(t) = \left(\Gamma_0^{-1} + t (\Sigma^S - \Sigma^{S,X}(\Sigma^X)^{-1} (\Sigma^{S,X})^{-1})^{-1}\right)^{-1}.
\end{equation}

Thus, the filtering equation for $\hat{\alpha}(t)$ becomes
\[
d\hat{\alpha}(t) = K(t) \Bigl[dN^y(t)-H\hat{\alpha}(t)\,dt\Bigr],
\]
where \(K(t)\) is given by \eqref{ECeq:K_gain} -- \eqref{ECeq:Gamma_t},
with the initial condition $\hat{\alpha}(0) = \alpha_0$. This completes the result. \hfill $\square$

\subsection{Proof of The HJB and Value Function Characterization in Section \ref{subsec:HJB_KF}}

Let \(M(t) = [X^y(t)^\top, \hat{\alpha}(t)^\top]^\top\) be the augmented state vector. The investor's value function is defined as
\[
V(t,z,m) = \sup_{\pi\in\Pi} \mathbb{E}\Biggl[\frac{Z(T)^{1-\gamma}}{1-\gamma}\,\Big|\, Z(t)=z,\; M(t)=m\Biggr],
\]
with the state variables evolving according to
\begin{align}
\label{ECeq:StateVariables}    dZ(t) &= Z(t)\Biggl( r_f\,dt + \pi(t)^\top\Bigl\{\alpha + \tilde{\lambda}(t,y) + \tilde{\beta}(t)x - r_f\,\mathbf{1}_N\Bigr\}dt + \pi(t)^\top L^S\,dW^\mathbb{Q}(t) \Biggr), \\[1mm]
    dM(t) &= \Theta^M(t) (\mu^M(t) - M(t)) dt + \begin{pmatrix}
        L^X & G
    \end{pmatrix}^\top dW^{M,\mathbb{Q}}(t)
\end{align}
where 
\begin{equation*}
    W^{M,\mathbb{Q}}(t) = [W^\mathbb{Q}(t)^\top, V^\mathbb{Q}(t)^\top]^\top
\end{equation*}
is a Brownian motion, and \(dM(t) (dM(t))^\top =  L^M_t ( L^M_t )^\top \) with
\begin{equation*}
    L^M_t = \begin{pmatrix} L^X \\ K(t) G \end{pmatrix} \in \mathbb{R}^{(d + N ) \times N'},
\end{equation*}
and the mean-reverting and long-run mean coefficients are
\[\tilde{\Theta}^M(t) = \begin{pmatrix} \tilde{\Theta}(t) & \mathbf{0} \\ \mathbf{0} & \mathbf{0} \end{pmatrix}, \quad \tilde{\mu}^M(t, y) = \begin{pmatrix} \tilde{\mu}(t,y) \\ \mathbf{0} \end{pmatrix}.\]

The Hamilton-Jacobi-Bellman (HJB) equation for the value function $V(t,z,m)$ is $\sup_{\pi} \mathcal{L}^{\pi}V = 0$, where $\mathcal{L}^{\pi}$ is the infinitesimal generator of the process $(Z(t), M(t))$. Applying Itô's lemma to $V(t,z,m)$ yields:
\begin{align*}
\mathcal{L}^{\pi}V &= \frac{\partial V}{\partial t} + \frac{\partial V}{\partial z}\mathbb{E}[dZ(t)] + (\nabla_m V)^\top \mathbb{E}[dM(t)] + \frac{1}{2}\frac{\partial^2 V}{\partial z^2}\mathbb{E}[(dZ(t))^2] \\
&\quad+ \frac{1}{2}\operatorname{Tr}\Bigl(\nabla_m^2 V \cdot \mathbb{E}[dM(t)dM(t)^\top]\Bigr) + (\nabla_{z,m}^2 V)^\top\mathbb{E}[dM(t)dZ(t)].
\end{align*}
To characterize the generator, we identify the drift and covariance terms from the state dynamics. From \eqref{ECeq:StateVariables}
    \begin{align*}
    \mathbb{E}[dZ(t)] &= z\Bigl(r_f + \pi^\top(\tilde{\lambda}(t,y) + \tilde{\beta}^{M}(t)m - r_f\mathbf{1}_N)\Bigr)dt, \\
    \mathbb{E}[dM(t)] &= \tilde{\Theta}^{M}(t)\bigl(\tilde{\mu}^{M}(t,y) - m\bigr)dt.
    \end{align*}
    where \(\tilde{\beta}^M(t) = \begin{pmatrix}
        \tilde{\beta}(t) & I_N
    \end{pmatrix}.\)
 The quadratic and cross-variation terms are determined by the diffusion components of \eqref{ECeq:StateVariables}. Let $\Sigma^S = L^S(L^S)^\top$. We use the covariance matrices for the augmented system as defined in the paper: $\Sigma_t^M := L_t^M(L_t^M)^\top$ and the cross-covariance $\Sigma_t^{S,M} := L^S(L_t^M)^\top$.
    \begin{align*}
    \mathbb{E}[(dZ(t))^2] &= z^2\,\pi^\top \Sigma^S\,\pi\,dt, \\
    \mathbb{E}[dM(t)dM(t)^\top] &= \Sigma_t^M\,dt, \\
    \mathbb{E}[dM(t)dZ(t)] &= L_t^M (z\,\pi^\top L^S)^\top dt = z\,(\Sigma_t^{S,M})^\top\,\pi\,dt.
    \end{align*}
Substituting these terms into the generator yields the HJB equation
\begin{align*}
0 = \max_{\boldsymbol{\pi}}\Biggl\{ & \frac{\partial V}{\partial t} + z\Bigl[r_f + \boldsymbol{\pi}^\top\bigl(\tilde{\boldsymbol{\lambda}}(t,y) + \tilde{\boldsymbol{\beta}}^{M}(t)m - r_f\mathbf{1}_N\bigr)\Bigr]\frac{\partial V}{\partial z} + \Bigl(\Theta^{M}(t)\bigl(\boldsymbol{\mu}^{M}(t,y)-m\bigr)\Bigr)^\top\nabla_m V \\
&\quad+ \frac{1}{2}z^2\boldsymbol{\pi}^\top\Sigma^S\boldsymbol{\pi}\frac{\partial^2 V}{\partial z^2} + \frac{1}{2}\operatorname{Tr}\bigl(\Sigma_t^M \nabla_m^2 V\bigr) + z\boldsymbol{\pi}^\top\Sigma_t^{S,M}\nabla_{m,z}^{2}V \Biggr\}.
\end{align*}
This completes the derivation of the HJB equation. \hfill $\square$

\subsection{Proof of Proposition \ref{prop:Policy_KF}}
\label{App:Proof_Policy_KF}
For a CRRA utility \(U(z) = z^{1-\gamma}/(1-\gamma)\), the value function is of the form
\begin{equation}
\label{eq:EC_V_filtering}
    V(t,z,m) = U(z) \exp\left(g(t,m)\right),
\end{equation}
where \(g(t,m)\) is quadratic in \(m = (x^\top, \hat{\alpha}^\top)^\top\). In particular, we let 
\begin{equation}
\label{eq:EC_g_filtering}
\begin{split}
    g(t,m) &= \frac{1}{2}m^\top A(t) m + m^\top b(t) + c(t)\\
   &=  \frac{1}{2}x^\top A^x(t)x + x^\top b^x(t) + \frac{1}{2}\hat{\alpha}^\top A^{\alpha}(t)\hat{\alpha} + \hat{\alpha}^\top b^\alpha(t) + x^\top A^{x,\alpha}(t)\hat{\alpha} + c(t),
    \end{split}
\end{equation}
with \begin{equation*}
    A(t) = \begin{pmatrix}
        A^x(t) & A^{x,\alpha}(t) \\
        A^{x,\alpha}(t)^\top & A^\alpha(t)
    \end{pmatrix} \in \mathbb{R}^{(N+d) \times (N+d)},
\end{equation*}
and \begin{equation*}
    b(t) = \begin{pmatrix}
        b^x(t) \\ b^\alpha(t)
    \end{pmatrix} \in \mathbb{R}^{N+d}.
\end{equation*}

This problem is structurally identical to the one solved in Section \ref{sec51}, but for a higher-dimensional state. We can therefore find the solution by defining an augmented system and applying the results in Section \ref{sec51} directly.

From Section \ref{sec61}, the dynamics of the augmented state, $M(t) = [X^y(t)^\top, \hat{\alpha}(t)^\top]^\top$, can be written in a compact mean-reverting form
\[
    dM(t) = \tilde{\Theta}^M(t)\left(M(t) - \tilde{\mu}^M(t,y) \right) dt + dW^{M,\mathbb{Q}}(t).
\]
The corresponding augmented system matrices for the wealth and state processes are:
\begin{gather*}
    \tilde{\beta}^M(t) = \begin{pmatrix} \tilde{\beta}(t) & I_N \end{pmatrix}, \quad
    \tilde{\Theta}^M(t) = \begin{pmatrix} \tilde{\Theta}(t) & \mathbf{0} \\ \mathbf{0} & \mathbf{0} \end{pmatrix}, \quad
    \tilde{\mu}^M(t,y) = \begin{pmatrix} \tilde{\mu}(t,y) \\ \mathbf{0} \end{pmatrix}, \\
    \Sigma_t^M := L_t^M (L_t^M)^\top, \quad \text{and} \quad \Sigma_t^{S,M} := L^S (L_t^M)^\top.
\end{gather*}
where
\begin{equation*}
    L^M_t = \begin{pmatrix}
        L^X \\ K(t) G
    \end{pmatrix}.
\end{equation*}

By direct analogy to the HJB characterization in Section \ref{sec51}, the coefficient matrices $A(t)$, $b(t)$, and scalar $c(t)$ for the quadratic function $g(t,m) = \frac{1}{2}m^\top A(t) m + m^\top b(t) + c(t)$ must satisfy the same system of ODEs, but with the original system matrices replaced by their augmented counterparts.

Specifically, $A(t) \in \mathbb{R}^{(d+N)\times(d+N)}$ solves the matrix Riccati equation:
\begin{equation}
\label{eq:EC_Riccati_A_filtering}
\begin{cases}
    A'(t) + \frac{1-\gamma}{\gamma} \tilde{\beta}^M(t)^\top (\Sigma^S)^{-1} \tilde{\beta}^M(t) + A(t) \left(\Sigma_t^M + \frac{1-\gamma}{\gamma} (\Sigma_t^{S,M})^\top (\Sigma^S)^{-1} \Sigma_t^{S,M}\right) A(t) \\
    \quad + A(t) \left(\frac{1-\gamma}{\gamma}(\Sigma_t^{S,M})^\top (\Sigma^S)^{-1} \tilde{\beta}^M(t) - \tilde{\Theta}^M(t)\right) \\
    \quad + \left(\frac{1-\gamma}{\gamma}(\Sigma_t^{S,M})^\top (\Sigma^S)^{-1} \tilde{\beta}^M(t) - \tilde{\Theta}^M(t)\right)^\top A(t) = \mathbf{0},\\
    A(T) = 0.
\end{cases}
\end{equation}
The vector $b(t) \in \mathbb{R}^{d+N}$ solves the linear ODE system:
\begin{equation}
\label{eq:EC_ODEsys_b_filtering}
\begin{cases}
    b'(t) + \frac{1-\gamma}{\gamma} \left( \tilde{\beta}^M(t)^\top + A(t) (\Sigma_t^{S,M})^\top\right) (\Sigma^S)^{-1} \left( \Sigma_t^{S,M} b(t) + \tilde{\lambda}(t,y) - r_f \mathbf{1}_N \right) \\
    \quad + \left(A(t) \Sigma_t^M - \tilde{\Theta}^M(t)^\top\right) b(t) + A(t) \tilde{\mu}^M(t,y) = \mathbf{0}, \\
    b(T) = 0.
\end{cases}
\end{equation}
And the scalar $c(t) \in \mathbb{R}$ is found by integrating:
\begin{equation}
\label{eq:EC_c_filtering}
\begin{cases}
    c'(t) +(1-\gamma)r_f + \frac{1}{2}\operatorname{Tr}(\Sigma_t^M A(t)) + \tilde{\mu}^M(t,y)^\top b(t) + \frac{1}{2}b(t)^\top \Sigma_t^M b(t)\\
    \quad + \frac{1-\gamma}{2\gamma}\left(\Sigma_t^{S,M} b(t) + \tilde{\lambda}(t,y) - r_f \mathbf{1}_N \right)^\top (\Sigma^S)^{-1}\left(\Sigma_t^{S,M} b(t) + \tilde{\lambda}(t,y) - r_f \mathbf{1}_N \right) = 0,\\
    c(T) = 0.
\end{cases}
\end{equation}
As established in Section \ref{sec51}, $A(t)$ is negative semi-definite for all $t\leq T$, which in turn implies its diagonal blocks, $A^x(t)$ and $A^\alpha(t)$, are also negative semi-definite. This completes the proof. \hfill $\square$

\newpage
\section{Derivation of the Black-Litterman Variant}
\label{App:BL}
In this section, we derive the conditional dynamics of asset returns and the corresponding allocation policy, following the classical Black-Litterman framework introduced in \cite{Black_1991,Black_1992}. While their original model incorporates views directly on asset returns, we extend the approach to include views on risk factors, maintaining the same principles of the original methodology.

We begin by deriving the conditional distribution of risk factors and asset returns through a one-step Bayesian update. Subsequently, we determine the optimal allocation policy using the standard mean-variance Markowitz optimal framework. Throughout, we fix the realization of the expert views \(Y(0,T)=y\) and work under the
conditional probability measure \(\mathbb Q=\mathbb P(\,\cdot\,\vert Y(0,T)=y)\).

\subsection{Conditional Mean and Variance of the Risk factors}
From Proposition \ref{prop:X^y} the conditional risk factor \(X\) is Gaussian and satisfies
\begin{equation*}
        dX^y(t) = \left(\tilde{\Theta}(t)\left(\mu - X^y\right) + L^X \eta(t) \left(y - P \mu\right)\right) dt + L^X dW^\mathbb{Q}(t),
    \end{equation*}
    where the parameters \(\tilde{\Theta}(t)\) and \(\eta(t)\) are defined in \eqref{eq:eta}.
    Fix a conditioning time \(s\in[0,t)\) and denote
\(\mathcal F_{s}^{Y}:=\sigma\bigl\{X(u), \, S(u), \, Y(0,T):0\le u \le s\bigr\}\). 

\paragraph{Solution of the SDE.}
Let
\(
\phi(u,t):=\exp\bigl(-\int_{u}^{t}\tilde{\Theta}(v)\,\mathrm dv\bigr)
\)
( \(0\le u\le t\le T\) ).  
A standard variation of constants gives
\begin{equation}
\label{eq:Xy-soln}
  X^{y}(t)
    =\phi(s,t)X^{y}(s)
     +\int_{s}^{t}\!\phi(u,t)
        \bigl[L^{X}\eta(u)(y-P\mu)+\tilde{\Theta}(u)\mu\bigr]\mathrm du
     +\int_{s}^{t}\!\phi(u,t)L^{X}\,\mathrm dW^{\mathbb Q}(u).
\end{equation}

\paragraph{Conditional mean.}
Taking the expectation over \eqref{eq:Xy-soln} gives the expectation of risk factors conditional on expert views \(Y(0,T) = y\) and observed market data up until time \(s \in [0,t)\).
\begin{equation}
\label{eq:BL_meanX}
  \mathbb E\bigl[X(t)\,\bigm|\,\mathcal F_{s}^{Y}\bigr]
      =\phi(s,t)X^{y}(s)
       +\int_{s}^{t}\!\phi(u,t)
          \bigl[L^{X}\eta(u)(y-P\mu)+\tilde{\Theta}(u)\mu\bigr]\mathrm du
\end{equation}

\paragraph{Conditional variance.}
Because the last term in~\eqref{eq:Xy-soln} is a martingale in \(\mathbb Q\), by Itô's isometry, we obtain
\begin{equation}
\label{eq:BL_varX}
 \mathbb V\bigl[X(t)\,\bigm|\,\mathcal F_{s}^{Y}\bigr]
      =\int_{s}^{t}\!\phi(u,t)\Sigma^X\phi(u,t)^{\!\top}\,\mathrm du
\end{equation}

\subsection{Conditional Mean and Variance of Log-Returns}

Similarly, from Corollary \ref{corr:S^y}, the conditional price dynamics are
\begin{equation*}
\mathrm dS^{y}(t)
    =D\bigl(S^{y}(t)\bigr)
      \Bigl[\mu^{S}(t,X^{y}(t),y)\,\mathrm dt
            +L^{S}\,\mathrm dW^{\mathbb Q}(t)\Bigr],
\qquad 
\mu^{S}(t,x,y)=\tilde\alpha(t,y)+\tilde\beta(t)x
\end{equation*}
where the coefficients \(\tilde{\alpha}(t,y)\) and \(\tilde{\beta}(t)\) are defined in \eqref{eq:coeff_mu^S}. Let the log–return over the horizon \([0,t]\) be
\[
R^{y}(t):=\ln\frac{S^{y}(t)}{S^{y}(0)},
\]
by Itô's Lemma, we get for \(0 \leq s \leq t \leq T\)
\begin{equation}
\label{eq:R^y-app}
        R^y(t) = R^y(s) +  \int_s^t \left(\tilde{\alpha}(u,y) - \frac{1}{2}\diag(\Sigma^S)  + \tilde{\beta}(u) X^y(u) \right) du + L^S \int_s^t dW^\mathbb{Q}(u).
\end{equation}

\paragraph{Conditional mean of the terminal log-return.}
Conditional on the information \(\mathcal{F}_s^Y\), we get from \eqref{eq:Xy-soln} and \eqref{eq:R^y-app}
\begin{align}
\label{eq:BL_meanR}
  \mu_{BL | s} := \mathbb E\bigl[R^{y}(T)\,\bigm|\,\mathcal F_{s}^{Y}\bigr]
     = R^{y}(s)
       +\int_{s}^{T}\!\!\Bigl[
          \tilde{\alpha}(u,y)
          -\tfrac12\mathrm{diag}(\Sigma_{S})
          +\tilde{\beta}(u)\,\phi(s,u)X^{y}(s)
          \Bigr]\mathrm du \\
  \nonumber\quad
  +\int_{s}^{T}\!\!\tilde{\beta}(u)
      \!\int_{s}^{u}\phi(r,u)
      \bigl[L^{X}\eta(r)(y-P\mu)+\tilde{\Theta}(r)\mu\bigr]\mathrm dr\,\mathrm du .
\end{align}

\paragraph{Conditional variance of the terminal log-return.}
Write \(R^{y}(T)=R^{y}(s)+J+M\) with
\[
  J:=\int_{s}^{T}
       \bigl[\tilde{\alpha}(u,y)-\tfrac12\mathrm{diag}(\Sigma_{S})
           +\tilde{\beta}(u)X^{y}(u)\bigr]\mathrm du,
  \qquad
  M:=\int_{s}^{T}L^{S}\mathrm dW^{\mathbb Q}(u).
\]

Conditional \(F_s^Y\), the log-returns \(R^y(s)\) at time \(s\) are known. The conditional covariance is then
\begin{equation}
\label{eq:BL_varR}
\begin{split}
   \Sigma_{BL|s} :&= \mathbb V\bigl[R^{y}(T)\,\bigm|\,\mathcal F_{s}^{Y}\bigr] \\ &=  \mathbb{V}[M] + \mathbb{V}[J] + \Cov(J,M) + \Cov(J,M)^\top \\
     &=(T-s)\,\Sigma_{S}
      +\iint_{[s,T]^{2}}\!
          \tilde{\beta}(u)\,C_{s}(u,v)\,\tilde{\beta}(v)^{\!\top}
          \mathrm du\,\mathrm dv
      +\int_{s}^{T} \int_s^u \tilde{\beta}(u) \phi(r,u) L^{X}L^{S\top}\mathrm dr \, du\\ 
      &\;\;\;+ \int_{s}^{T} \int_s^u \left(\tilde{\beta}(u) \phi(r,u) L^{X}L^{S\top}\right)^\top\mathrm dr \, du,
\end{split}
\end{equation}
where the (conditional) factor covariance kernel is
\[
  C_{s}(u,v)
   :=\Cov(X^{y}(u),X^{y}(v)
                    \,\bigm|\,\mathcal F_{s}^{Y})
   =\!\int_{s}^{\min\{u,v\}}\!\!\phi(r,u)\Sigma^X
      \phi(r,v)^{\!\top}\mathrm dr .
\]

\begin{remark}
    If the price shocks are orthogonal to factor shocks,
         $L^{S}L^{X\top}=0$, then the risk factor and the noise in the asset price are independent and the cross term is equal to 0. 
\end{remark}

\subsubsection*{Optimal Portfolio Policy}

An investor with risk-aversion \(\gamma>0\) chooses weights
\(
\pi\in\mathbb R^{N}\) at time \(t\) to solve
\[
  \max_{\pi}\;
     \pi^{\top}\bigl(\mu_{BL|t}-r_{f}\mathbf 1_N\bigr)
     -\frac{\gamma}{2}\,\pi^{\top}\Sigma_{BL | t}\,\pi,
\]
where \(r_{f}\) is the risk-free rate and
\(\mu_{BL|s},\Sigma_{BL| s}\) are given by
\eqref{eq:BL_meanR}--\eqref{eq:BL_varR}, and \(1_N \in \mathbb{R}^N\) is a vector of ones,

The first-order condition yields the familiar Black–Litterman-style
optimal weights
\begin{equation}
\label{eq:optBL}
  \pi^{*}_{BL| s}
     =\frac{1}{\gamma}\,\Sigma_{BL| s}^{-1}
      \bigl(\mu_{BL|s}-r_{f}\mathbf 1_{N}\bigr).
\end{equation}
\hfill\(\square\)

\end{appendices}

\end{document}